\documentclass[11pt]{amsart}
\usepackage[margin=2.6cm]{geometry}

\usepackage{microtype} 
\usepackage{multicol}
\usepackage{soul}
\usepackage{xcolor}
\setstcolor{orange}

\usepackage{amsthm,amsmath,amssymb,amsfonts}
\usepackage[alphabetic]{amsrefs}
\usepackage[skip=4pt plus1pt, indent=10pt]{parskip} 
\usepackage{enumerate}
\usepackage{tikz}
\usepackage[colorlinks,bookmarks, anchorcolor=blue, citecolor=blue, linkcolor=blue,urlcolor=blue, bookmarksopenlevel=1]{hyperref}
\usepackage{tikz-cd}

\usepackage{tabularx}
\usepackage{booktabs}
\newcolumntype{M}{>{$}l<{$}}

\newtheorem{theorem}{Theorem}[section]
\newtheorem{proposition}[theorem]{Proposition}
\newtheorem{lemma}[theorem]{Lemma}
\newtheorem{corollary}[theorem]{Corollary}

\theoremstyle{definition}

\newtheorem{remark}[theorem]{Remark}
\newtheorem{example}[theorem]{Example}

\definecolor{darkgreen}{RGB}{0,90,0}

\DeclareMathOperator{\Id}{Id}
\DeclareMathOperator{\divis}{div}

\newcommand{\YMD}[1]{}

\begin{document}\baselineskip=15pt

\begin{center}
\title[Kodaira dimension of moduli of HK varieties]{Kodaira dimension of moduli spaces of hyperk\"{a}hler varieties}

\author{Ignacio Barros}
\address{\parbox{0.9\textwidth}{
Department of Mathematics\\[1pt]
Universiteit Antwerpen\\[1pt]
Middelheimlaan 1, 2020 Antwerpen, Belgium
\vspace{1mm}}}
\email{{ignacio.barros@uantwerpen.be}}

\author{Pietro Beri}
\address{\parbox{0.9\textwidth}{
Institut Elie Cartan\\[1pt]
Université de Lorraine, Site de Nancy\\[1pt]
B.P. 70239, F-54506 Vandoeuvre-les-Nancy Cedex, France
\vspace{1mm}}}
\email{{pietro.beri@univ-lorraine.fr}}

\author{Emma Brakkee}
\address{\parbox{0.9\textwidth}{
Mathematical Institute\\[1pt]
Universiteit Leiden\\[1pt]
Niels Bohrweg 1, 2333 CA Leiden, The Netherlands
\vspace{1mm}}}
\email{{e.l.brakkee@math.leidenuniv.nl}}

\author{Laure Flapan}
\address{\parbox{0.9\textwidth}{
Department of Mathematics\\[1pt]
Michigan State University\\[1pt]
619 Red Cedar Road, East Lansing, MI 48824, USA
\vspace{1mm}}}
\email{{flapanla@msu.edu}}

\thanks{I.B. was supported by the ERC Synergy Grant ERC-2020-SyG-854361-HyperK, the Deutsche Forschungsgemeinschaft
(DFG, German Research Foundation) -- SFB-TRR 358/1 2023 -- 491392403, and Fonds voor Wetenschappelijk Onderzoek – Vlaanderen (FWO, Research Foundation – Flanders) -- G0D9323N. P. B. was supported by the ERC Synergy Grant ERC-2020-SyG-854361-HyperK. E.B.~was supported by NWO grants 016.Vidi.189.015 and VI.Veni.212.209. 
L.F.~was supported by the NSF grant DMS-2200800.}

\maketitle
\end{center}

\begin{abstract}
We study the Kodaira dimension of moduli spaces of polarized hyperk\"{a}hler varieties deformation equivalent to the Hilbert scheme of points on a K3 surface or to O'Grady's ten-dimensional variety. This question was studied by Gritsenko--Hulek--Sankaran in the cases of K$3^{[2]}$ and OG10 type when the divisibility of the polarization is one. We generalize their results to higher dimension and divisibility. As a main result, for almost all dimensions $2n$ we provide a lower bound on the degree such that for all higher degrees, every component of the moduli space of polarized hyperk\"{a}hler varieties of K$3^{[n]}$ type is of general type.
\end{abstract}

\setcounter{tocdepth}{1} 
\tableofcontents

\section{Introduction}

A hyperk\"ahler manifold is a simply connected compact K\"ahler manifold with a unique non-degenerate $2$-form (up to scaling). 
Hyperk\"ahler manifolds are always even-dimensional with trivial canonical bundle. In this paper we study moduli spaces of hyperk\"{a}hler varieties, by which we mean hyperk\"ahler manifolds which are projective.

Two-dimensional hyperk\"ahler manifolds are always K3 surfaces. Mukai gave unirational parametrizations of the moduli spaces $\mathcal{F}_{2d}$ of polarized K3 surfaces of degree $2d$ for $d\leq 12$ and $d=15,17,19$ \cites{Muk88,Muk92,Muk06,Muk10,Muk16}. 
Recently the same was done for $d=13,21$, see \cites{Nue17, FV18, FV21}. In contrast, Gritsenko--Hulek--Sankaran showed in \cite{GHS07} that $\mathcal{F}_{2d}$ is of general type for $d>61$ and $d=46,50,54,58,60$.

For a hyperk\"ahler manifold $X$, the group 
$H^2(X,\mathbb{Z})$ carries a quadratic form $q_X$, turning it into a lattice called the {\textit{Beauville--Bogomolov--Fujiki lattice}. 
In this paper, we focus on moduli spaces of 
hyperk\"ahler manifolds coming from K3 surfaces: deformations of Hilbert schemes of $n$ points on a K3 surface, called K$3^{[n]}$ type, and deformations of O'Grady's 10-dimensional examples, called OG10 type. 
The corresponding moduli spaces $\mathcal{M}_{\mathrm{K}3^{[n]}, 2d}^\gamma$ (resp.\ $\mathcal{M}_{\mathrm{OG10}, 2d}^\gamma$) parametrize pairs $(X,H)$, with $X$ a projective hyperk\"ahler variety of K$3^{[n]}$ (resp.\ OG10) type and $H$ a primitive polarization on $X$ so that $c_1(H)$ has degree $2d$ with respect to $q_X$ and divisibility $\gamma$ in $H^2(X,\mathbb{Z})$.}

The study of
the moduli spaces $\mathcal{M}_{\mathrm{K}3^{[n]}, 2d}^\gamma$ and $\mathcal{M}_{\mathrm{OG10}, 2d}^\gamma$ has attracted much attention in the past decade. A fundamental question in hyperk\"{a}hler geometry is whether one can effectively construct a general projective hyperk\"{a}hler variety of a given deformation type. 
This is a question about the existence of unirational locally complete families. There are few known examples of such families; 
to our knowledge, there are no constructions outside K$3^{[n]}$ type. 
For instance, such unirational parametrizations are known for $\mathcal{M}_{\mathrm{K}3^{[n]},2d}^\gamma$ in the cases $(n,d,\gamma)=(2,3,2)$ \cite{BD85}, $(2,1,1)$ \cite{OG06},  $(2,19,2)$ \cites{IR01, IR07} (see also \cite{Mon13}*{Proposition 1.4.1}), $(2,11,2)$ \cite{DV10}, $(4,1,2)$ \cite{LLSvS17}, $(3,2,2)$ \cite{IKKR19}, 
and additional higher dimensional examples \cites{BLM+21, PPZ19} (see \eqref{BLM+} below).

The space $\mathcal{M}_{\mathrm{K}3^{[n]}, 2d}^\gamma$ is not irreducible in general, but it is when $\gamma=1,2$, see \cite{Apo14}. The space $\mathcal{M}_{\mathrm{OG10}, 2d}^\gamma$ is always irreducible, \cites{Ono22, Ono22b}. Gritsenko--Hulek--Sankaran treat the ``split'' case $\gamma=1$ in \cites{GHS10, GHS11} (see also \cite{Ma21}) and prove:
\begin{itemize}
\item[(i)] The moduli space $\mathcal{M}_{\mathrm{K}3^{[2]}, 2d}^1$ is of general type if $d\ge 12$.
\item[(ii)] The moduli space $\mathcal{M}_{\mathrm{OG10}, 2d}^1$ is of general type for all $d\ge 3$, $d\ne 2^m$ for $m\ge 0$.
\end{itemize}

Recently, similar results have been proven for moduli spaces of polarized hyperkähler manifolds deformation equivalent to a generalized Kummer fourfold, see \cites{Daw24_1,Daw24_2}.

\subsection{Main Results}

We produce analogous general type results for the higher divisibility and dimension cases of these moduli spaces. In particular,  we give, for almost all dimensions, the first uniform bound (quadratic in $\gamma$ and $n$) on the degree after which these moduli spaces are of general type (see Theorem \ref{thm1} below). Moreover, in the more technical split case $\gamma=1$, in which new subtleties arise, we give a list $(n,d)\in \mathbb{Z}_{>0}\times\mathbb{Z}_{>0}$ of positive density for which the moduli space $\mathcal{M}_{K3^{[n]}, 2d}^1$ is of general type (see Theorem \ref{thm: thm5} below).

In the case $n=2$ the remaining case (non-split) is $\gamma=2$. In this case $\mathcal{M}_{\mathrm{K}3^{[2]}, 2d}^2$ is non-empty if and only if $d\equiv -1\mod 4$. Similarly, for OG10 type, the remaining case (non-split) is $\gamma=3$ and $\mathcal{M}_{\mathrm{OG10}, 2d}^3$ is non-empty if and only if $d\equiv -3\mod 9$. We show:

\begin{theorem}
\label{thm:sec_intro:K32}
Let $\mathcal{M}_{\mathrm{K}3^{[2]},2d}^{2}$ be the moduli space of primitively polarized hyperk\"{a}hler varieties of K$3^{[2]}$-type with non-split polarization of degree $2d=8t-2$. Then for all $t\geq 12$ and $t=10$ the moduli space $\mathcal{M}_{\mathrm{K}3^{[2]},2d}^{2}$ is of general type.
\end{theorem}

\begin{theorem}
\label{thm:sec_intro:OG10}
Let $\mathcal{M}_{\mathrm{OG10},2d}^{3}$ be the moduli space of primitively polarized hyperk\"{a}hler varieties of OG10-type with non-split polarization of degree $2d=18t-6$. Then for all $t\geq 4$ the moduli space $\mathcal{M}_{\mathrm{OG10},2d}^{3}$ is of general type.  
\end{theorem}
We remark that 
the cases $t=10,12$ in Theorem \ref{thm:sec_intro:K32} and $t=4$ in Theorem \ref{thm:sec_intro:OG10} were proved in \cite{GHS13}*{Proposition 9.2}, \cite{GHS11}* {Corollary 4.3}. 

The main results of our paper focus on the moduli spaces $\mathcal{M}_{\mathrm{K}3^{[n]},2d}^{\gamma}$ in the case $n\ge 3$. 

\begin{theorem}
\label{thm1}
(see Theorem \ref{thm:sec_K3n:div3_2})
Let $(n,d,\gamma)$ be a triple such that the moduli space $\mathcal{M}_{\mathrm{K}3^{[n]},2d}^\gamma$ is non-empty (see Proposition \ref{prop:sec_K3n:nonemp}). We assume further that $n\ge 6$, $n\ne 11,13$, and $\gamma\geq 3$. 
Then every component of $\mathcal{M}_{\mathrm{K}3^{[n]},2d}^\gamma$ is of general type provided 
\[d\geq 6\gamma^2\left(n+3+\sqrt{2(n-1)}\right)^2,\]
except for one possible value of $d\ge 5\cdot 10^{10}$ for each $n$ which is odd or in the set $\{10,12,52,\frac{\star}{2}+1\}$.
\end{theorem}
In the above theorem, $\star$ is a fixed integer greater than $5\cdot 10^{10}$ (see Section \ref{ssec:H-K}). This number $\star$ does not exist if the Generalized Riemann Hypothesis holds.

\begin{remark} In Theorem \ref{thm: diophantine bounds version} we give a much sharper lower bound on $d$ after which a given connected component of $\mathcal{M}_{\mathrm{K}3^{[n]},2d}^\gamma$ is of general type. 
However, computing this bound requires solving a linear Diophantine equation whose coefficients depend on the connected component, the value $\gamma$, and integers $\alpha_1, \alpha_2, \alpha_3$ such that $\alpha_1^2+\alpha_2^2+\alpha_3^3=2(n-1)$. 
By contrast, Theorem \ref{thm1} gives a uniform bound on $d$ that does not require specifics about the numerics of $n$ and $\gamma$.
\end{remark}

\begin{example}
As an illustration of the difference between the bounds obtained from Theorem \ref{thm: diophantine bounds version} versus Theorem \ref{thm1}, in the case of the moduli space $\mathcal{M}_{K3^{[26]},2d}^5$, Theorem \ref{thm: diophantine bounds version} yields that both of its components 
are of general type provided $d\ge 225$ (see Example \ref{ex: components different}). 
By contrast, for the same moduli space, Theorem \ref{thm1} yields the bound $d\ge 195169$.
\end{example}

We also consider $\mathcal{M}_{\mathrm{K}3^{[n]},2d}^\gamma$ for $n\ge 3$ when $\gamma=1,2$. These low divisibility cases come with several additional technical challenges. The hardest and perhaps most interesting case we tackle is the ``split case'' $\gamma=1$. 
This case departs from previous literature: modularity of the quasi-pullback is not guaranteed and vanishing of the resulting cusp form at the ramification of the modular projection cannot be proven via classification of root systems of small rank and discriminant. We obtain general type results for a set of pairs $(n,d)$ of density roughly $1/2$:

\begin{theorem}
\label{thm: thm5}
Suppose $n\ge 3$ and write $n-1=4^c \cdot k$ for $c\ge 0$, $4\nmid k$. 
\begin{enumerate}
\item \label{eq: thm5b} If $k$ is a square, then $\mathcal{M}_{\mathrm{K}3^{[n]},2d}^1$ is of general type for all $d\ge 12$.
\item \label{eq: thm5a} If $k\equiv 1,2 \bmod 4$ and $k\not\in \{1,2,5,6,10,13,25,37,58, 85, 130, \star\}$, then 
$\mathcal{M}_{\mathrm{K}3^{[n]},2d}^1$ is of general type for 
all $d$ of the following form:
\begin{enumerate}[(a)]
\item $d=4^e \cdot m$ with $e\ge 0$, $m\geq 3$, $m\not\equiv 0,4,7\bmod 8$ and $m\notin\{5,10,13,25,37,58,85,130,\star\}$;
\item $d=p\cdot r^2$ where $p$ is a prime congruent to $3$ modulo $4$, or, for any square $d=r^2$ when additionally $k\neq 9$.
\end{enumerate}
\end{enumerate}
\end{theorem}

As an example, the lowest dimension for which one can rule out the existence of unirational locally complete families in degree $2$ is $2n=30$ and in degree $6$ is $2n=20$.

\begin{corollary}
There is no unirational locally complete family of polarized hyperk\"{a}hler $20$-folds of K$3$-type with split polarization of degree $6$ or of $30$-folds with split polarization of degree $2$. 
\end{corollary}

The second case already follows from \cite{GHS10} together with \textit{strange duality}, see Proposition \ref{prop:sec_K3n:SD}. In the case of $\gamma=2$, as a consequence of Theorem \ref{thm:sec_intro:K32} we obtain the following general type results for an infinite quadratic series of dimensions.

\begin{theorem}
\label{thm:gamma2}
Suppose $n$ is an even integer such that $n-1$ is square. Then  the moduli space $\mathcal{M}_{\mathrm{K}3^{[n]},2d}^2$ of primitively polarized hyperk\"ahler varieties of degree $2d=8t-2$ and divisibility $2$ is of general type for $t\ge 12$ and $t=10$.  
\end{theorem}

\subsection{Relation to existing literature}

To prove our main results, we reduce the question to the existence of a certain cusp form for an orthogonal modular variety.
This strategy was introduced by Gritsenko--Hulek--Sankaran in \cite{GHS07} and has been subsequently used in, among others, \cites{GHS10, GHS11, TVA19, FM21}.

The moduli spaces $\mathcal{M}_{\mathrm{K}3^{[n]},2d}^\gamma$ 
(resp.\ $\mathcal{M}_{\mathrm{OG10}, 2d}^\gamma$) are 20-dimensional 
(resp.\ 21-dimensional) quasi-projective varieties with finite quotient singularities, see \cite{Vie95}. Their irreducible components $Y$ admit algebraic open embeddings to orthogonal modular varieties 
\begin{equation}\label{eq: domain quotient} Y\longrightarrow \Omega(\Lambda_h)/ \Gamma,\end{equation}
see \cites{BB66, Ver13}. The target of this map is defined as follows. Let 
\[\Lambda=U^{\oplus 3}\oplus E_8(-1)^{\oplus 2}\oplus \langle -2(n-1)\rangle \hspace{3mm} \text{ 
(respectively }
\Lambda=U^{\oplus 3}\oplus E_8(-1)^{\oplus 2}\oplus A_2(-1))\]
be the lattice isomorphic to $\left(H^2(X,\mathbb{Z}), q_X\right)$ and $\Lambda_h\subset \Lambda$ the lattice of signature $(2,20)$ (respectively $(2,21)$) isomorphic to the orthogonal complement of the first Chern class $h\in H^2(X,\mathbb{Z})$ of $H$. Then 
\[\left\{[x]\in \mathbb{P}(\Lambda_h\otimes \mathbb{C})\mid (x,x)=0 \text{ and } (x, \overline{x})>0\right\}\]
is a Type IV bounded Hermitian symmetric domain that has two isomorphic components and the orthogonal group $O(\Lambda_h)$ acts on this domain. We fix one of the components $\Omega(\Lambda_h)$ and let $O^+(\Lambda_h)$ be the subgroup of $O(\Lambda_h)$ fixing $\Omega(\Lambda_h)$. The arithmetic group $\Gamma$, called the \textit{monodromy group}, is a finite index subgroup of $O^+(\Lambda_h)$.

In the case of the moduli spaces $\mathcal{M}_{\mathrm{K}3^{[n]},2d}^\gamma$, the monodromy group $\Gamma$ was computed in \cites{Mar11, Ono22} and it is given by
\begin{equation}\label{eq: hat O}\widehat{O}^+\left(\Lambda,h\right)=\left\{g\in O^+\left(\Lambda,h\right)\mid g\mid_{D(\Lambda)=\pm {\rm{Id}}}\right\},
\end{equation}
where $O^+(\Lambda,h)\subset O(\Lambda)$ is the stabilizer of $h\in\Lambda$ acting on $\Lambda_h$ by restriction and $D(\Lambda)$ is the discriminant group of the lattice $\Lambda$. As an important ingredient in the proof of Theorem \ref{thm1} (see Lemma \ref{lemma:sec_K3n:h} and Proposition \ref{prop:sec_K3n:Otilde_h}), we show that if $n=2$ or $\gamma\geq3$, then this monodromy group $\Gamma$ has the simpler description
\begin{equation}\label{eq: tilde O}\widehat{O}^+\left(\Lambda,h\right)=\widetilde{O}^+(\Lambda_h):=\{g\in O^+(\Lambda_h)\mid g|_{D(\Lambda_h)}=\mathrm{Id}\}.
\end{equation}
 By contrast, when $n\ge 3$ and $\gamma=1,2$, the group 
 $O^+(\Lambda,h)$ is an index two extension of $\widetilde{O}^+(\Lambda_h)$. This gives rise to several technical difficulties in the proof of Theorem~\ref{thm: thm5}.

Remarkably, Ma shows in \cite{Ma18}*{Theorem 1.3} that there are only finitely many isomorphism classes of even lattices $\Lambda_h$ of signature $(2,m)$ with $m\ge 9$ for which the modular variety $\Omega(\Lambda_h)/ \widetilde{O}^+(\Lambda_h)$ is not of general type. 
Additionally, there are only finitely many lattices $\Lambda_h$ (up to isomorphism) of signature $(2,m)$ with $m\ge 21$ or $m=17$ such that $\Omega(\Lambda_h)/O^+(\Lambda_h)$ is not of general type \cite{Ma18}*{Theorem 1.1}. 

One should note however that this does \textit{not} imply that there are only finitely many choices of $(n,d, \gamma)$ such that $\mathcal{M}_{\mathrm{K}3^{[n]},2d}^\gamma$ is not of general type, since infinitely many $(n,d, \gamma)$ may yield isometric lattices $\Lambda_h$. 
In particular, $\mathcal{M}_{\mathrm{K}3^{[n]},2d}^\gamma$ may be not general type even when $n$ or $d$ are arbitrarily large. For instance, for any coprime integers $(a,b)$, letting $n=a^2-ab+b^2$, the moduli spaces
\begin{equation}
\label{BLM+}
\begin{cases}
\mathcal{M}_{K3^{[n+1]},\frac{2}{3}n}^{\frac{2}{3}n} & \mbox{ if } 3\vert n\\
\mathcal{M}_{K3^{[n+1]},6n}^{2n} & \mbox{ otherwise }
\end{cases}
\end{equation}
are unirational, dominated by the moduli space of cubic fourfolds, see \cite{BLM+21}*{Corollary 29.5}. Theorems \ref{thm:sec_intro:K32} and \ref{thm1} thus serve, once one has fixed $n$ and $\gamma$, to give an explicit lower bound on $d$ after which $\mathcal{M}_{\mathrm{K}3^{[n]},2d}^\gamma$ is always of general type. 

Note moreover that Ma's result does not apply in the cases $n\geq3$ and $\gamma=1,2$ due to the failure of \eqref{eq: tilde O}. The natural expectation that there are finitely many birational classes of moduli spaces $\mathcal{M}_{\mathrm{K}3^{[n]}, 2d}^\gamma$ that are not of general type when $\gamma=1,2$ remains open. Theorems 
\ref{thm: thm5} and \ref{thm:gamma2} are a contribution in this direction. 

\subsection{Ideas of the proof}
The 
proofs of Theorems~\ref{thm:sec_intro:K32}, \ref{thm:sec_intro:OG10}, \ref{thm1}, and \ref{thm: thm5} use the strategy developed by Gritsenko--Hulek--Sankaran in \cites{GHS07, GHS10, GHS11} involving modular forms of orthogonal type. For the relevant lattice $\Lambda_h$ of signature $(2,m)$, the quotient  
\[\mathcal{F}_{\Lambda_h}=\Omega(\Lambda_h)/ \Gamma\]
admits many projective toroidal compactifications with mild singularities, however, by \cite{GHS07}*{Theorem 2} one can choose one $\overline{\mathcal{F}}_{\Lambda_h}$ with at worst canonical singularities. Taking a resolution of singularities of $\overline{\mathcal{F}}_{\Lambda_h}$ yields a smooth projective model $\overline{Y}$ of $\mathcal{F}_{\Lambda_h}$. Showing that the relevant moduli spaces are of general type thus amounts to showing that there is an abundance of holomorphic pluri-canonical forms on $\overline{Y}$.

This is accomplished using modular forms of orthogonal type and the so-called ``low weight cusp form trick'' of Gritsenko--Hulek--Sankaran. 
A modular form of weight $k$ and character $\chi\colon \Gamma\rightarrow \mathbb{C}^*$ is a holomorphic function $F\colon\Omega^\bullet(\Lambda_h)\to\mathbb{C}$
on the affine cone $\Omega^\bullet(\Lambda_h)$ of $\Omega(\Lambda_h)$
such that for all $Z\in\Omega^\bullet(\Lambda_h)$, $t\in \mathbb{C}^*$, and $g\in \Gamma$
\[
F(tZ)=t^{-k}F(Z)\;\;\;\hbox{and}\;\;\;F(gZ)=\chi(g)\cdot F(Z).
\]
These form a finite-dimensional vector space ${\rm{Mod}}_k\left(\Gamma,\chi\right)$. In the cases we consider, every modular form is holomorphic at the boundary. 
Those vanishing on the boundary are called {\textit{cusp forms}} and form a subspace $S_k\left(\Gamma,\chi\right)\subset{\rm{Mod}}_k\left(\Gamma,\chi\right)$. By a classical result of Freitag \cite{Fre83}*{Chapter 3}, if $F$ is a cusp form of weight $m\ell$ and character $\mathrm{det}$ vanishing on the ramification divisor of $\Omega(\Lambda_h)\rightarrow \mathcal{F}_{\Lambda_h}$, then the form $F \cdot dz^{\otimes \ell}$, where $dz$ is a holomorphic volume form on $\Omega(\Lambda_h)$, descends to a global section of $\ell K_{\overline{Y}}$.

\subsubsection{Low weight cusp form trick} Gritsenko--Hulek--Sankaran's ``low-weight cusp form trick'' \cite{GHS07} gives a method to produce such pluricanonical forms on $\overline{Y}$. To do this, by Freitag's result, we need to produce a fixed cusp form $F_a\in S_{a}\left(\Gamma,\det\right)$ with weight $a<m$ vanishing on the ramification divisor of $\Omega(\Lambda_h)\longrightarrow \mathcal{F}_{\Lambda_h}$. As long as $\Gamma$ has no irregular cusps \cite{Ma21} (see Section \ref{SS:lwcusp}), there is then an injection \cite{GHS07}*{Theorem 1.1}
\[
\begin{aligned}
{\rm{Mod}}_{(m-a)\ell}\left(\Gamma,1\right)&\longrightarrow H^0(\overline{Y}, \ell K_{\overline{Y}})\\
\rho&\mapsto \rho \cdot F_a^\ell \cdot dz^{\otimes \ell}.
\end{aligned}
\]
By Hirzebruch--Mumford Proportionality \cite{Mum77}, the dimension of ${\rm{Mod}}_{(m-a)\ell}\left(\Gamma,1\right)$ grows like $\ell^m$ and from this it follows that $\overline{Y}$ is indeed of general type.

\subsubsection{Borcherds form and root counting}
In order to produce this low-weight cusp form $F_a$, one uses Borcherds modular form \cite{Bor95}
$\Phi_{12}\in M_{12}\left(O^+(\mathrm{II}_{2,26}),{\rm{det}}\right)$
where $\mathrm{II}_{2,26}$ is the unique even unimodular lattice of signature $(2,26)$ given by
\begin{equation}
{\rm{II}}_{2,26}:=U^{\oplus 2}\oplus E_{8}(-1)^{\oplus 3}.
\end{equation}
If one can produce a primitive embedding of lattices 
\begin{equation}
\label{embedding}
\Lambda_h\hookrightarrow \mathrm{II}_{2,26} \end{equation}
with the right properties (see e.g.\ Propositions~\ref{prop:sec_K3n:F=0}, \ref{prop:sec_K3n:R(Q)}, \ref{prop:sec_K3n+:Fcusp}, \ref{prop:sec_OG10:mod}), then the so-called ``quasi-pullback'' of $\Phi_{12}$ to the cone $\Omega^\bullet\left(\Lambda_h\right)$
yields the needed low weight cusp form $F_a$. Thus the proofs of the main results boil down to producing embeddings $\Lambda_h\hookrightarrow \mathrm{II}_{2,26}$ with the desired properties. 
Modularity of the quasi-pullback is not automatic and strongly depends on the arithmetic group involved. 
The condition is automatically satisfied if $\Gamma$ is the group $\widetilde{O}^+\left(\Lambda_h\right)$ defined in \eqref{eq: tilde O}, but is much more delicate when $\Gamma$ is larger. 
This is one source of difficulty for $\gamma=1,2$. Moreover, in these cases, the embedding \eqref{embedding} has to be chosen so that isometries of $\Lambda_h$ in the group $\Gamma$ extend to isometries of ${\rm{II}}_{2,26}$. 
This is again automatic when $\Gamma=\widetilde{O}^+(\Lambda_h)$, but if $\Gamma$ is larger, 
it is a more delicate endeavor.

Once we have an embedding where one can ensure modularity of the quasi-pullback, then
vanishing at the ramification divisor follows in the case $\gamma\geq3$ by the classification of irreducible low-rank lattices of small discriminant \cite{CS88}. 
In the case $\gamma=1,2$, this problem is again more challenging and we can solve it only by imposing additional constraints on both the dimension $2n$ and the degree $2d$.

\subsection*{Acknowledgements}
The paper benefited from helpful discussions and correspondence with the following people who we gratefully acknowledge: Olivier Debarre, Klaus Hulek, Paul Kiefer, Shouhei Ma, Emanuele Macr\`i, Lillian Pierce, Francesca Rizzo, Preston Wake. 
Finally we would also like to thank the
anonymous referee for many helpful comments and suggestions.

\section{Preliminaries}
\label{S:prelim}
\subsection{Lattices}
\label{SS:lattices}
Let $L$ be a lattice, i.e.\ a free abelian group of finite rank together with a symmetric non-degenerate bilinear pairing $\left(\cdot,\cdot\right):L\times L\to \mathbb{Z}$. 
 We denote by $O(L)$ the group of isometries of $L$ and by $L(m)$, for an integer $m$, the lattice whose underlying abelian group is $L$ and whose pairing is $m$ times the pairing of $L$.

 The dual lattice $L^{\vee}={\rm{Hom}}(L,\mathbb{Z})$ can be embedded in $L\otimes \mathbb{Q}$ as those elements $x\in L\otimes \mathbb{Q}$ such that $(x,\ell)\in \mathbb{Z}$ for all $\ell\in L$, and the bilinear form is the restriction of the $\mathbb{Q}$-linear extension of $(\cdot,\cdot)$. 
 The inclusion $L\subset L^{\vee}$ has finite index and the quotient $D(L)=L^{\vee}\big/L$ is a finite abelian group, called the \textit{discriminant group}. When $L$ is even, i.e.\ $(x,x)\in 2\mathbb{Z}$ for all $x\in L$, the discriminant group $D(L)$ comes endowed with a $\mathbb{Q}\big/2\mathbb{Z}$-valued quadratic form given by $x+L\mapsto (x,x)+2\mathbb{Z}$. In this case, the natural projection induces a homomorphism 
\begin{equation}
\label{eq:sec_prelim:pi}
\pi\colon O(L)\to O(D(L))
\end{equation}
that is surjective if $L$ is indefinite and $D(L)$ can be generated by at most ${\rm{rk}}(L)-2$ elements. The following two groups play a prominent role:
\[\widetilde{O}(L)=\pi^{-1}\left({\rm{Id}}\right)\;\;\;\hbox{and}\;\;\;\widehat{O}(L)=\pi^{-1}\left({\rm{\pm Id}}\right).\]
The first one is called the {\textit{stable orthogonal group}}. 
For $h\in L$ we denote by $O(L,h)$ the stabilizer of $h$ in $O(L)$, and by $\widetilde{O}(L,h)\subset O(L,h)$ (resp. $\widehat{O}(L,h)$) the subgroup of $\widetilde{O}(L)$ (resp. $\widehat{O}(L)$) of elements fixing $h$. 

Let $h\in L$, then $\left(h,L\right)\subset \mathbb{Z}$ is an ideal of the form $\gamma \mathbb{Z}$, where $\gamma$ is a positive integer. We call $\gamma$ the {\textit{divisibility}} of $h$ and write ${\rm{div}}(h)=\gamma$. 
Alternatively, $\gamma$ can be defined as the positive integer making $h/\gamma$ a primitive element in $L^{\vee}$. We denote by $h_*$ the class of $h/\gamma$ in $D(L)$.

The lattice $E_8$ plays a crucial role in our arguments. It can be realized as a sublattice of $\left(\frac{1}{2}\mathbb{Z}\right)^{\oplus 8}$ with the euclidean quadratic form via
\[E_8=\left\{(x_1,\ldots,x_8)\in \mathbb{Z}^{\oplus8}\cup\left(\frac{1}{2}+\mathbb{Z}\right)^{\oplus8}\left|\sum_{i=1}^8x_i\equiv 0\mod 2\right.\right\}.\]
Elements in $E_8$ are points in $\mathbb{Q}^{\oplus 8}$ whose coordinates are either all integers or all half integers,  
such that the sum of the coordinates is an even integer. Recall that $E_8$ has $240$ \emph{roots}, elements $r$ with $(r,r)=2$. 
The $112$ of these roots that have integer coordinates are of the form $\pm e_i\pm e_j$, for $i,j\in\{1,\ldots,8\}$ and $i\neq j$, where $\{e_1,\ldots,e_8\}$ is the standard basis of $\mathbb{Z}^{\oplus 8}$. 
The remaining $128$ roots have half-integer coordinates and are of the form $\frac{1}{2}\sum_{i=1}^8 \pm e_i$, where the number of minus signs is even, or equivalently, the sum of all coordinates is even. 
Let us call the first set of roots {\textit{integral}} and the second {\textit{fractional}}. 
We view the lattice $D_8$ as a sublattice of $E_8$ via
\[D_8=\left\{(x_1,\ldots,x_8)\in \mathbb{Z}^{\oplus8}
\left|\sum_{i=1}^8x_i\equiv 0\mod 2\right.\right\}.\]
Finally, recall that if $L$ is a fixed even lattice, there is natural one-to-one correspondence between finite index even overlattices $L\subset L'$ and isotropic subgroups $H\subset D(L)$. Namely, given $L\subset L'$, there is a sequence of inclusions
\begin{equation}
\label{eq:sec_prelim:inclusions}
L\subset L'\subset (L')^{\vee}\subset L^\vee
\end{equation}
and the isotropic subgroup associated to $L'$ is given by $$H=L'\big/L\subset L^\vee\big/L=D(L).$$ 
Conversely, given an isotropic subgroup $H\subset D(L)$, the overlattice $L'$ is given by $\pi^{-1}(H)$, where $\pi\colon L^{\vee}\to D(L)$ is the quotient map. The discriminant groups of $L$ and $L'$ are related by the following diagram
\begin{equation}
\label{eq:sec_prelim:DLDL'}
\begin{tikzcd}
H=L'\big/L\arrow[r, hook]&H^{\perp}=(L')^{\vee}\big/L\arrow[r, hook]\arrow[d, twoheadrightarrow]&D(L)\\
&D(L')=H^\perp\big/H.&
\end{tikzcd}
\end{equation}
An element $g\in O(L)$ can be extended to $L'$ if and only if $\overline{g}(H)=H$, where $\overline{g}\in O(D(L))$ is the image of $g\in O(L)$ under the projection \eqref{eq:sec_prelim:pi}. 

Recall that a lattice embedding $L\hookrightarrow L'$ is said to be {\textit{primitive}} if the quotient $L'/L$ is torsion free. We now place ourselves in the following situation. Let $\Lambda$ be an even lattice, $h\in \Lambda$ a primitive element of degree $2d$ and divisibility $\gamma$, and $\Lambda_h$ the orthogonal complement of $h$ in $\Lambda$. Then
\begin{equation}
\label{eq:sec_prelim:D(Lambda_h+h)}
D\left(\Lambda_h\oplus \langle h\rangle\right)=D\left(\Lambda_h\right)\oplus \mathbb{Z}\big/2d\mathbb{Z},
\end{equation}
where the last factor is the discriminant group of $\langle h \rangle$ generated by $\frac{1}{2d}h \mod \langle h\rangle$. Now $\Lambda$ is a finite index overlattice of $\Lambda_h\oplus \langle h\rangle$ that corresponds to the isotropic subgroup 
$$H=\Lambda\big/\left(\Lambda_h\oplus\langle h \rangle\right)$$
in \eqref{eq:sec_prelim:D(Lambda_h+h)}. Moreover, since both inclusions $\Lambda_h\subset \Lambda$ and $\langle h \rangle\subset \Lambda$ are primitive, one checks that the projections 
\begin{equation}
\label{eq:sec_prelim:Pi_H}
\pi_1:H\to D(\Lambda_h)\;\;\;\hbox{and}\;\;\;\pi_2:H\to D(\langle h\rangle)
\end{equation}
are injective, see also \cite{Nik80}*{Proposition 1.5.1}. That is, for any $x\in \pi_1(H)$ there exists a unique $y\in \pi_2(H)$ such that $x+y\in H$. The induced map $\pi_1(H)\to\pi_2(H)$ defines an automorphism of $H$ and if $g\in O(\Lambda,h)$, the restriction $g|_{\Lambda_h}$ acts as the identity on $\pi_1(H)$. Moreover, $\pi_1(H)\to \pi_2(H)$ respects the bilinear form up to a sign.

\begin{lemma}
\label{lemma:sec_prelim:O-tildes}
Let $\Lambda$ be an even lattice and $h\in \Lambda$ primitive.
View $\widetilde{O}(\Lambda,h)$ as a subgroup of $O(\Lambda_h)$ via the restriction map $O(\Lambda,h)\to O(\Lambda_h)$. Then $\widetilde{O}(\Lambda,h)$ contains $\widetilde{O}(\Lambda_h)$.
\end{lemma}

\begin{proof}
Let $g\in \widetilde{O}(\Lambda_h)$ and consider the orthogonal transformation given by $g\oplus{\rm{Id}}$ on the lattice $\Lambda_h\oplus \langle h \rangle$. Then $g\oplus {\rm{Id}}$ acts as the identity on $D(\Lambda_h)\oplus D(\langle h \rangle)$. In particular it fixes $H$, so $g\oplus {\rm{Id}}$ extends to an element of $O(\Lambda,h)$. Moreover, $\overline{g\oplus{\rm{Id}}}$ is the identity on $H^{\perp}$. From \eqref{eq:sec_prelim:DLDL'} it follows that the extension must act as the identity on $D(\Lambda)=H^{\perp}\big/H$.
\end{proof}

\subsection{Modular forms}
Let $L$ be an even lattice of signature $(2,q)$ with $q\geq3$. Recall that the symmetric domain
\[\left\{[x]\in \mathbb{P}\left(L\otimes \mathbb{C}\right)\left|x^2=0\hbox{ and }x\cdot \overline{x}>0\right.\right\}\]
consists of two components exchanged by complex conjugation and is acted on by the arithmetic group $O(L)$. We fix one of the two connected components, denoted $\Omega(L)$ and called the {\textit{period domain}} for $L$. The index two subgroup of orientation-preserving isometries $O^+(L)\subset O(L)$ is the subgroup of $O(L)$ fixing $\Omega(L)$. 

We denote by $\Omega^\bullet(L)\subset L\otimes\mathbb{C}$ the affine cone of $\Omega(L)\subset \mathbb{P}(L\otimes \mathbb{C})$ and let $\Gamma$ be a finite index subgroup of $O^+(L)$. 
A {\textit{modular form of weight $k$ and character $\chi\colon\Gamma\to\mathbb{C}^*$}} is a holomorphic function $F\colon\Omega^\bullet(L)\longrightarrow \mathbb{C}$ such that for all $Z\in\Omega^\bullet(L)$, $t\in \mathbb{C}^*$, and $g\in \Gamma$
\begin{equation}
\label{eq:sec_prelim:mod}
F(tZ)=t^{-k}F(Z)\;\;\;\hbox{and}\;\;\;F(gZ)=\chi(g)\cdot F(Z).
\end{equation}
Modular forms of fixed weight and character form a finite-dimensional vector space ${\rm{Mod}}_k\left(\Gamma,\chi\right)$. When $q\geq 3$, every 
modular form is holomorphic at the boundary. Those vanishing on the boundary are called {\textit{cusp forms}} and form a subspace denoted $S_k\left(\Gamma,\chi\right)\subset{\rm{Mod}}_k\left(\Gamma,\chi\right)$. This subspace is fundamental in many ways. By a classical result of Freitag \cite{Fre83}*{Chapter 3}, cusp forms of weight $k=q$ and character $\mathrm{det}$ descend to global sections of the canonical divisor for any smooth model of the quotient. More precisely:
\begin{equation}
\label{eq:sec_prelim:Fre}
S_q\left(\Gamma,{\rm{det}}\right)\cong H^0\left(\overline{Y}, K_{\overline{Y}}\right),
\end{equation}
where $\overline{Y}$ is a smooth projective model of the quasi-projective variety $\Omega(L)\big/\Gamma$, see \cite{BB66}. 

\subsection{Irregular cusps and the low-weight cusp form trick}
\label{SS:lwcusp}
The period domain $\Omega(L)$ is an open subset of the isotropic quadric $\mathcal{Q}\subset\mathbb{P}(L\otimes \mathbb{C})$ defined by $(w,w)=0$. It has two types of boundary components called \textit{cusps}: $0$-dimensional and $1$-dimensional ones, corresponding to isotropic sublattices of $L$ of rank one and two respectively.

The Fourier expansion of a modular form $F\in {\rm{Mod}}_k\left(\Gamma,\chi\right)$ at a $0$ -dimensional cusp is defined via Eichler transvections \cite{Eic74}*{Section 3} (see also \cite{GHS13}*{Section 8.3}). 
Recall that if $I$ is an isotropic rank one sublattice of $L$ and $U(I)_{\mathbb{Q}}$ is the unipotent part of the stabilizer of $I$ in $O^+\left(L\otimes\mathbb{Q}\right)$, the group of translations defining the Fourier expansion of $F$ at the cusp associated to $I$ is $U(I)_{\mathbb{Q}}\cap \Gamma$, whereas the lattice of translations around the cusp $I$ in the $\Gamma$-action is given by $U(I)_{\mathbb{Q}}\cap \langle \Gamma,-{\rm{Id}}\rangle$. 
When these two groups do not coincide, $I$ is called an \textit{irregular cusp} \cite{Ma21}. Irregular $1$-dimensional cusps are defined in a similar fashion and their existence can be reduced to the $0$-dimensional case:

\begin{proposition}[\cite{Ma21}*{Corollary 6.5}]
\label{prop:sec_prelim:irr_1cusp}
Let $\Gamma\subset O^+(L)$ be a finite index subgroup. If $\Omega(L)$ has no irregular $0$-dimensional cusps for $\Gamma$, then it has no irregular $1$-dimensional cusps. 
\end{proposition}

After Freitag's result \eqref{eq:sec_prelim:Fre}, the fundamental link between the existence of low-weight cusp forms and the Kodaira dimension of orthogonal Shimura varieties is the following theorem due to Gritsenko--Hulek--Sankaran, with a correction by Ma:

\begin{theorem}[\cite{GHS07}*{Theorem 1.1} and \cite{Ma21}*{Theorem 1.2}]
\label{thm:sec_prelim:low-weight}
Let $L$ be a lattice of signature $(2,q)$ with $q\geq 9$ and $\Gamma\subset O^+(L)$ a finite index subgroup with no irregular cusps. Then the modular variety $\Omega(L)\big/\Gamma$ is of general type if there exists a character $\chi$ and a cusp form $F\in S_a\left(\Gamma,\chi\right)$ of weight $a<q$ that vanishes with order at least one at the ramification divisor of the projection \[\Omega(L)\longrightarrow \Omega(L)\big/\Gamma.\]
Moreover, if $S_q\left(\Gamma,{\rm{det}}\right)\neq 0$, then $\Omega(L)\big/\Gamma$ has non-negative Kodaira dimension.
\end{theorem}

\begin{remark}
The condition of $\Gamma$ not having irregular cusps is fairly mild and for our purposes it will be satisfied with the possible exception of very few cases. This hypothesis can be dropped by requiring that $F$ vanishes with certain order at the boundary of a toroidal compactification of the modular variety, see \cite{Ma21}*{Theorem 1.2} for details. 
\end{remark}

An immediate observation is that if $-{\rm{Id}}\in \Gamma$, then $\Gamma$ has no irregular cusps. We explain now a second criterion ensuring the non-existence of irregular cusps. Suppose $I\subset L$ is an isotropic sublattice defining a $0$-dimensional cusp. Consider the lattice
$$L(I):=\left(I^\perp\big/I\right)\otimes I$$
and let $\Gamma(I)_{\mathbb{Q}}$ be the stabilizer of $I$ in $O^+(L\otimes\mathbb{Q})$. For $m\otimes l\in L(I)_{\mathbb{Q}}$, the \textit{Eichler transvection} $E_{m\otimes l}\in \Gamma(I)_{\mathbb{Q}}$ is defined by 
\[E_{m\otimes l}(v)=v-(\widetilde{m},v)l+(l,v)\widetilde{m}-\frac{1}{2}(m,m)(l,v)l,\]
for any $v\in L_{\mathbb{Q}}$, where $\widetilde{m}\in I_{\mathbb{Q}}^\perp$ is an arbitrary lift of $m\in (I^\perp/I)_{\mathbb{Q}}$, see \cites{Eic74,Sca87}. The construction induces a canonical isomorphism
\begin{equation}
\label{eq:sec_prelim:E_w}
\begin{array}{rcl}
L(I)_{\mathbb{Q}}&\longrightarrow& U(I)_{\mathbb{Q}}\subset \Gamma(I)_{\mathbb{Q}}\\
m\otimes l&\mapsto&E_{m\otimes l}.
\end{array}
\end{equation}

\begin{proposition}[\cite{Ma21}*{Proposition 3.1}]
\label{prop:sec_prelim:irr_cusp}
The $0$-dimensional cusp $I$ is irregular for $\Gamma\subset O^+(L)$ if and only if $-\mathrm{Id}\not \in \Gamma$ and $-E_w\in \Gamma \cap \Gamma(I)_{\mathbb{Q}}$ for some $w\in L(I)_{\mathbb{Q}}$.
\end{proposition}

In Section \ref{s:K3n+}, we use \cite{Ma21}*{Proposition 1.1} together with Propositions \ref{prop:sec_prelim:irr_1cusp} and \ref{prop:sec_prelim:irr_cusp} to rule out the existence of irregular cusps in the cases we consider. 

\subsection{The Borcherds modular form and quasi-pullback}
\label{SS:Phi_{12}}
Cusp forms of low weight are rare in nature. As in \cites{GHS07, GHS10, GHS11, TVA19}, we produce the necessary cusp form to apply Theorem \ref{thm:sec_prelim:low-weight} by using Borcherds form found in \cite{Bor95}:
\[\Phi_{12}\in M_{12}\left(O^+(\mathrm{II}_{2,26}),{\rm{det}}\right),\]
where $\mathrm{II}_{2,26}$ is the unique even unimodular lattice of signature $(2,26)$ given by
\begin{equation}
{\rm{II}}_{2,26}:=U^{\oplus 2}\oplus E_{8}(-1)^{\oplus 3}.
\end{equation}
We denote by $R(\mathrm{II}_{2,26})$ the set of $(-2)$-roots and for a primitive embedding of lattices $L\hookrightarrow {\rm{II}}_{2,26}$, denote by $R(L^\perp)$ the set of $(-2)$-roots in the orthogonal complement $L^\perp\subset \mathrm{II}_{2,26}$. 
Let $r\in R(L^\perp)$ and let $\sigma_r\in O^+({\rm{II}}_{2,26})$ be the reflection with respect to $r$:
\begin{equation}
\label{eq:sec_prelim:sigma_r}\sigma_r(v)=v-\frac{2(v,r)}{(r,r)}r.
\end{equation}
The definition \eqref{eq:sec_prelim:mod} of a modular form implies that for any $Z$ in the $\sigma_r$-invariant hyperplane
\[H_r:=r^{\perp}\cap \Omega^\bullet({\rm{II}}_{2,26})\subset {\rm{II}}_{2,26}\otimes\mathbb{C}\]
one has $\Phi_{12}(Z)=-\Phi_{12}(Z)$. In particular $\Phi_{12}(Z)$ vanishes along $H_r$. Furthermore, $\Phi_{12}$ vanishes only at the union of all $H_r$ with $r\in R(\mathrm{II}_{2,26})$ and the vanishing multiplicity is one, see \cite{Bor95}*{Section 10, Example 2} and \cite{BKPS98}. Note that $H_r=H_{-r}$ and the image of the induced embedding 
$\Omega^\bullet(L)\longrightarrow\Omega^\bullet(\mathrm{II}_{2,26})$
lands inside $H_r$ for any $\pm r\in R(L^{\perp})$. To get a non-zero modular form on $\Omega^\bullet(L)$ by means of restricting $\Phi_{12}$ one has to divide by the corresponding linear factors $(Z,r)$, one for each $\pm r\in R(L^\perp)$. 

We fix a choice of {\textit{positive roots}}, 
i.e.\ a subset $R(L^\perp)_{>0}\subset R(L^\perp)$ such that $R(L^\perp)_{>0}$ and $-R(L^\perp)_{>0}$ are disjoint and their union is 
all of $R(L^\perp)$. We call $-R(L^\perp)_{>0}$ the set of {\textit{negative roots}} and denote it by $R(L^\perp)_{<0}$. The function on $\Omega^\bullet(L)$ defined by
\begin{equation}
\label{eq:sec_prelim:qpullback}
F(Z)=\left.\frac{\Phi_{12}(Z)}{\prod_{r\in R(L^\perp)_{>0}}(Z,r)}\right|_{\Omega^\bullet(L)}
\end{equation}
is called the \textit{quasi-pullback} of $\Phi_{12}$ to $\Omega^\bullet(L)$. Observe that for $t\in \mathbb{C}^*$, 
\begin{equation}\label{eq_weight_quasipullback}
F(tZ)=t^{-(12+\left|R_{>0}\right|)}F(Z).
\end{equation}

On the other hand, the second condition of \eqref{eq:sec_prelim:mod} is not always guaranteed for the full group $O^{+}(L)$ and a given character $\chi$. Let us assume that $g\in O^{+}(L)$ is the restriction of $\widetilde{g}\in O^{+}({\rm{II}}_{2,26})$. Then $\widetilde{g}\mid_{L^\perp}$ permutes the roots in $R(L^\perp)$ and for $Z\in \Omega^\bullet(L)$,
$$\prod_{r\in R_{>0}}(g(Z),r)=\prod_{r\in R_{>0}}(\widetilde{g}(Z),r)=\prod_{r\in R_{>0}}(Z,\widetilde{g}^{-1}(r))=(-1)^M\prod_{r\in R_{>0}}(Z,r),$$
where $M$ is the number of sign-changing roots via $\widetilde{g}$, i.e.,
\begin{equation}
\label{eq:sec_prelim:M}M=\left|\widetilde{g}^{-1}\left(R_{>0}\right)\cap R_{<0} \right|=\left|\widetilde{g}\left(R_{>0}\right)\cap R_{<0}\right|.
\end{equation}

\begin{lemma}
\label{lemma:sec_premil:F}
The quasi-pullback $F$ of $\Phi_{12}$ is modular with respect to 
\[\chi\colon\Gamma\longrightarrow \mathbb{C}^*\] 
for a finite index subgroup $\Gamma\subset O^+(L)$ if for every $g\in \Gamma$ there exists an extension $\widetilde{g}\in O^+({\rm{II}}_{2,26})$ such that
\[\chi(g)=(-1)^M\cdot {\rm{det}}(\widetilde{g}),\]
where $M$ is the number of sign-changing roots via $\widetilde{g}$ in $R(L^\perp)$ defined in \eqref{eq:sec_prelim:M}.
\end{lemma}
\begin{proof}
It follows immediately from the definitions:
\[F(g(Z))=\frac{{\rm{det}}(\widetilde{g})\Phi_{12}(Z)}{(-1)^M\prod_{r\in R_{>0}}(Z,r)}=(-1)^{M}{\rm{det}}(\widetilde{g})F(Z). \qedhere\]
\end{proof}

As an immediate consequence we have the following:

\begin{corollary}
\label{coro:sec_prelim:ModOtilde}
For any finite index subgroup $\Gamma\subset \widetilde{O}^+(L)$, the quasi-pullback $F$ is modular with respect to 
${\rm{det}}\colon\Gamma\longrightarrow \mathbb{C}^*$.
\end{corollary}

\begin{proof}
Every element $g$ in $\Gamma\subset\widetilde{O}^+(L)$ admits an extension $\widetilde{g}\in O^{+}({\rm{II}}_{2,26})$ such that $\widetilde{g}$ restricts to the identity on $L^{\perp}\subset {\rm{II}}_{2,26}$, \cite{Nik80}*{Theorem 1.6.1, Corollary 1.5.2}, see also \cite{Huy16}*{Chapter 14, Proposition 2.6}. In particular $M=0$ and ${\rm{det}}(g)={\rm{det}}(\widetilde{g})$.
\end{proof}

Provided that there are no irregular cusps for $\Gamma\subset O^+\left(L\right)$, in order to use Theorem \ref{thm:sec_prelim:low-weight} it is not enough to show that the quasi-pullback $F$ of Borcherds modular form $\Phi_{12}$ is modular and of low weight, one also has to ensure that it vanishes at all the cusps
and along the ramification divisor of the projection
\[\pi_L:\Omega(L)\longrightarrow \Omega(L)\big/\Gamma.\]
Let us place our attention on the latter. The ramification divisor of the projection $\pi_L$ is given by the union of all reflective divisors
\begin{equation}
\label{eq:sec_prelim:Ram}
{\rm{Ram}}(\pi_L)=\bigcup_{\substack{r\in L\hbox{ primitive}\\(r,r)<0\\ \sigma_r\hbox{ or }-\sigma_r\in \Gamma}}\mathcal{D}_r,
\end{equation}
where $\mathcal{D}_r=\left\{[Z]\in \Omega(\Lambda_h)\mid (Z,r)=0\right\}$ and $\sigma_r$ is the reflection with respect to $r\in L$, see \cite{GHS07}*{Corollary 2.13}. An element $r\in L$ of negative square $(r,r)<0$ such that $\sigma_r$ or $-\sigma_r\in\Gamma\subset O^+(L)$ is called a \textit{reflective element for $\Gamma$}. The condition $\pm\sigma_r\in \Gamma$ in general imposes strong restrictions on the numbers $(r,r)$ and $\mathrm{div}(r)$. For instance, for $\sigma_r$ to be integral, i.e., an element in $O(L)$, one has to have 
\begin{equation}
\label{eq:sec_prelim:sigma_r1}
(r,r)\in\left\{{\rm{div}}(r), 2\mathrm{div}(r)\right\}.
\end{equation}
An immediate consequence of the description of the ramification divisor in terms of reflective divisors is the following:
\begin{lemma}\label{lemma_vanishing_ramif}
Let $G\in \mathrm{Mod}_k\left(\Gamma,\mathrm{det}\right)$ and assume that
$\mathrm{rk}(L)\equiv k\mod2$.
Then $G$ vanishes at $\mathrm{Ram}(\pi_L)$.
\end{lemma}
\begin{proof}
Let $[Z]\in \mathcal{D}_r$. If $\sigma_r\in \Gamma$, then modularity of $G$ for $\mathrm{det}\colon\Gamma\longrightarrow \mathbb{C}^*$ and the fact that $\sigma_r(Z)=Z$ if $(Z,r)=0$ imply
\[G(Z)=G(\sigma_r(Z))=-G(Z)\]
and $G$ vanishes at $\mathcal{D}_r$. If $-\sigma_r\in \Gamma$, then modularity implies 
$$G(-Z)=G(-\sigma_r(Z))=(-1)^{\mathrm{rk}(L)+1}G(Z)\;\;\;\hbox{and}\;\;\;G(-Z)=(-1)^kG(Z).$$
If $\mathrm{rk}(L)$ and $k$ have the same parity, then $G(Z)=-G(Z)$ and $G$ vanishes at $\mathcal{D}_r$.
\end{proof}

As $\mathrm{II}_{2,26}$ is unimodular,
if $\sigma_r$ or $-\sigma_r\in O^+(\mathrm{II}_{2,26})$, then $(r,r)=\pm2$. Moreover, Borcherds form $\Phi_{12}\in M_{12}\left(\mathrm{II}_{2,26}, \mathrm{det}\right)$ vanishes with order one at all reflective divisors associated to $(-2)$-roots $r\in \mathrm{II}_{2,26}$, see \cites{Bor95, BKPS98} (see also \cite{GHS07}*{Section 6}). 

\begin{proposition}
\label{prop:sec_prelim:F=0}
Let $L$ be an even lattice of signature $(2,q)$ with $3\leq q\leq 26$ and $L\hookrightarrow \mathrm{II}_{2,26}$ a primitive embedding. Assume that the quasi-pullback $F$ of Borcherds form $\Phi_{12}$ to $\Omega(L)$ is modular with character ${\mathrm{det}}\colon\Gamma\longrightarrow \mathbb{C}^*$. Let $L_r=r^\perp\subset L$ be the orthogonal complement of a reflective element $r\in L$ and consider the induced primitive embedding $L_r\hookrightarrow \mathrm{II}_{2,26}$. If for every reflective element $r\in L$ we have 
\[\left|R(L^\perp)\right|<\left|R(L_r^\perp)\right|,\]
then $F$ vanishes along the ramification divisor of the modular projection
\[\pi_L\colon\Omega(L)\longrightarrow\Omega(L)\big/\Gamma.\]
\end{proposition}

\begin{proof}
The components of the ramification divisor \eqref{eq:sec_prelim:Ram} are given by reflective divisors
\begin{equation}
\label{eq:sec_prelim:Dr}
\mathcal{D}_r=\left\{\left[Z\right]\in \Omega(L)\mid \left(Z,r\right)=0\right\}\cong \Omega(L_r)\subset \Omega(L),
\end{equation}
where $(r,r)<0$ and $\sigma_r$ or $-\sigma_r\in \Gamma$. If $\sigma_r\in \Gamma$, then modularity of $F$ with respect to $\mathrm{det}$ gives us the vanishing we want. So assume $-\sigma_r\in\Gamma$. Recall that for any primitive $(2,q)$-sublattice $S\subset \mathrm{II}_{2,26}$, Borcherds form $\Phi_{12}$ vanishes at $\Omega(S)\subset \Omega(\mathrm{II}_{2,26})$ with order $\left|R(S^{\perp})\right|\big/2$. In particular, the order of vanishing of $\Phi_{12}$ at $\mathcal{D}_r$ is $\left|R(L_r^\perp)\right|\big/2$ and by construction \eqref{eq:sec_prelim:qpullback}, the quasi-pullback $F$ vanishes at $\mathcal{D}_r$ with order
\[\mathrm{ord}_{\mathcal{D}_r}\left(F\right)=\frac{\left|R(L_r^\perp)\right|-\left|R(L^\perp)\right|}{2}>0. \qedhere\]
\end{proof}

\subsection{Monodromy groups, moduli spaces, and components}

Let $X$ be a projective hyperk\"{a}hler manifold $X$. The second cohomology group $H^2(X,\mathbb{Z})$ comes endowed with a lattice structure induced by a quadratic form $q_X$ known as the {\textit{Beauville--Bogomolov--Fujiki form}}.
From now on, we focus on those $X$ of OG10 or K$3^{[n]}$ type. In these cases,
the lattice $H^2(X,\mathbb{Z})$ is isomorphic to 
\begin{equation}
\label{eq:sec_prelim:Lambda0}
\Lambda=U^{\oplus 3}\oplus E_8(-1)^{\oplus 2}\oplus L,
\end{equation}
where $L=A_2(-1)$ if $X$ is of OG10 type and $L= \langle -2(n-1)\rangle$ if $X$ is of K$3^{[n]}$ type. We call an isomorphism $\eta\colon H^2(X,\mathbb{Z})\longrightarrow \Lambda$ a {\textit{marking}}. Let $H$ be a polarization on $X$ with first Chern class $h\in H^2(X,\mathbb{Z})$ and we assume it to be primitive. 
The $O(\Lambda)$-orbit of $\eta(h)$ is called a {\textit{polarization type}} and it is by definition independent of the marking. We denote the polarization type of $h$ by ${\mathfrak{h}}$. 
There is a moduli space $\mathcal{M}_{\Lambda,\mathfrak{h}}$ parameterizing pairs $(X,H)$, where $X$ is a projective hyperk\"ahler variety as above, $H$ is a primitive polarization, $H^2(X,\mathbb{Z})\cong \Lambda$, and the $O(\Lambda)$-orbit of $\eta(c_1(H))$ is $\mathfrak{h}$ for any marking $\eta$, see \cite{Vie95}. 
The degree $2d$ and divisibility $\gamma$ are constant for any polarization of a given polarization type, but there may be more than one polarization type for the same pair $(\deg,\divis)=(2d,\gamma)$. We denote by $\mathcal{M}_{\mathrm{K}3^{[n]}, 2d}^\gamma$ the union of all moduli spaces 
\begin{equation}
\label{eq:sec_prelim:moduli}
\mathcal{M}_{\mathrm{K}3^{[n]}, 2d}^\gamma=\bigcup_{\substack{\rm{deg}(\mathfrak{h})=2d\\ {\rm{div}}(\mathfrak{h})=\gamma}}\mathcal{M}_{\Lambda, \mathfrak{h}},
\end{equation}
where $\Lambda$ is the K3${}^{[n]}$ lattice defined in \eqref{eq:sec_prelim:Lambda0}. The definition of $\mathcal{M}_{\mathrm{OG}10, 2d}^\gamma$ is analogous. Note that for each polarization type $\mathfrak{h}$, the moduli space $\mathcal{M}_{\Lambda,\mathfrak{h}}$ may have several components. In the OG10 case, after fixing $2d$ and $\gamma$, there is at most one polarization type \cite{Son21}*{Proposition 3.6} and by \cites{Ono22, Ono22b} the moduli spaces $\mathcal{M}_{\mathrm{OG}10, 2d}^\gamma$ are irreducible, see also \cite{Son21}*{Proposition 3.4}.

Let $h\in H^2(X,\mathbb{Z})$ be the first Chern class of a primitive polarization on $X$ of degree given by $q_X(h)=2d$. Recall the following classical result \cite{GHS11}*{Lemma 3.3} known as \textit{Eichler's criterion}. 

\begin{theorem}
\label{thm:sec_prelim:Eichler}
Let $L$ be an even lattice containing two copies of the hyperbolic lattice $U$. Two primitive elements $h, h'$ are in the same $\widetilde{O}(L)$-orbit if and only if they have the same square $(h,h)=(h',h')$ and the same class $h_*=h'_*$ in $D(L)$.
\end{theorem}

In particular two primitive elements $h,h'\in \Lambda$ with the same degree and discriminant class define the same polarization type. In light of the criterion, if $X$ is of K$3^{[n]}$-type we can always assume 
\[h=\gamma(e+tf)-a\ell,\]
for appropriate $t$ and $a$, where $\{e,f\}$ is the standard basis of a copy of $U$, $\ell$ is the generator of the last factor $\langle-2(n-1)\rangle$, and $\gamma$ is the divisibility of $h$. Similarly, if $X$ is of OG10-type, we can choose $h$ of the form $h=\gamma(e+tf)+v$ with $v=0$ or $v$ a primitive element in $A_2(-1)$. We will make these choices explicit later on.

Let $(X,H)$ be a primitively polarized hyperk\"{a}hler variety of K3${}^{[n]}$ or OG10 type. After fixing a marking $\eta$ for $(X,H)$, the group of monodromy operators acting on $\Lambda$ fixing $h$ is denoted ${\rm{Mon}}^2(\Lambda,h)=\eta \circ {\rm{Mon}}^2(X,h)\circ \eta^{-1}$. In both cases, we have
\begin{equation}
\label{eq:sec_prelim:Mon}
{\rm{Mon}}^2(X,h)=\widehat{O}^+(\Lambda,h)
\end{equation}
see \cites{Mar11, Ono22}. In particular, the definition does not depend on the choice of $(X,H)$ nor $\eta$.

Note that $D(\Lambda)=\mathbb{Z}\big/3\mathbb{Z}$ and $\widehat{O}^+(\Lambda)=O^+(\Lambda)$ in the OG10 case. The Torelli Theorem for primitively polarized hyperk\"{a}hler varieties of K3${}^{[n]}$ or OG10-type with polarization of fixed divisibility $\gamma$ and degree $2d$ reads: 
\begin{theorem}[\cite{Ver13}, \cite{Mar11}*{Theorem 8.4}]
\label{thm:sec_prelim:Torelli}
Let $Y$ be an irreducible component of $\mathcal{M}_{\mathrm{K}3^{[n]}, 2d}^\gamma$ or $\mathcal{M}_{\mathrm{OG10}, 2d}^\gamma$. Then there exists an algebraic open embedding
\begin{equation}
\label{eq:sec_prelim:Torelli}Y\longrightarrow \Omega(\Lambda_h)\big/\widehat{O}^+(\Lambda,h),
\end{equation}
where $\Lambda_h$ is the orthogonal complement of $h$ in $\Lambda$ and $\widehat{O}^+(\Lambda,h)$ acts on $\Omega(\Lambda_h)$ via the restriction map
\[\widehat{O}^+(\Lambda,h)\longrightarrow O^+(\Lambda_h).\]
\end{theorem}
In particular \eqref{eq:sec_prelim:Torelli} is a birational map and the Kodaira dimension of every component $Y$ of the moduli space $\mathcal{M}_{\mathrm{K}3^{[n]}, 2d}^\gamma$ or $\mathcal{M}_{\mathrm{OG10}, 2d}^\gamma$ is given by the Kodaira dimension of the modular variety $\Omega(\Lambda_h)\big/\widehat{O}^+\left(\Lambda,h\right)$, where 
\begin{equation}
\label{eq:sec_prelim:Lambda}
\Lambda=\left\{\begin{array}{ll}
U^{\oplus 3}\oplus E_8(-1)^{\oplus 2}\oplus \langle-2(n-1)\rangle&\hbox{for K$3^{[n]}$-type,}\\
U^{\oplus 3}\oplus E_8(-1)^{\oplus 2}\oplus A_2(-1)&\hbox{for OG10-type.}
\end{array}\right.
\end{equation}
The map \eqref{eq:sec_prelim:Torelli} depends not only on the polarization type, but also on the monodromy orbit of $h\in \Lambda$. The number of connected components of $\mathcal{M}_{\Lambda,\mathfrak{h}}$ is 
equal to the number of ${\rm{Mon}}(\Lambda)$-orbits in $\mathfrak{h}$. The number of such orbits has been computed in the $\mathrm{K}3^{[n]}$ case in \cite{Apo14} (see also \cite{Son21}*{Proposition 3.4}). We recall the Mukai lattice
\begin{equation}
\label{eq:sec_prelim:Mlattice}
\widetilde{\Lambda}=U^{\oplus 4}\oplus E_8(-1)^{\oplus 2}.
\end{equation}
For $(X,H)$ fixed there is a canonical choice \cite{Mar11}*{Corollary 9.5} of
primitive embedding $i_X\colon H^2(X,\mathbb{Z})\hookrightarrow \widetilde{\Lambda}$ up to the action of $O(\widetilde{\Lambda})$. 
The number of connected components of $\mathcal{M}_{\Lambda,\mathfrak{h}}$ is the number of isometries of a rank two lattice $T$ fixing $i_X(h)$, where $T$ is the saturation in $\widetilde{\Lambda}$ of $i_X(h)\oplus i_X\left(H^2(X,\mathbb{Z})\right)^\perp$. 
This has discriminant $\frac{4d(n-1)}{\gamma^2}$ and $h^{\perp T}=\langle \ell\rangle$ with $(\ell,\ell)=-2(n-1)$. Moreover, the isomorphism class of the modular variety in the target of \eqref{eq:sec_prelim:Torelli} does not depend on the monodromy orbit, but only on the polarization type. This in particular means that any two components of $\mathcal{M}_{\Lambda,\mathfrak{h}}$ are birational.

\subsection{Sums of squares}
\label{ssec:H-K}

The representability of a positive integer by a quadratic form is a very classical problem and essential in many of our arguments. It was already known to Gauss that a positive integer $n$ can be expressed as a sum of three squares if and only if $n$ is not of the form $4^k\cdot u$ with $u\equiv -1\mod 8$, see \cite{Iwa97}*{Chapter 11}. In this paper we make use of several finer representability results due to Halter-Koch \cite{HK82}.

\begin{theorem}[\cite{HK82}*{Korollar 1}]
\label{thm:sec_prelim:HK21}
Let $n$ be a positive integer such that $n\not \equiv 0,4,7 \bmod 8$. Then $n$ can be expressed as a sum of three positive coprime squares if and only if 
\[n\notin\{1, 2, 5, 10, 13, 25, 37, 58, 85, 130, \star\},\]
where $\star$ is a number at least $5\cdot 10^{10}$ whose existence is unknown. Moreover, $n$ can be expressed as a sum of three pairwise distinct coprime squares if and only if 
\[n\notin\left\{\begin{array}{c}1, 2, 3, 6, 9, 11, 18, 19, 22, 27, 33, 43, 51,\\ 57, 67, 99, 102, 123, 163, 177, 187, 267, 627, \star\end{array}\right\}.\]
\end{theorem}

An immediate consequence of this theorem that will be used later is the following:

\begin{corollary}
\label{cor:sec_prelim:HKn_n-2}
Let $n$ be a positive even integer. Then either $n$ or $n-2$ can be expressed as the sum of three pairwise distinct coprime squares, with the exception of 
\[n\in\left\{2, 4, 6, 8, 18, 20, 22, 24, 102, 104,\star,\star+2\right\}.\]
\end{corollary}
As for quadratic forms of rank $4$, every positive integer can be expressed as the sum of four squares. 
In the case of primitive representations of odd numbers, we have:
\begin{theorem}[\cite{HK82}*{Satz 3}]
\label{thm:sec_prelim:HK2}
An odd number $n\in\mathbb{N}$ can be expressed as the sum of four pairwise distinct positive coprime squares if and only if 
\[n\notin\left\{\begin{array}{c}1,\dots,37,41,43,45,47,49,55,59,61,\\67,69,73,77,83,89,97,101,103,115, 157\end{array}\right\}.\]
\end{theorem}

Checking by hand the $n$ listed as exceptions in the above theorem, one concludes

\begin{corollary}
\label{cor:sec_prelim:sumsofsquares}
Every odd number $n$ can be expressed as a sum of four positive coprime squares, with the exception of $n\in\{1,3,5,9,11,17,29,41\}$. The numbers $9,11,17,29,41$ can be expressed as sums of three positive coprime squares.
\end{corollary}

Combining Theorem \ref{thm:sec_prelim:HK21} and Corollary \ref{cor:sec_prelim:sumsofsquares} then yields

\begin{corollary}\label{cor:sec_prelim:sumsofsquares2}
    Any integer $n\ge 3$, $n\ne 5$ such that $n\not \equiv 0 \mod 4$
can be written as a sum of four coprime squares  at most one of which is zero.
\end{corollary}

\section{Monodromy, period domains, and ramification for \texorpdfstring{K$3^{[n]}$}{K3[n]} type}
\label{sec:K3n}

Let $X$ be a hyperk\"{a}hler variety of $\mathrm{K}3^{[n]}$-type. Recall that the Beauville--Bogomolov--Fujiki lattice $H^2(X,\mathbb{Z})$ is isometric to 
\[\Lambda=U^{\oplus3}\oplus E_8(-1)^{\oplus 2}\oplus \mathbb{Z}\ell,\]
where $(\ell,\ell)=-2(n-1)$, and for a fixed primitive element of positive square $h\in \Lambda$, the monodromy group acting on $\Omega(\Lambda_h)$ is given by the restriction to $\Lambda_h$ of
\[{\rm{Mon}}^2\left(\Lambda,h\right)=\widehat{O}^+\left(\Lambda,h\right).\]
Recall moreover that ${\rm{div}}(h)=\gamma$ divides $\gcd(2d,2(n-1))$.

\begin{proposition}
\label{prop:sec_K3n:nonemp}
The moduli space $\mathcal{M}_{\mathrm{K}3^{[n]},2d}^{\gamma}$ is non-empty if and only if there exists an $a\in \mathbb{Z}$, coprime to $\gamma$, such that
\begin{equation}
\label{eq:sec_K3n:non_emp}
    \frac{2d}{\gamma}\equiv - \frac{2(n-1)}{\gamma} a^2 \pmod{2\gamma}.
\end{equation}
In particular, once $n$ and $\gamma$ are fixed, the degree $2d$ of any class $h$ in a polarization type $\mathfrak{h}$ for which the corresponding moduli space $\mathcal{M}_{\mathfrak{h}}$ is non-empty must satisfy
\begin{equation}
\label{eq:sec_K3n:d}
   d=\gamma^2 t - (n-1)a^2
\end{equation}
for some $(t,a)$ with $a$ coprime with $\gamma$.

\end{proposition}

\begin{remark}
Once a triple $(n,\gamma, a)$ is fixed, $d$ and $t$ determine each other. In view of our lattice computations, it will be easier to keep track of $t$ instead of $d$.
\end{remark}
\begin{proof}
First, assume $\mathcal{M}_{\mathrm{K}3^{[n]},2d}^{\gamma}\neq\emptyset$.
Suppose that $(X,H)\in \mathcal{M}_{\mathrm{K}3^{[n]},2d}^{\gamma}$ and that $\eta:H^2(\Lambda,\mathbb{Z})\longrightarrow \Lambda$ is a marking, and let $h=\eta\left(c_1(H)\right)$. Write $h=\alpha x-a\ell$, with $x$ primitive in $U^{\oplus 3}\oplus E_8(-1)^{\oplus 2}$. Unimodularity implies ${\rm{div}}(x)=1$ and $(\ell,\Lambda)=2(n-1)\mathbb{Z}$. We are under the assumption 
\[{\rm{div}}(h)=\gamma=\gcd(\alpha,2a(n-1))=\gcd(\alpha,2(n-1)),\]  
the last inequality following from the fact that $h$ is primitive and thus $\alpha$ and $a$ are coprime. Take $t,d',k,\alpha'$ such that $2t=(x,x)$, $2d=\gamma d'$, $2(n-1)=\gamma k$, and $\alpha=\gamma \alpha'$. 
By hypothesis, we have $2d=(\alpha x-a\ell)^2=2t\alpha^2-2(n-1)a^2$, so $d'\equiv -ka^2 \pmod{2\gamma}$.

Conversely, let $a\in \mathbb{Z}$ be an integer coprime to $\gamma$ and satisfying equation \eqref{eq:sec_K3n:non_emp}. 
Then for some integer $t$, one has $2d=2t\gamma^2-2(n-1)a^2$. Note that $h=\gamma(e+tf) -a\ell$ is a primitive element of divisibility $\gamma$ and degree $2d$. Let 
\[\Omega_{marked}=\left\{[\omega]\in \mathbb{P}\left(\Lambda\otimes\mathbb{C}\right)\left|(w,w)=0\hbox{ and }(\omega,\overline{\omega})>0\right.\right\}\]
be the period domain for marked hyperk\"{a}hler varieties. By choosing $[\omega]\in \Omega_{marked}$ very general in the hyperplane 
\[\mathbb{P}\left(h^\perp\otimes\mathbb{C}\right)\cap\Omega_{marked}\]
one concludes that the only integral point in $\omega^\perp$ is $h$. Then, surjectivity of the period map for marked hyperk\"{a}hler varieties \cite{Huy99}
implies that there is a marked hyperk\"{a}hler variety $(X,\eta)$ such that $\eta\left({\rm{Pic}}(X)\right)=h\mathbb{Z}$. Since $h^2>0$, up to a sign, the class $H=\eta^{-1}(h)\in {\rm{Pic}}(X)$ is ample of degree $2d$. In particular $(X,H)\in\mathcal{M}_{\mathrm{K}3^{[n]},2d}^{\gamma}$.
\end{proof}

\begin{remark}
When $\gamma=1$,
we can fix $a=0$, so that $t=d$. In this case, the moduli space $\mathcal{M}^\gamma_{\mathrm{K}3^{[n]},2d}$ is never empty. This is the only divisibility with this property. When $\gamma=2$, we can take $a=1$, and then $\mathcal{M}^\gamma_{\mathrm{K}3^{[n]},2d}$ is non-empty if and only if $d\equiv -(n-1)\pmod 4$
i.e. $d=4t-(n-1)$ for some $t>\frac{n-1}{4}$.
\end{remark}

\begin{lemma}
\label{lemma:sec_K3n:h}
For all $h\in \Lambda$ primitive of positive square and divisibility $\gamma$, there exist $a,t$, with $a$ coprime to $\gamma$, such that
\[\gamma(e+tf)-a\ell\]
lies in the $\widetilde{O}(\Lambda)$-orbit of $h$, where $\{e,f\}$ are the standard generators of the first copy of $U$ in $\Lambda$.
\end{lemma}
\begin{proof}
Let $(h,h)=2d$ and $h=\alpha x-a\ell$ as in the proof of Proposition~\ref{prop:sec_K3n:nonemp}.  In particular, $\gamma$ divides $\alpha$ and $a$ is coprime to $\gamma$.
Write $h'=\gamma(e+tf)- a\ell$, where $t=\frac{d+(n-1)a^2}{\gamma^2}$. Then $(h',h')=(h,h)$ and in $D(\Lambda)$, we have 
\[h_*=\frac{h}{\gamma}\equiv \frac{-a\ell}{\gamma} = -\frac{2a(n-1)}{\gamma}\ell_*=h'_*.\]
It follows from Theorem \ref{thm:sec_prelim:Eichler} that $h$ and $h'$ are in the same $\widetilde{O}(\Lambda)$-orbit.
\end{proof}

\begin{proposition}
\label{prop:sec_K3n:a's}
The connected components of $\mathcal{M}_{\mathrm{K}3^{[n]},2d}^{\gamma}$ are in one-to-one correspondence with the $a\in\{0,\dots,\lfloor\frac{\gamma}{2}\rfloor\}$ coprime with $\gamma$ satisfying \eqref{eq:sec_K3n:non_emp} (note that $a=0$ if and only if $\gamma=1$). 
\end{proposition} 
\begin{proof}
The moduli space $\mathcal{M}_{\mathrm{K}3^{[n]},2d}^{\gamma}$
is a disjoint union of moduli spaces with fixed polarization type ${\mathfrak{h}}$ as in \eqref{eq:sec_prelim:moduli}, each of which has a connected component for every $\widehat{O}(\Lambda)$-orbit in $\mathfrak{h}$ \cite{Son21}*{Proposition 3.4}.
Hence, the connected components of $\mathcal{M}_{\mathrm{K}3^{[n]},2d}^{\gamma}$
are in one-to-one correspondence with the $\widehat{O}(\Lambda)$-orbits of all primitive $h\in\Lambda$ with square $2d$ and divisibility $\gamma$.
Since the map \eqref{eq:sec_prelim:pi} is surjective, the proof of Lemma \ref{lemma:sec_K3n:h} shows that
classes $h=\gamma(e+tf)-a\ell$ and $h'=\gamma(e+t'f)-a'\ell$ with the same square $2d$ are in the same $\widehat{O}(\Lambda)$-orbit if and only if $h_*=-a\frac{2(n-1)}{\gamma}\ell_*$ equals $\pm h'_*=\pm\left(-a'\frac{2(n-1)}{\gamma}\ell_*\right)$ in $D(\Lambda)$. 
This holds if and only if $a\equiv \pm a'\mod \gamma$. 
\end{proof}
We will denote the connected component corresponding to $a$ by $\mathcal{M}_{\mathrm{K}3^{[n]},2d}^{\gamma,a}$.
\begin{remark}
Using \cite{Son21}*{Lemma~3.2}, one sees that in Proposition~\ref{prop:sec_K3n:a's}, two such $a$'s give rise to the same polarization type ${\mathfrak{h}}$ if and only if the corresponding discriminant classes $h_*$ and $h'_*$ are related via $O(D(\Lambda))$.
\end{remark}

From Lemma \ref{lemma:sec_K3n:h} we have that the orthogonal complement of $h$ in $\Lambda$ is given by
\[\Lambda_h=U^{\oplus 2}\oplus E_{8}(-1)^{\oplus 2}\oplus Q_h(-1),\]
where $Q_h(-1)$ is generated by 
\begin{equation}
\label{eq:sec3:z_1z_2}
z_1=\frac{2a(n-1)}{\gamma}f-\ell\;\;\hbox{and}\;\; z_2=e-tf.
\end{equation}
In particular, the Gram matrix of $Q_h$ is
\begin{equation}
\label{eq:sec_K3n:Q_h}
Q_h=\left(\begin{array}{cc}2(n-1)&-\frac{2a(n-1)}{\gamma}\\
-\frac{2a(n-1)}{\gamma}&2t
\end{array}\right).
\end{equation}
Moreover, by equation \eqref{eq:sec_K3n:d} we have
\begin{equation}
\label{eq:sec_K3n:disc}\left|D(\Lambda_h)\right|=\left|D(Q_h)\right|=\frac{4d(n-1)}{\gamma^2}.
\end{equation}
Recall \eqref{eq:sec_prelim:Mon} that the monodromy group in the polarized K$3^{[n]}$ case is given by $\widehat{O}^+(\Lambda,h)$.

\begin{lemma}
\label{lemma:sec_K3n:Ohat=Otilde}
Let $n\geq2$ and $h\in \Lambda$ a primitive element of positive square. Then the index of $\widetilde{O}\left(\Lambda,h\right)\subset \widehat{O}\left(\Lambda,h\right)$
is given by 
\begin{equation}
\left[\widehat{O}\left(\Lambda,h\right):\widetilde{O}\left(\Lambda,h\right)\right]=\left\{\begin{array}{lcl}
1&\hbox{ if }& n=2\hbox{ or }{\rm{div}}(h)\geq3\\
2&\hbox{otherwise.}&
\end{array}\right.
\end{equation}
The same holds for $\left[\widehat{O}^+\left(\Lambda,h\right):\widetilde{O}^+\left(\Lambda,h\right)\right]$. Moreover, when $\gamma=1,2$ and $n\geq3$, the group $\widehat{O}^+\left(\Lambda,h\right)$ is generated by $\widetilde{O}^+\left(\Lambda_h\right)$ and the restriction to $\Lambda_h$ of the reflection $\sigma_{z_1}\in O(\Lambda)$ with respect to $z_1$ defined in Equation \eqref{eq:sec3:z_1z_2}. 
\end{lemma}

\begin{proof}
Note that 
\begin{equation}\label{eq:sec_K3n:index1}
\left[\widehat{O}\left(\Lambda\right):\widetilde{O}\left(\Lambda\right)\right]=\left|\left\{\pm{\rm{Id}}_{D(\Lambda)}\right\}\right|\leq 2
\end{equation}
and equality holds when $n\geq3$. 
In particular,
\begin{equation}
\label{eq:sec_K3n:index}
\left[\widehat{O}\left(\Lambda,h\right):\widetilde{O}\left(\Lambda,h\right)\right]\leq 2.
\end{equation} 
By Lemma~\ref{lemma:sec_K3n:h} we can assume 
\[h=\gamma(e+tf)-a\ell.\] 
Suppose there exists $g\in O(\Lambda,h)$ such that $\overline{g}$=$-\Id_{D(\Lambda)}$. 
Write 
\[g(x)=u+r\ell\;\;\;\hbox{and}\;\;\;g(\ell)=u'+s\ell\]
for some integers $r$, $s$ and $u,u'\in \ell^{\perp}$. The condition 
\[h = g(h)=\gamma u+au'+(\gamma r+as)\ell\]
implies that $a(s-1)=-\gamma r$.
Since $\gamma$ and $a$ are coprime, it follows that $\gamma$ divides $s-1$.
Now $s\ell_*=\overline{g}(\ell_*)=-\ell_*$ implies that $s\equiv -1\mod 2(n-1)$, so $\gamma$ must divide $s-1=2(n-1)k-2$ for some $k\in \mathbb{Z}$. 
Since $\gamma$ divides $2(n-1)$, it follows that $\gamma$ divides $2$. This shows that $\widehat{O}(\Lambda,h)=\widetilde{O}(\Lambda,h)$ when $\gamma\geq 3$, proving the first statement of the lemma. To finish, we assume $\gamma$ is $1$ or $2$. By \eqref{eq:sec_K3n:index} it is enough to exhibit an element in $\widehat{O}\left(\Lambda,h\right)$ not acting as the identity on $D(\Lambda)$. Let $\sigma_{z_1}\colon \Lambda\otimes \mathbb{Q}\longrightarrow \Lambda\otimes\mathbb{Q}$ be the reflection (see \eqref{eq:sec_prelim:sigma_r}) with respect to $z_1$.
Concretely, $\sigma_{z_1}$ is given by
\begin{equation}
\label{eq:sec_K3n:g_1}
\begin{array}{rcl}
    e&\mapsto &e + \left(\frac{2}{\gamma}\right)^2a^2(n-1)f
    -\frac{2}{\gamma}a\ell,\\
    f&\mapsto& f,\\
    \ell&\mapsto& \frac{4}{\gamma}a(n-1)f-\ell,
    \end{array}
\end{equation}
and is the identity on $\left(U\oplus\langle\ell\rangle\right)^\perp\otimes \mathbb{Q}$. One checks \cite{Mar11}*{Proposition 9.12} that $\sigma_{z_1}\in \widehat{O}(\Lambda,h)$ and $\overline{\sigma}_{z_1}(\ell_*)=-\ell_*$ if and only if $\gamma\in\{1,2\}$. 
As $\sigma_{z_1}$ is a reflection with respect to a negative-square primitive element, it preserves orientation, see \cite{Mar11}*{Section 9}.
\end{proof}

Lemma \ref{lemma:sec_K3n:Ohat=Otilde} and the following generalization of \cite{GHS10}*{Proposition 3.12 (i)} will yield modularity of the quasi-pullback when $n=2$ or $\gamma\geq 3$ (see Corollary \ref{coro:sec_prelim:ModOtilde}).  

\begin{proposition}
\label{prop:sec_K3n:Otilde_h}
In the above setting, the restriction map $O(\Lambda,h)\to O(\Lambda_h)$ induces an isomorphism
\[\widetilde{O}\left(\Lambda,h\right)\cong \widetilde{O}(\Lambda_h).\]

\end{proposition}
\begin{proof}
First observe that $D(\Lambda_h)=D(Q_h(-1))$. Consider the following elements in $Q_h(-1)\otimes\mathbb{Q}$:
\[u=\frac{\gamma}{2d}\left(az_1+\gamma z_2\right)\;\;\;\hbox{and}\;\;\; v=\frac{\gamma^2}{2d(n-1)}\left(tz_1+a\frac{n-1}{\gamma}z_2\right).\]
From \eqref{eq:sec_K3n:d} and \eqref{eq:sec_K3n:Q_h} one checks that 
\[(u,z_1)=(v,z_2)=0\;\;\;\hbox{and}\;\;\; (v,z_1)=(u,z_2)=-1.\] 
In particular $\{u,v\}$ forms a basis of $Q_{h}(-1)^{\vee}$ and $\{u_*,v_*\}$ generates the discriminant group $D\left(Q_h(-1)\right)$. 

Now, since the inclusion $\widetilde{O}(\Lambda_h) \subset \widetilde{O}\left(\Lambda,h\right)$ holds by Lemma \ref{lemma:sec_prelim:O-tildes}, for the proposition we just need to show the reverse inclusion. 
So, let $g\in \widetilde{O}\left(\Lambda,h\right)$ acting on $\Lambda_h$ via restriction. We have to show $g|_{\Lambda_h}$ is in $\widetilde{O}(\Lambda_h)$, meaning that
\begin{equation}
\label{eq:sec_K3n:g_2}
g(u)\equiv u\;\;\;\hbox{and}\;\;\;g(v)\equiv v\mod \Lambda_h.
\end{equation}
From \eqref{eq:sec_K3n:d} it follows that $u,v\in \Lambda\otimes \mathbb{Q}$ can be expressed as
\begin{equation}
\label{eq:sec_K3n:uv}
u=\frac{\gamma}{2d}h-f\;\;\;\hbox{and}\;\;\;v=\frac{a}{2d}h-\frac{1}{2(n-1)}\ell.
\end{equation}
Since $g(h)=h$ and $\overline{g}(\ell_*)=\ell_*$ in $D(\Lambda)$, one has that $g(v)=v+w$ with $w\in \Lambda$. Moreover, we have
\[0=(h,v)=(h,g(v))=(h,w).\]
Thus, $w\in \Lambda_h$ and $g(v)\equiv v\mod \Lambda_h$. For the first congruence in \eqref{eq:sec_K3n:g_2} observe that 
\[\left[\Lambda: \Lambda_h\oplus \langle h \rangle \right]=\frac{2d}{\gamma},\]
see \cite{Huy16}*{Chapter 14, Equation~0.2} and 
\eqref{eq:sec_K3n:disc}. Moreover, $f$ generates the group 
\[H=\Lambda\big/\left(\Lambda_h\oplus \langle h \rangle\right).\]
Indeed, 
\[\frac{2d}{\gamma}f=-az_1-\gamma z_2+h\in \Lambda_h\oplus \langle h\rangle\]
is primitive in $\Lambda_h\oplus \langle h\rangle$. In particular, the order of the image $\overline{f}$ of $f$ in the quotient $H$ is exactly $\frac{2d}{\gamma}$. 
Recall \eqref{eq:sec_prelim:Pi_H} that $\pi_1:H\to D(\Lambda_h)$ and $\pi_2:H\to D(\langle h \rangle)$ are injective and every element in $H$ can be written uniquely as a sum of an element in $\pi_1(H)$ and one in $\pi_2(H)$. From the first equation in \eqref{eq:sec_K3n:uv} one observes that $\overline{f}=-u+\gamma h_*$, where 
\[\pi_1(\overline{f})=-u\in D(\Lambda_h)\;\;\;\hbox{and}\;\;\;\pi_2(\overline{f})=\gamma h_*\in D\left(\langle h \rangle\right).\]
In particular, $\pi_1(H)$ is generated by $u$. Finally, by \cite{Nik80}*{Corollary 1.5.2} (see also \cite{GHS10}*{Lemma 3.2}), the image $\overline{g}$ of $g$ in $O\left(D(\Lambda_h)\right)$ acts as the identity on $\pi_1(H)=\langle u\rangle$. This shows the first congruence in \eqref{eq:sec_K3n:g_2}, thus $\widetilde{O}\left(\Lambda,h\right)\subset \widetilde{O}\left(\Lambda_h\right)$.
\end{proof}

We next explain an interesting phenomenon known as {\textit{strange duality}}. It was first observed by Apostolov in \cite{Apo14+}*{Proposition 3.2} in the case $\gamma=1, 2$ and $d=1$,
and then generalized to $\gamma=1,2$ and $d>1$ by J.~Song, see \cite{Deb18}*{Remark 3.24}. As a consequence of Lemma~\ref{lemma:sec_K3n:Ohat=Otilde} and Proposition~\ref{prop:sec_K3n:Otilde_h}, one can generalize strange duality to all triples $(n,d,\gamma)$ such that the corresponding moduli space $\mathcal{M}^\gamma_{\mathrm{K}3^{[n]},2d}$ is non-empty. For completeness, we prove the result for all $\gamma\geq 1$.
In Section~\ref{section:K3n_div1}, we will make use of strange duality to propagate general type results.

\begin{proposition}[Strange duality]\label{prop:sec_K3n:SD}
Let $(n,d,\gamma)$ be a triple such that $\mathcal{M}^\gamma_{\mathrm{K}3^{[n]},2d}$ is non-empty (see Proposition \ref{prop:sec_K3n:nonemp}).
There is a natural bijection between the connected components of $\mathcal{M}^\gamma_{\mathrm{K}3^{[n]},2d}$ and those of $\mathcal{M}^\gamma_{K3^{[d+1]},2(n-1)}$, so that each component of $\mathcal{M}^\gamma_{\mathrm{K}3^{[n]},2d}$ is birational to the corresponding component of $\mathcal{M}^\gamma_{K3^{[d+1]},2(n-1)}$.
\end{proposition}

\begin{proof}
To prove non-emptiness of $\mathcal{M}^\gamma_{K3^{[d+1]},2(n-1)}$, consider $a\in \mathbb{Z}$ such that \eqref{eq:sec_K3n:non_emp} holds. Since $\gcd(a,\gamma)=1$, there exist $z,a'\in \mathbb{Z}$, such that $aa'=1+z\gamma$. Then \eqref{eq:sec_K3n:d} yields
\[ 
\frac{2(n-1)}{\gamma}=2\gamma\left(t(a')^2-\frac{2(n-1)}{\gamma}(2z+z^2\gamma)\right)-\frac{2d}{\gamma}(a')^2.
\]
Let $t'=t(a')^2-\frac{2(n-1)}{\gamma}(2z+z^2\gamma)$,  
then \eqref{eq:sec_K3n:d} holds for $a'$, $t'$, $d'=n-1$ and $n'=d+1$.
In fact, there is a unique $a'\in\{0,\dots,\lfloor \frac{\gamma}{2}\rfloor\}$ such that $\pm a'a\equiv 1\mod \gamma$. Hence by Proposition~\ref{prop:sec_K3n:a's}, the map sending $\mathcal{M}^{\gamma,a}_{\mathrm{K}3^{[n]},2d}$ to $\mathcal{M}^{\gamma,a'}_{K3^{[d+1]},2(n-1)}$ is a bijection between the sets of components.

Let $h\in\Lambda$ be as in Lemma~\ref{lemma:sec_K3n:h}.
Likewise, denote $\Lambda'=U^{\oplus 3}\oplus E_8(-1)^{\oplus 2}\oplus\mathbb{Z}\ell'$ with $(\ell',\ell')=-2(n'-1)=-2d$ and let $h'=\gamma(e+t'f)-a'\ell'\in\Lambda'$.
To prove that the two corresponding components of the moduli spaces are birational, by Theorem~\ref{thm:sec_prelim:Torelli} it suffices to provide an isometry $\Lambda_h\to(\Lambda')_{h'}$ that identifies the projectivization $P\widehat{O}^+(\Lambda,h)$ with $P\widehat{O}^+(\Lambda',h')$.
Let $Q_h$ be as in \eqref{eq:sec_K3n:Q_h}, with generators $\{z_1,z_2\}$. Analogously, define $Q_{h'}$ so that $\Lambda_{h'}=U^{\oplus 2}\oplus E_8(-1)^{\oplus 2}\oplus Q_{h'}(-1)$, so $Q_{h'}$ has generators $\{z_1',z_2'\}$ with respect to which the Gram matrix is 
\[
Q_{h'}=\left(\begin{array}{cc}2d&-\frac{2a'd}{\gamma}\\
-\frac{2a'd}{\gamma}&2t'
\end{array}\right).
\]
One checks that the map $Q_h\to Q_{h'}$ sending $z_1$ to $a'z_1'+\gamma z_2'$ and $z_2$ to $-(zz_1'+az_2')$ is an isomorphism of lattices.
The induced isometry $\alpha\colon\Lambda_h\to (\Lambda')_{h'}$ identifies $\widetilde{O}(\Lambda_h)$ with $\widetilde{O}((\Lambda')_{h'})$, which completes the proof when $\gamma\geq 3$.

For $\gamma=1,2$, first assume $n>2$ and $d>1$. 
For $\gamma=1$ we have $(a,a',z)=(0,0,1)$; for $\gamma=2$ we have $(a,a',z)=(1,1,0)$.
In both cases, $\alpha(z_1)\in Q_{h'}$ is orthogonal to $z_1'$; hence, we have $\sigma_{z_1'}=-\sigma_{\alpha(z_1)}$. 
It follows from Lemma~\ref{lemma:sec_K3n:Ohat=Otilde} that $\alpha$ identifies $P\widehat{O}^+(\Lambda,h)$ with $P\widehat{O}^+(\Lambda',h')$.
Finally, for $\gamma=1,2$ and $n=2$ (or analogously, $d=1$), we have $\widehat{O}^+(\Lambda,h)=\widetilde{O}^+(\Lambda_h)$ and $\widehat{O}^+(\Lambda',h')/\widetilde{O}^+((\Lambda')_{h'})=\{\pm\Id\}$,
so again, $\alpha$ identifies $P\widehat{O}^+(\Lambda,h)$ with $P\widehat{O}^+(\Lambda',h')$.
\end{proof}

Given a primitive embedding $\Lambda_h\hookrightarrow\mathrm{II}_{2,26}$, the quasi-pullback $F$ of Borcherds form $\Phi_{12}$ vanishes along the ramification divisor of the modular projection
\[\Omega(\Lambda_h)\longrightarrow \Omega(\Lambda_h)\big/\widetilde{O}\left(\Lambda_h\right)\]
if for every reflective element $r\in \Lambda_h$ for which $-\sigma_r\in \widetilde{O}(\Lambda_h)$, there is a root in $\mathrm{II}_{2,26}$ orthogonal to $\left(\Lambda_h\right)_r$ but not to all of $\Lambda_h$, see Proposition \ref{prop:sec_prelim:F=0}.

\begin{proposition}
\label{prop:sec_K3n:F=0}
Let $Q_h\hookrightarrow E_8$ be a primitive embedding such that the number of $2$-roots in the orthogonal complement is bounded by $\left|R(Q_h^\perp)\right|\leq 54$. Then the quasi-pullback $F$ of Borcherds form $\Phi_{12}$ to $\Omega(\Lambda_h)$ with respect to the induced embedding $\Omega\left(\Lambda_h\right)\hookrightarrow \Omega\left(\mathrm{II}_{2,26}\right)$ vanishes along the ramification divisor of the modular projection
\[\Omega(\Lambda_h)\longrightarrow \Omega(\Lambda_h)\big/\widetilde{O}\left(\Lambda_h\right).\]
\end{proposition}

\begin{proof}
Let $r\in \Lambda_h$ be a reflective element such that $\sigma_r$ or $-\sigma_r\in \widetilde{O}\left(\Lambda_h\right)$. 
If $\sigma_r\in \widetilde{O}(\Lambda_h)$, modularity of $F$ (see Corollary \ref{coro:sec_prelim:ModOtilde}) implies that $F(Z)=-F(Z)$ for all $Z\in \Omega^\bullet\left(\Lambda_h\right)$ such that $\left[Z\right]\in \mathcal{D}_r$. 
In particular $F$ vanishes along $\mathcal{D}_r$. We assume now that $-\sigma_r\in \widetilde{O}(\Lambda_h)$. Then by \cite{GHS07}*{Proposition 3.2} one must have 
\[D(\Lambda_h)\cong D(Q_h)\cong \left(\mathbb{Z}\big/2\mathbb{Z}\right)^m\times \mathbb{Z}\big/D\mathbb{Z},\]
where $m\in \{0,1\}$ depending on whether $D(Q_h)$ is cyclic or not and 
\[D=\frac{\mathrm{disc}(\Lambda_h)}{2^m}=\frac{\mathrm{disc}(Q_h)}{2^m}=\frac{4d(n-1)}{2^m\gamma^2}.\]
Moreover, the possibilities for $(r,r)<0$ and $\mathrm{div}(r)$ are:
\begin{itemize}
\item[(i)]$(r,r)=-2D$ and $\mathrm{div}(r)=D$,
\item[(ii)] $(r,r)=-D$ and $\mathrm{div}(r)=D$, or
\item[(iii)] $(r,r)=-D$ and $\mathrm{div}(r)=D/2$.
\end{itemize}
Therefore, 
\[\mathrm{disc}\left(\left(\Lambda_h\right)_r\right)=\left|\frac{(r,r)\cdot \mathrm{disc}\left(\Lambda_h\right)}{\mathrm{div}(r)^2}\right|\in\left\{2,4,8\right\},\]
see \cite{GHS13}*{Lemma 7.2}. Since $\mathrm{II}_{2,26}$ is unimodular, the orthogonal complement $\left(\Lambda_h\right)_r^\perp$ of $\left(\Lambda_h\right)_r$ in $\mathrm{II}_{2,26}$ has rank $7$ and discriminant $2,4$ or $8$. By \cite{CS88}*{Table 1} it contains one of the following root systems:
\[E_6, E_7, D_6, D_7,\;\;\hbox{or}\;\; A_7.\]
In particular, 
\[\left|R\left(\left(\Lambda_h\right)_r^\perp\right)\right|\geq \left|R(A_7)\right|=56\]
and $F$ must vanish with order at least one along $\mathcal{D}_r$, see Proposition \ref{prop:sec_prelim:F=0}.
\end{proof}

Finally our main proposition before going into lattice embeddings and root-counting is the following:

\begin{proposition}
\label{prop:sec_K3n:R(Q)}
Let $n,d,\gamma,$ and $a$ be positive integers satisfying the hypothesis of Proposition \ref{prop:sec_K3n:nonemp}. Assume further that
\[\frac{d(n-1)}{\gamma^2}>4.\]
If there exists a primitive embedding $Q_h\subset E_8$ such that the number of roots $\left|R(Q_h^\perp)\right|$ in $Q_h^{\perp}\subset E_8$ is at least $2$ and at most $14$ (resp. $16$), then the modular variety $$\Omega\left(\Lambda_h\right)\big/\widetilde{O}^+\left(\Lambda_h\right)$$ 
is of general type (resp.\ non-negative Kodaira dimension). In particular, 
if $n=2$ or $\gamma\geq 3$ and such an embedding exists, 
then the component of the moduli space $\mathcal{M}_{\mathrm{K}3^{[n]},2d}^{\gamma}$ corresponding to $a$ (see Proposition~\ref{prop:sec_K3n:a's}) is of general type (resp.\ non-negative Kodaira dimension). 
\end{proposition}
\begin{proof}
Let $\Omega(\Lambda_h)\subset \Omega(I_{2,26})$ be the induced embedding of period domains and $F$ the corresponding quasi-pullback of Borcherds form $\Phi_{12}\in M_{12}\left(O^+({\rm{II}}_{2,26}), {\rm{det}}\right)$. Then Corollary \ref{coro:sec_prelim:ModOtilde} implies that 
\[F\in {\rm{Mod}}_{k}\left(\widetilde{O}^+\left(\Lambda_h\right),{\rm{det}}\right),\]
with weight $k=12+\frac{\left|R(Q_h)\right|}{2},$  see  \eqref{eq_weight_quasipullback}. Moreover, $F$ is a cusp form that vanishes along the ramification divisor of the projection $\Omega(\Lambda_h)\to\Omega(\Lambda_h)\big/\widetilde{O}^+(\Lambda_h)$, see Proposition \ref{prop:sec_K3n:F=0} and \cite{GHS13}*{Corollary 8.12}. Recall that \eqref{eq:sec_K3n:disc} the discriminant group of $\Lambda_h$ has order 
\[{\rm{disc}}\left(\Lambda_h\right)=\frac{4d(n-1)}{\gamma^2}\]
and the minimal number of generators is at most two. From \cite{Ma21}*{Proposition 4.4} one concludes that if $\frac{4d(n-1)}{\gamma^2}>16$, then $\widetilde{O}^+\left(\Lambda_h\right)$ has no irregular cusps. Lemma \ref{lemma:sec_K3n:Ohat=Otilde} and Proposition \ref{prop:sec_K3n:Otilde_h} imply
\[\widehat{O}^+\left(\Lambda,h\right)\cong \widetilde{O}^+\left(\Lambda_h\right).\]
The proposition now follows from Theorems \ref{thm:sec_prelim:low-weight} and \ref{thm:sec_prelim:Torelli}.
\end{proof}

\section{Kodaira dimension for \texorpdfstring{K$3^{[n]}$}{K3[n]} type  with divisibility at least 3}

\subsection{Rank two primitive embeddings and root counting}
\label{ssec:M}

We start with a numerical lemma:
\begin{lemma}
\label{lemma:sec_K3n:even_odd}
Let $C\in \mathbb{Z}$ and $m_1, m_2, m_3$ be pairwise distinct non-negative integers. There is at most one choice of signs in front of $m_1, m_2, m_3$ with an even number of $+$ signs such that
\[C=\pm m_1\pm m_2\pm m_3.\]
The same holds for an odd number of $+$ signs.
\end{lemma}

\begin{proof}
Suppose there are two choices of signs both with an even number of $+$ signs such that $C=\pm m_1\pm m_2\pm m_3$. The the two equations must differ by two signs. Hence subtracting the two equations and dividing by $\pm 2$ yields an equation of the form
$$0=m_i\pm m_j$$
for some $1\le i <j\le 3$. However, this is impossible since $m_i$ and $m_j$ are assumed non-negative and distinct. The argument when the number of $+$ signs is odd is analogous.
\end{proof}

A similar analysis shows:

\begin{lemma}
\label{lemma:sec_K3n:no even odd}
Let $C, m_1, m_2, m_3$ be non-negative integers with $C\geq1$ and $m_1> m_2> m_3\geq 0$. Then there are at most 3 choices of signs such that
\begin{equation}
\label{eq:sec_K3n:pm_eq}
m_1\pm m_2\pm m_3\pm C\pm C=0.
\end{equation}
\end{lemma}

We now place ourselves in the following general situation. Let $Q$ be a rank $2$ even positive definite lattice having Gram matrix 
\begin{equation}
\label{eq:sec_K3n:Q}
\left(
\begin{array}{cc} M&N\\ N&P\end{array}\right)
\end{equation}
for some basis $\{z_1, z_2\}$. In the interest of applying Proposition \ref{prop:sec_K3n:R(Q)}, our goal will be to show that under certain assumptions on $M,N$ and $P$ one can primitively embed $Q$ in $E_8$ in such a way that the number roots in $E_8$ orthogonal to $Q$ is between $2$ and $14$. We start with a lemma.
We view $E_8$ as a sublattice of $(\frac{1}{2}\mathbb{Z})^{\oplus 8}$, see Section~\ref{SS:lattices}.

\begin{lemma}
\label{lemma:sec_K3n:odd-coeffs}
Let $x=\alpha_1e_1+\cdots+\alpha_8e_8$ be an integral vector in the lattice $E_8$, i.e., with $\alpha_1, \ldots, \alpha_8\in \mathbb{Z}$. Assume further that $x$ is primitive in $\mathbb{Z}^8$, i.e.\ $\alpha_1,\ldots,\alpha_8$ are coprime. If $x=mv$ for some integer $m$ with $v\in E_8\setminus D_8$, then $\alpha_i$ is odd for all $i=1,\ldots,8$.
\end{lemma}

\begin{proof}
Since $v\not\in D_8$, we have $v\in \frac{1}{2}(e_1+\ldots+e_8)+D_8$. Writing $v=\frac{1}{2}(e_1+\cdots+e_8)+\sum_{i=1}^8 y_i e_i$ with $\sum_{i=1}^8 y_i e_i\in D_8$, we have 
\[x=\sum_{i=1}^8 \left(\frac{m}{2}+my_i\right)e_i.\]
Since $\alpha_1,\ldots, \alpha_8\in \mathbb{Z}$, it follows that $m$ is even. However, since $v\not \in D_8$ and $x\in \mathbb{Z}^8$ is primitive, we must have that $4$ does not divide $m$. Hence for each $i=1,\ldots, 8$, we have $\alpha_i=\frac{m}{2}+my_i$, where $\frac{m}{2}$ is odd and $my_i$ is even. 
\end{proof}

In Proposition \ref{prop:sec_K3n:roots} below we consider solutions $(x_1,x_2,x_3)$ to an integral linear Diophantine equation of the form 
\begin{equation}
\label{eq:sec_K3n:basic diophantine}\alpha_1 X_1+\alpha_2X_2+\alpha_3X_3=K,
\end{equation}
where $\alpha_1, \alpha_2, \alpha_3 $ are fixed coprime integers exactly one of which is even. Further, we require all the $x_i$ to be odd when $K$ is even, and exactly one $x_i$ to be even when $K$ is odd. In both cases, solutions exist. If $K$ is odd, if we put $X_i = 2Y_i + 1$ then the odd solutions that we want correspond to solutions of
\[\alpha_1 Y_1+\alpha_2Y_2+\alpha_3Y_3=\frac{1}{2}(K-(\alpha_1+\alpha_2+\alpha_3)),\]
which exist because the $\alpha_i$ are coprime. If $K$ is even, we can reduce to the odd case: we may assume without loss of generality that $\alpha_1$ is odd and then, if we put $X_1 =Y_1+1,X_2=Y_2$ and $X_3=Y_3$, the solutions that we want correspond to odd solutions of
\[\alpha_1 Y_1+\alpha_2Y_2+\alpha_3Y_3=K-\alpha_1.\]

The following proposition constructs 
the embeddings of the lattice \eqref{eq:sec_K3n:Q} into $E_8$ needed to apply Proposition \ref{prop:sec_K3n:R(Q)} by making essential use of Theorem~\ref{thm:sec_prelim:HK21}. The given embeddings will depend on the residues of $M$ and $P$ modulo $4$.

\begin{proposition}
\label{prop:sec_K3n:roots}
Let $Q$ be an even lattice of rank $2$ with basis and Gram matrix given by \eqref{eq:sec_K3n:Q}. Assume further that $M>8$ and $M\ne 20,24$.
We fix
\[\Omega=
\begin{cases}
0 &\mbox{ if } M\equiv 2 \mod 4,\; M\not\in \{18,22,102,\star\},\\
1 &\mbox{ if } M\equiv 0 \mod 4,\; M\not \in \{104, \star+2\},\\
2 &\mbox{ if } M\in \{18,22,102,\star\},\\
3 & \mbox{ if } M \in \{104, \star+2\}.
\end{cases}
\]
and write $M-2\Omega^2$ as a sum of three pairwise distinct coprime squares (Theorem~\ref{thm:sec_prelim:HK21}) 
\[M-2\Omega^2=\alpha_1^2+\alpha_2^2+\alpha_3^2\]
with $\alpha_1>\alpha_2>\alpha_3\ge 0$. Let $(x_1,x_2,x_3)$ be a solution  to the equation
\[\alpha_1 X_1 +\alpha_2  X_2+\alpha_3 X_3=N-2\Omega\Theta,\]
where
\[\Theta=
\begin{cases}
0 &\mbox{ if } 4 | P,\\
1 &\mbox {otherwise}
\end{cases}\]
and we require all $x_i$'s to be odd when $N$ is even and exactly one $x_i$ to be even when $N$ is odd; when $\alpha_3=0$, we set $x_3=1$. Let 
$S=P-(x_1^2+x_2^2+x_3^2)-2\Theta^2$.
If $S>5$ and 
\[S\ne 
\begin{cases}
8 &\mbox{ if } \Omega=\Theta=0\\
6 &\mbox{ if } \Omega=0, \Theta=1\\
6,9,18,22,33,57,102,177,\star &\mbox{ otherwise},
\end{cases}
\]
there is a primitive embedding $Q\hookrightarrow E_8$ such that the number of roots in the orthogonal complement $Q^{\perp}$ satisfies $$2\leq\left|R(Q)\right|\leq 14.$$
\end{proposition}

\begin{proof}
Begin by noting that
 \[S=P-(x_1^2+x_2^2+x_3^3)-2\Theta^2\]
 must satisfy $S\equiv 1,2,5,6 \mod 8$. We separate the proof into several cases depending on the values of $S$. We are under the assumption $S>5$.

\vspace{0.3cm}
\textbf{CASE 1: } Assume $S\not \in \begin{cases} \{10, 13, 25, 37, 58, 85, 130, \star\} &\mbox{ if } \Omega=\Theta=0\\
 \{6,9,18, 22, 33, 57, 102, 177, \star\} &\mbox{ otherwise}.
 \end{cases}$
 
\vspace{0.1cm}
Let $x_6\ge x_7\ge x_8$ be non-negative integers such that
\[S=x_6^2+x_7^2+x_8^2,\]
where if $\Omega=\Theta=0$, we ask that $x_6, x_7, x_8$ are positive and coprime, and if $\Theta, \Omega$ are not both zero, then we ask the $x_i$ to be pairwise distinct and coprime, see Theorem \ref{thm:sec_prelim:HK21}. 
Consider the embedding $Q\hookrightarrow E_8$ given by $z_1\mapsto v_1$, $z_2\mapsto v_2$ where 
\[
\begin{aligned}
v_1&=\alpha_1e_1+\alpha_2e_2+\alpha_3e_3 + \Omega e_4+\Omega e_5\\
v_2&= x_1e_1 + x_2 e_2 + x_3e_3 + \Theta e_4+\Theta e_5 + x_6 e_6+x_7e_7+x_8e_8.
\end{aligned}
\]
Note that $(v_1,v_1)=M$, $(v_1,v_2)=N$, and $(v_2,v_2)=P$. Moreover, as $M$ and $P$ are even, the sums of the coefficients of $v_1$, resp. $v_2$, are even, hence $v_1,v_2$ are primitive elements in $D_8\subset \mathbb{Z}^8$. In order to check primitivity of the embedding, assume there exist coprime integers $r,s$ such that
\begin{equation}
\label{eq:sec_K3n:rv1+sv2}
rv_1+sv_2=mv
\end{equation}
is a multiple of an element $v\in E_8$. If $v\in D_8$, then $m$ divides $sx_6$, $sx_7$, and $sx_8$. Since $x_6,x_7,x_8$ are coprime, it follows that $m$ divides $s$. 
Similarly, $m$ must divide $r\alpha_i+sx_i$ for $i=1,2,3$, and as $\alpha_1$, $\alpha_2$, $\alpha_3$ are coprime, it follows that $m$ divides $r$. 
But $r$ and $s$ are assumed to be coprime, thus $m=\pm 1$. If $v\in E_8\setminus D_8$, since neither $M$ nor $S$ are congruent to $3$ modulo $8$, we know that at least one $\alpha_i$ and one of $x_6, x_7, x_8$ must be even. Primitivity then follows from Lemma \ref{lemma:sec_K3n:odd-coeffs}.

We now count the roots in $Q^\perp$. If $\Omega=\Theta=0$, one observes that the integral roots are
\begin{enumerate}
\item $\pm e_4\pm e_5$
\item\label{en1:sec_K3n:2} $\pm (e_3-e_j)$ for $j\in \{6,7,8\}$ if $\alpha_3=0$ and $1=x_j$

\item\label{en1:sec_K3n:3} $\pm (e_i-e_j)$ for $6\le i<j\le 8$ if $x_i=x_j$

\end{enumerate}
Note that since $S\ne 3$, there are at most $4$ roots of type \eqref{en1:sec_K3n:2} and since $x_6,x_7,x_8$ are coprime, there are at most $2$ roots of type \eqref{en1:sec_K3n:3}. Hence there are between $4$ and $10$ integral roots in $Q^\perp\subset E_8$. Moreover, if $\alpha_3\ne 0$, there are at most $6$ integral roots.

If $\Omega$ and $\Theta$ are not both $0$, the integral roots in $Q^\perp$ are
\begin{enumerate}
\item $\pm (e_4-e_5)$
\item \label{en2:sec_K3n:2}$\pm (e_k-e_i)$ for $k\in \{1,2,3\}$ and $i\in \{4,5\}$, if $\alpha_k=\Omega$ and $x_k=\Theta$ 
\item\label{en2:sec_K3n:3} $\pm (e_3-e_j)$ for $j\in\{6,7,8\}$ if $\alpha_3=0$ and $x_j=1$
\item\label{en2:sec_K3n:4} $\pm (e_i-e_j)$ for $i=4,5$ and $j\in \{6,7,8\}$ if $\Omega=0$ and $\Theta=x_j$
\end{enumerate}

Note that since $\alpha_1, \alpha_2, \alpha_3$ are distinct, there are at most $4$ roots of type \eqref{en2:sec_K3n:2}. 
Since $x_6,x_7,x_8$ are distinct, there are at most $2$ roots of type \eqref{en2:sec_K3n:3} and $4$ of type \eqref{en2:sec_K3n:4}. 
Therefore, in this case there are between $2$ and $12$ integral roots in $Q^\perp$. 
Moreover, if $\alpha_3\ne 0$ the number of integral roots is at most $6$ (since in this case \eqref{en2:sec_K3n:2} and \eqref{en2:sec_K3n:4} cannot happen simultaneously) and if $\Omega\ne 0$ the number of integral roots is at most $8$.

Next, we count the fractional roots in $Q^{\perp}$. Suppose that
\begin{equation}
\label{eq:sec_K3n:w}
w=\frac{1}{2}\left(\pm e_1\pm e_2\pm e_3\pm e_4\pm e_5\pm e_6\pm e_7\pm e_8\right),
\end{equation}
is a fractional root in $Q^\perp$. So, the number of $+$ signs in \eqref{eq:sec_K3n:w} is even. Since $(w,v_1)=0$, we must have a choice of signs such that
\begin{equation}
\label{eq:sec_K3n:alpha signs}
\alpha_1\pm \alpha_2\pm \alpha_3\pm \Omega \pm \Omega=0.
\end{equation}
Let $C$ be the sum $x_1\pm x_2\pm x_3\pm \Theta \pm \Theta $ for the same choice of signs as in \eqref{eq:sec_K3n:alpha signs}. Then as $(w,v_2)=0$, we have 
\begin{equation}
\label{eq:sec_K3n:R signs}
C=\pm x_6\pm x_7\pm x_8,
\end{equation}
where the number of $+$ signs in \eqref{eq:sec_K3n:R signs} and \eqref{eq:sec_K3n:alpha signs} have opposite parities. If $\Omega=\Theta=0$, then $x_6, x_7, x_8$ are all positive and not necessarily distinct. 
Begin by noting that since $\alpha_1>\alpha_2>\alpha_3$ and $\Omega=0$, the only way for an equation of the form \eqref{eq:sec_K3n:alpha signs} to hold is if $\alpha_1=\alpha_2+\alpha_3$, 
so the first three signs in $w$ are $\pm(e_1-e_2-e_3)$. 
Moreover, since $\alpha_1,\alpha_2,\alpha_3$ are distinct, this can only happen if $\alpha_3\ne0$. Hence when $\Omega=\Theta=0$, if there are fractional roots, there are at most $6$ integral roots.

Note moreover that in the case $\Omega=\Theta=0$, if there are 
 two choices of signs differing by one sign such that \eqref{eq:sec_K3n:R signs} holds, then subtracting the two equations will yield $0=\pm 2x_i$ for some $i\in \{6,7,8\}$. But this is impossible since $x_6, x_7, x_8$ are assumed to be non-zero. 
 
If there  are two choices of signs differing by two signs such that 
\eqref{eq:sec_K3n:R signs} holds, then adding the two equations yields $2x_i=\pm2C$ for some $i\in \{6,7,8\}$ and subtracting the two equations yields $2x_j\pm 2x_k=0$ for $j,k\in \{6,7,8\}$ and $j,k\ne i$. 
Since $x_6, x_7$, and $x_8$ are coprime and positive, we then know that  $x_j$ and $x_k$ cannot also be equal to $\pm C$ and moreover $C\ne 0$. 
It follows that the choices of signs such that $C=\pm x_6\pm x_7\pm x_8$ are precisely $C=C \pm (x-x).$ Hence possible fractional roots in this case are of the form
\[w=\pm \frac{1}{2}\left( e_1-e_2-e_3\pm e_4\pm e_5- (e_6\pm (e_7-e_8))\right),\]
where the number of $+$ signs is even. Possible fractional roots in this case satisfy
\begin{align*}
\pm w&=\frac{1}{2}\left( e_1-e_2-e_3\pm e_4\pm e_5- (e_6\pm (e_7-e_8))\right)\\
&=\frac{1}{2}\left( e_1-e_2-e_3\pm (e_4+ e_5)- e_6\pm (e_7-e_8)\right),
\end{align*}
as the number of $+$ signs on the right hand side must be even. In particular, there are at most $8$ fractional roots.

The only remaining way to have two choices of signs such that Equation \eqref{eq:sec_K3n:R signs} holds is if the two choices differ by three signs. 
In this case $C=0$, so by the previous analysis we cannot have any additional choices of signs such that $C=\pm x_6\pm x_7\pm x_8$. These two choices will each correspond to at most $4$ fractional roots. 
We have shown that when $\Omega=\Theta=0$ there are at most $6$ integral and $8$ fractional roots, in particular $$2\leq \left|R(Q)\right|\leq 14.$$ 
Now consider the case when $\Omega$ and $\Theta$ are not both $0$. Then $x_6, x_7, x_8$ are pairwise distinct and by Lemma \ref{lemma:sec_K3n:even_odd}, there is at most one choice of signs such that \eqref{eq:sec_K3n:R signs} holds. 
It follows that for each choice of signs such that \eqref{eq:sec_K3n:alpha signs} holds, there are at most $2$ fractional roots in $Q^\perp$ (of the form $\pm w$). 
However by Lemma \ref{lemma:sec_K3n:no even odd}, there are at most $3$ choices of signs such that Equation \eqref{eq:sec_K3n:alpha signs} holds. Hence there are at most $6$ fractional roots in $Q^\perp$.  
 
Note moreover that when $\Omega\ne 0$, 
there are at most $8$ integral roots in $Q^\perp$, 
so the total number of roots in $Q^\perp$ is between $2$ and $14$. 
If $\Omega=0$ and $\Theta \ne 0$, then as $\alpha_1>\alpha_2>\alpha_3$ the only way to have Equation \eqref{eq:sec_K3n:alpha signs} hold is if $\alpha_1=\alpha_2+\alpha_3$, meaning in particular that $\alpha_3\ne 0$. 
It follows that if $\Omega=0$, $\Theta\ne 0$, and there are fractional roots, then there are most $6$ integral roots. Hence the total number of roots is between $2$ and $12$. 
 
\vspace{0.3cm}
\textbf{CASE 2: }Assume $\Omega=0$ and $S\in \begin{cases} \{13, 25, 37, 58, 85, 130, \star\} &\mbox{ if } \Theta=0\\
 \{33, 57, 102, 177, \star\} &\mbox{ if } \Theta=1.
 \end{cases}$

\vspace{0.1cm}

We fix $\Xi=2$ if $\Theta=0$ and $\Xi=3$ if $\Theta=1$. Observe that $R:=S+2\Theta-2\Xi^2$ can be written as a sum of three different coprime squares, $R=y_6^2+y_7^2+y_8^2$, where without loss of generality $y_6>y_7>y_8$. 
Note moreover, that $R\equiv S \mod 8$, thus in particular $R$ is not congruent to $3$ modulo $8$. Hence one of $y_6, y_7, y_8$ must be even. In this case, consider the embedding $Q^\perp\subset E_8$ given by $z_1\mapsto v_1$ and $z_2\mapsto v_2$, where
\begin{align*}
v_1&=\alpha_1e_1+\alpha_2e_2+\alpha_3e_3\\
v_2 &= x_1e_1 + x_2 e_2 + x_3e_3 + \Xi e_4+\Xi e_5 + y_6 e_6+y_7e_7+y_8e_8.
\end{align*}
As with the previous case, note that $v_i$ are in $D_8$. From 
\[{\rm{gcd}}(\alpha_1,\alpha_2,\alpha_3)={\rm{gcd}}(y_6,y_7,y_8)=1\]
together with the fact that one of the $\alpha_j$ and one of the $y_j$ must be even,
we see that the embedding is primitive, see Lemma~\ref{lemma:sec_K3n:odd-coeffs}. 
    
The integral roots for the above embedding are the same as those listed for the previous embedding in the case that $\Omega$ and $\Theta$ are not both $0$. 
Thus there are between $2$ and $12$ integral roots in $Q^\perp$, with at most $6$ integral roots when $\alpha_3\ne 0$. The analysis of the fractional roots is also the same as in the previous case when $\Omega$ and $\Theta$ are not both $0$. 
In particular, because $\Omega=0$, in order to have fractional roots we must have $\alpha_3\ne 0$ and there are at most 6 fractional roots. Since in this case, there are at most $6$ integral roots, it follows that the total number of roots is between 2 and 12. 

For the next cases we give the embedding. The primitivity and root counting arguments are analogous to the first two cases. 

\vspace{0.3cm}
\textbf{CASE 3: }Assume $\Omega=0$, $\Theta=0$, and $S=10$.
\vspace{0.1cm}

In  this case the embedding $Q^\perp\subset E_8$ is given by
 \begin{align*}
    z_1&\mapsto v_1=\alpha_1e_1+\alpha_2e_2+\alpha_3e_3\\
    z_2 &\mapsto v_2 = x_1e_1 + x_2 e_2 + x_3e_3 +  e_4+e_5 + 2 e_6+2e_7.
    \end{align*}
  
\vspace{0.2cm}
\textbf{CASE 4: }Assume $\Omega=0$, $\Theta=1$, and $S=9$.
\vspace{0.1cm}
   
In this case the embedding $Q^\perp\subset E_8$ is given by
 \begin{align*}
    z_1&\mapsto v_1=\alpha_1e_1+\alpha_2e_2+\alpha_3e_3\\
    z_2 &\mapsto v_2 = x_1e_1 + x_2 e_2 + x_3e_3  + 3 e_6+e_7+e_8.
    \end{align*}

\vspace{0.2cm}
\textbf{CASE 5: }Assume $\Omega=0$, $\Theta=1$, and $S=18$.
 \vspace{0.1cm}

In this case the embedding $Q^\perp\subset E_8$ is given by
 \begin{align*}
    z_1&\mapsto v_1=\alpha_1e_1+\alpha_2e_2+\alpha_3e_3\\
    z_2 &\mapsto v_2 = x_1 e_1+x_2 e_2+x_3 e_3 + e_4+e_5+3e_6+3 e_7.
    \end{align*}

\vspace{0.2cm}
\textbf{CASE 6: }Assume $\Omega=0$, $\Theta=1$, and $S=22$.
 \vspace{0.1cm}
 
In this case the embedding $Q^\perp\subset E_8$ is given by
 \begin{align*}
    z_1&\mapsto v_1=\alpha_1e_1+\alpha_2e_2+\alpha_3e_3\\
    z_2 &\mapsto v_2 =x_1 e_1+x_2 e_2+x_3 e_3 + e_4+e_5+3e_6+3 e_7+2e_8.
    \end{align*}
We have exhausted all cases of the proposition.
\end{proof}


\subsection{General type results for divisibility at least three.}

Combining Propositions \ref{prop:sec_K3n:R(Q)}  and \ref{prop:sec_K3n:roots} then yields the following result. 

\begin{theorem}
\label{thm: diophantine bounds version}
Let $n,d,\gamma,a$ be positive integers satisfying the hypothesis of Proposition \ref{prop:sec_K3n:nonemp} such that $n\ge 6$, $n\ne 11,13$, $\gamma\ge 3$ and 
$\frac{d(n-1)}{\gamma^2}>4.$
Let $M=2(n-1)$, $N=-a\frac{2(n-1)}{\gamma}$, and $P=2t$. Then, in the notation of 
 Proposition  \ref{prop:sec_K3n:roots}, if $S>5$ and 
\[S\ne 
\begin{cases}
8 &\mbox{ if } \Omega=\Theta=0\\
6 &\mbox{ if } \Omega=0, \Theta=1\\
6,9,18,22,33,57,102,177,\star &\mbox{ otherwise},
\end{cases}
\]
the component of $\mathcal{M}_{\mathrm{K}3^{[n]},2d}^\gamma$ corresponding to $a$ (see Proposition~\ref{prop:sec_K3n:a's}) is of general type.
\end{theorem}

We spell out an example below to show that the above result can give different bounds on $d$ for different components of the moduli space $\mathcal{M}_{\mathrm{K}3^{[n]},2d}^\gamma$.

\begin{example}\label{ex: components different}
Consider the moduli space $\mathcal{M}_{K3^{[26]},2d}^5$ given by taking $n=26$ and $\gamma=5$. Then the number $M$ in Proposition~\ref{prop:sec_K3n:roots} is $2(n-1)=50$, and $\Omega=0$. Note that there are exactly two ways to write $2(n-1)=50$ as a sum of three distinct coprime squares
\[50=7^2+1^2+0^2=5^2+4^2+3^2.\]
We can take $(\alpha_1,\alpha_2,\alpha_3)$ to be $(7,1,0)$ or $(5,4,3)$. The numbers $N$ and $P$ are respectively $-10a$ and $2t=\frac{2d}{25}+2a^2$, where $a\in\{1,2\}$ (see \eqref{eq:sec_K3n:d}).
If we consider the component $\mathcal{M}_{K3^{[26]},2d}^{5,1}$ of $\mathcal{M}_{K3^{[26]},2d}^5$ corresponding to $a=1$, we then consider odd solutions to the equations, for the two possibilities of $(\alpha_1,\alpha_2,\alpha_3)$ respectively:
\[
7x_1+x_2=-10 \hspace{3mm} \text{ and } \hspace{3mm}
5x_1+4x_2+3x_3=-10.
\]
The minimal norm solutions are then respectively
\[(x_1, x_2, x_3)=(-1,-3,1)\;\;\; \hbox{and}\;\;\;(x_1, x_2, x_3)=(1,-3,-1),\]
both of which have norm $11$. 
The number $\Theta$ is 0 when $2t\equiv 0\mod 4$ and is 1 otherwise, and we find $S=2(t-\Theta^2)-11$. The bounds given in Proposition~\ref{prop:sec_K3n:roots} imply, via Proposition~\ref{prop:sec_K3n:R(Q)} that $\mathcal{M}_{K3^{[26]},2d}^{5,1}$ is of general type when $t\ge 10$, i.e., $d\ge 225$. By contrast, if we consider the component $\mathcal{M}_{K3^{[26]},2d}^{5,2}$, we take odd solutions to the equations, respectively
\[
7x_1+x_2=-20 \hspace{3mm} \text{ and } \hspace{3mm}
5x_1+4x_2+3x_3=-20.
\]
The minimal norm solutions are then respectively
\[(x_1, x_2, x_3)=(-3,-1,1)\;\;\;\hbox{and}\;\;\;(x_1, x_2, x_3)=(-1,-3,-1),\] 
both of which again have norm $11$. Hence, $\mathcal{M}_{K3^{[26]},2d}^{5,2}$ is of general type when $t\ge 10$, i.e., $d\geq 150$.

The above illustrates for instance that the moduli space $\mathcal{M}_{K3^{[26]},300}^5$ has one connected component (corresponding to $a=2$) which our results show is of general type and one connected component (corresponding to $a=1$) for which our results do not yield a statement about the Kodaira dimension.

\end{example}

Given particular $(n,\gamma, a)$, Theorem \ref{thm: diophantine bounds version} allows one to compute a lower bound on $d$ after which the given connected component of $\mathcal{M}_{\mathrm{K}3^{[n]},2d}^\gamma$ is of general type. 
However, since obtaining this bound requires computing $\alpha_1, \alpha_2, \alpha_3$ as well as $x_1, x_2, x_3$, Theorem \ref{thm: diophantine bounds version} is not so useful for understanding uniformly when $\mathcal{M}_{\mathrm{K}3^{[n]},2d}^\gamma$ is of general type. For this, we will compute a much coarser bound obtained by using the following:

\begin{lemma}[\cite{BFRT89}*{Main Theorem}]
\label{lemma:sec_K3n:BFRT}
 Let $\alpha_1X_1+\alpha_2X_2+\cdots+\alpha_NX_N=b$ be a linear equation with integer coefficients. If it admits an integral solution, then it has a solution 
 $(x_1, \ldots, x_N)$ such that $\max|x_i|\le \max\{|a_1|, \ldots,|a_N|, |b|\} $.
 \end{lemma}
 
 \begin{theorem}
 \label{thm:sec_K3n:div3_2}
 Let $(n,d,\gamma)$ be a triple such that the moduli space $\mathcal{M}_{\mathrm{K}3^{[n]},2d}^\gamma$ is non-empty (see Proposition \ref{prop:sec_K3n:nonemp}). We assume further that $n\ge 6$, $n\ne 11,13$, and $\gamma\geq 3$. 
Then every component of $\mathcal{M}_{\mathrm{K}3^{[n]},2d}^\gamma$ is of general type provided 
\[d\geq 6\gamma^2\left(n+3+\sqrt{2(n-1)}\right)^2,\]
except for one possible value of $d\ge 5\cdot 10^{10}$ for each $n$ which is odd or in the set $\{10,12,52,\frac{\star}{2}+1\}$.

 \end{theorem}
 
 \begin{proof}
 Let $Q_h$ be the lattice $$Q_h=\left(\begin{array}{cc}2(n-1)& -a\frac{2(n-1)}{\gamma}\\ -a\frac{2(n-1)}{\gamma}& 2t \end{array}\right)$$
where $a$ and $t$ satisfy 
$$\frac{2d}{\gamma}\equiv -\frac{2(n-1)}{\gamma}a^2\mod 2\gamma\;\;\;\hbox{and}\;\;\;d=\gamma^2t-(n-1)a^2.$$

Observe that if $d\geq 6\gamma^2\left(n+3+\sqrt{2(n-1)}\right)^2$, then certainly $\frac{d(n-1)}{\gamma^2}>4$ and so by Proposition \ref{prop:sec_K3n:R(Q)}, we need only verify that the hypotheses of Proposition \ref{prop:sec_K3n:roots} are satisfied. 

Let $M=2(n-1)$, $N=-a\frac{2(n-1)}{\gamma}$, and $P=2t$. Then fix $\Omega$, $\alpha_1, \alpha_2, \alpha_3$, $(x_1, x_2, x_3)$, $\Theta$, and $S$ as in the statement of Proposition \ref{prop:sec_K3n:roots}. Then since $M-2\Omega^2=\alpha_1^2+\alpha_2^2+\alpha_3^2$, for all $i=1,2,3$ we have
\begin{equation}\label{eq:sec_K3n:alpha bound}
0\le  \alpha_i< \sqrt{M-2\Omega^2}.
 \end{equation}
Additionally, the fact that $0<a<\gamma$ implies
 \begin{equation}
 \label{eq:sec_K3n:MN bound}
0<-N< M.
 \end{equation}
 
 Let us first consider the case $N$ even. Recall that odd solutions $(x_1, x_2, x_3)$ to the Diophantine equation
 \[\alpha_1X_1+\alpha_2 X_2+\alpha_3 X_3=N-2\Omega\Theta\]
 are (by setting $X_i=2Y_i+1$) equivalent to integral solutions of 
 \begin{equation}
 \label{eq:sec_K3n:basic diophantine2} \alpha_1 Y_1+\alpha_2Y_2+\alpha_3Y_3=\frac{1}{2}(N-2\Omega\Theta-(\alpha_1+\alpha_2+\alpha_3)).
 \end{equation}
 Lemma \ref{lemma:sec_K3n:BFRT} implies the existence of a solution $(y_1, y_2, y_3)$ to \eqref{eq:sec_K3n:basic diophantine2} satisfying
 \[\max|y_i|\le \max\{\alpha_1, \alpha_2, \alpha_3, \frac{1}{2}(-N+2\Theta\Omega+\alpha_1+\alpha_2+\alpha_3)\}.\]
Note also
\begin{equation}\label{eq:sec_K3n:cauchy-schwartz} \alpha_1+\alpha_2+\alpha_3\le \sqrt{\alpha_1^2+\alpha_2^2+\alpha_3^2}=\sqrt{M-2\Omega^2}.\end{equation}
Equations \eqref{eq:sec_K3n:MN bound} and \eqref{eq:sec_K3n:cauchy-schwartz} give us
\[
\frac{1}{2}(-N+2\Omega\Theta+(\alpha_1+\alpha_2+\alpha_3))<\frac{1}{2}(M+2\Omega\Theta+\sqrt{M-2\Omega^2}).
\]
Moreover note for all integers $M>0$ we have
\[0<\sqrt{M-2\Omega^2}<\frac{1}{2}(M+2\Omega\Theta+\sqrt{M-2\Omega^2}).\]
Therefore, when $N$ is even
$$\max|y_i|< \frac{1}{2}(M+2\Omega\Theta+\sqrt{M-2\Omega^2})\;\;\;\hbox{and}\;\;\;\max |x_i|<M+2\Omega\Theta+1+\sqrt{M-2\Omega^2}.$$

We now turn to the case that $N$ is odd. As discussed before Proposition \ref{prop:sec_K3n:roots}, we may produce a solution $(x_1, x_2, x_3)$ with exactly one $x_i$ even to the equation
 \[\alpha_1X_1+\alpha_2 X_2+\alpha_3 X_3=N-2\Omega\Theta\]
 where $\alpha_i$ is odd. By setting $X_i=2(Y_i+1)$ and $X_j=2Y_j+1$ for $i\ne j$, such solution $(x_1, x_2, x_3)$ can be obtained from an integral solution $(y_1, y_2, y_3)$ to 
  \[\alpha_1Y_1+\alpha_2 Y_2+\alpha_3 Y_3=\frac{1}{2}(N-2\Omega\Theta-\alpha_i-(\alpha_1+\alpha_2+\alpha_3)).\]
 Lemma \ref{lemma:sec_K3n:BFRT} implies that such a solution $(y_1, y_2, y_3)$ satisfies
 \[\max|y_i|\le \max\bigl\{\alpha_1, \alpha_2, \alpha_3, \frac{1}{2}(-N+2\Omega\Theta+\alpha_i+(\alpha_1+\alpha_2+\alpha_3))\bigr\}.\]
 Hence using the bounds \eqref{eq:sec_K3n:alpha bound}, \eqref{eq:sec_K3n:MN bound}, and \eqref{eq:sec_K3n:cauchy-schwartz}, when $N$ is odd we get
 \[\max|y_i|< \frac{1}{2}(M+2\Omega\Theta+2\sqrt{M-2\Omega^2})\;\;\;\hbox{and}\;\;\;\max|x_i|< M+2\Omega\Theta+2+2\sqrt{M-2\Omega^2}.
 \]
 Comparing the bounds on $|x_i|$ in the $N$ even and $N$ odd cases, we see that for any $N$
\begin{equation}
\label{eq: norm bound N}
x_1^2+x_2^2+x_3^2<3(M+2\Omega\Theta+2+2\sqrt{M-2\Omega^2})^2.
\end{equation}
It follows that 
\begin{equation}\label{eq: S bound}S>P-3(M+2\Omega\Theta+2+2\sqrt{M-2\Omega^2})^2-2\Theta^2.\end{equation}
However, the assumption $d\geq 6\gamma^2\left(n+3+\sqrt{2(n-1)}\right)^2$ can be rewritten as 
\[
P\gamma^2-Ma^2\ge 3\gamma^2(M+8+\sqrt{M})^2.
\]
In particular, we have
\begin{equation}\label{eq: P bound} P \ge 3(M+8+\sqrt{M})^2.
\end{equation}
Now a verification shows that for all $M\ge 10$, $\Omega\in \{0,1,2,3\}$, and $\Theta\in \{0,1\}$ we have
\[3(M+8+\sqrt{M})^2-3(M+2\Omega\Theta+2+2\sqrt{M-2\Omega^2})^2-2\Theta^2>600.\]
In particular, \eqref{eq: S bound} implies that $S>600$, which by Proposition~\ref{prop:sec_K3n:roots} finishes the proof. 
\end{proof}

\section{Monodromy and pullbacks for \texorpdfstring{K$3^{[n]}$}{K3[n]} type with low divisibility.}
\label{s:K3n+}

In this section we treat the case $\gamma\in \{1,2\}$. The main difficulty 
lies in the fact that the monodromy group $\widehat{O}^+(\Lambda,h)$ and the stable orthogonal group $\widetilde{O}^+(\Lambda_h)$ are no longer isomorphic. 
For a primitive embedding $\Lambda_h\hookrightarrow {\rm{II}}_{2,26}$ the quasi-pullback $F(Z)$ defined in \eqref{eq:sec_prelim:qpullback} is not necessarily modular for $\widehat{O}^+(\Lambda,h)$, so we have to choose the embedding carefully to ensure that $\widehat{O}^+\left(\Lambda,h\right)$ satisfies the hypothesis of Lemma \ref{lemma:sec_premil:F}. 
Recall that the restriction $O(\Lambda,h)\rightarrow O(\Lambda_h)$ induces an isomorphism $\widetilde{O}^+(\Lambda_h)\cong \widetilde{O}^+(\Lambda,h)$, see Proposition \ref{prop:sec_K3n:Otilde_h}. When $\gamma=1,2$ and $n\geq3$, the group $\widehat{O}^+\left(\Lambda,h\right)$ is generated (see Lemma \ref{lemma:sec_K3n:Ohat=Otilde}) 
by $\widetilde{O}^+\left(\Lambda_h\right)$ and $\sigma_v\in O(\Lambda)$ where 
\[v=\frac{2a(n-1)}{\gamma}f-\ell\in \Lambda.\]

In order to treat the case $\gamma=1,2$, we need to ensure that the reflection $\sigma_v\in \widehat{O}^+\left(\Lambda, h\right)$ of Lemma \ref{lemma:sec_K3n:Ohat=Otilde} extends to ${\rm{II}}_{2,26}$ after choosing our embedding $Q_h\hookrightarrow E_8$. Let $\sigma\in O(E_8)$ be an involution on $E_8$ and $V_+, V_-$ the invariant and anti-invariant lattices respectively. Note that both $V_-$ and $V_+$ are primitive, $V_-= (V_+)^{\perp}$, and the sublattice
\[V_-\oplus V_+\subset E_8\]
has finite index. Moreover, unimodularity of $E_8$ implies that there is an isometry
\[\phi: D(V_+)\longrightarrow D(V_-)\]
(up to sign)
and the finite index extension $V_-\oplus V_+\subset E_8$ corresponds (see Section \ref{SS:lattices}) to the isotropic subgroup 
\[H_\phi=\{ (x,\phi(x))\in D(V_+\oplus V_-)\mid x\in D(V_+)\}\subset D(V_+\oplus V_-).\]
Moreover, an isometry of the form
\[g=g_-\oplus g_+\in O(V_-\oplus V_+)\]
extends to an isometry on $E_8$ if $\overline{g}$ fixes $H_\phi$, which is equivalent to
\begin{equation}
\label{eq:sec_K3n+:g+g-}
\phi\circ \overline{g_+}=\overline{g_-}\circ\phi,
\end{equation}
see \cite{Nik80}*{Corollary 1.5.2}. Recall that a lattice $L$ is called $n$-{\textit{elementary}} if $D(L)$ is isomorphic to a direct sum of cyclic groups of order $n$. An immediate consequence is then the following: 
\begin{lemma}
\label{lemma:sec_K3n+:2-elementary}
Let $Q\subset E_8$ be a primitive sublattice and $\sigma_r\in O(Q)$ a reflection with respect to an element $r\in Q$. Let $M$ be the orthogonal complement of $r$ in $Q$. The following are equivalent:
\begin{enumerate}[(1)]
\item There exists a $2$-elementary primitive sublattice $V_+\subset E_8$ such that $M\subset V_+$ and $r\in V_-=(V_+)^\perp$.
\item The 
map $\sigma_r$ extends to an involution on $E_8$ acting as ${\rm{Id}}$ on $V_+$ and as $-{\rm{Id}}$ on $V_-$.
\end{enumerate}
\end{lemma}

\begin{proof}
Assume (i). Then 
\[g=-{\rm{Id}}\oplus {\rm{Id}}\in O(V_-\oplus V_+)\]
extends $\sigma_r\in O(Q)$. Moreover, $g$ extends to an element of $O(E_8)$
since $V_-$ is 2-elementary, $-\Id$ acts as the identity on $D(V_-)\cong D(V_+)$, and 
\eqref{eq:sec_K3n+:g+g-} is immediately satisfied. The other implication is \cite{GHS07}*{Lemma 3.5}.
\end{proof}

\begin{lemma}
\label{lemma:sec_K3n+:F}
Let $Q\subset E_8$ be a primitive embedding and $\sigma_r\in O(Q)$ a reflection with respect to a primitive element $r\in Q$ that extends to an involution $\widetilde{\sigma}_r\in O(E_8)$. As before we write $V_+$ and $V_-$ for the invariant and anti-invariant sublattices of $\widetilde{\sigma}_r$. 
Denote by $R(r^{\perp}\cap V_-)$ 
the set of roots in $r^{\perp}\cap V_-$.
Let $L=U^{\oplus 2}\oplus E_8(-1)^{\oplus 2}\oplus Q(-1)$ and consider the induced primitive embedding $L\hookrightarrow {\rm{II}}_{2,26}$. If 
\begin{equation}
\label{eq:sec_K3n+:V}
{\rm{rk}}\left(V_-\right)+\frac{1}{2}\left|R(r^{\perp}\cap V_-)\right|\equiv 1\mod 2,
\end{equation}
 then the quasi-pullback 
 of $\Phi_{12}\in M_{12}\left(O^+({\rm{II}}_{2,26}), {\rm{det}}\right)$ to $\Omega^\bullet\left(L\right)$ is modular with respect to 
\[{\rm{det}}:\Gamma\longrightarrow\mathbb{C}^*\]
for the subgroup $\Gamma\subset O^+\left(L\right)$ generated by $\widetilde{O}^+\left(L\right)$ and $\sigma_r$. 
\end{lemma}

\begin{proof}
Let $R_{>0}\cup R_{<0}$ be a sign-partition of the set of roots $R(Q^{\perp})$ in $E_8$ orthogonal to $Q$ (see Section \ref{SS:Phi_{12}}). \
Let $M$ be the number of roots that changes sign under $\widetilde{\sigma}_r$ \eqref{eq:sec_prelim:M}.
By Lemma~\ref{lemma:sec_premil:F} and Corollary~\ref{coro:sec_prelim:ModOtilde} it is enough to check that ${\rm{det}}(\sigma_r)=(-1)^M\cdot{\rm{det}}(\widetilde{\sigma}_r)$ or, equivalently,
\[M+{\rm{rk}}(V_-)\equiv 1\mod 2.\]
Note that for any root $\eta\in R_{>0}$, if $\eta\in V_+$, then $\widetilde{\sigma}_r(\eta)=\eta\in R_{>0}$, and if $\eta\in V_-$, then $\widetilde{\sigma}_r(\eta)=-\eta\in R_{<0}$. 
If $\eta\not\in V_-\cup V_+$, then $\widetilde{\sigma}_r(\eta)\neq \pm \eta$. Let $\{\eta_1,\ldots,\eta_k\}$ be all roots in $R_{>0}$ not in $V_-\cup V_+$. Then $\widetilde{\sigma}_r$ induces an injection
\[\{\eta_1,\ldots,\eta_k\}\hookrightarrow \left\{\eta_1,\ldots,\eta_k, -\eta_1,\ldots,-\eta_k\right\}.\]
with $\widetilde{\sigma}_r(\eta_i)\neq \pm \eta_i$. One immediately sees that if $\widetilde{\sigma}_r(\eta_i)=-\eta_j$, then $\widetilde{\sigma}_r(\eta_j)=-\eta_i$, thus
\[\left|\left\{\widetilde{\sigma}_r(\eta_1),\ldots,\widetilde{\sigma}_r(\eta_k)\right\}\cap \left\{-\eta_1,\ldots,-\eta_k\right\}\right|\equiv 0\mod 2\]
and $M\equiv \left|R_{>0}\cap V_-\right|\mod 2$. Finally, we have $Q\cap V_-=\langle r \rangle$, and 
\[\left|R_{>0}\cap V_-\right|=\frac{1}{2}R(Q^{\perp}\cap V_-)=\frac{1}{2}R(r^\perp\cap V_-).\qedhere\]
\end{proof}

We have to choose our embedding $Q_h\hookrightarrow E_8$ not only to have a controlled number of roots in the orthogonal complement, but also in such a way that the reflection $\sigma_v$ defined in Lemma \ref{lemma:sec_K3n:Ohat=Otilde} can be extended to an involution $\widetilde{\sigma}_v\in O(E_8)$ satisfying \eqref{eq:sec_K3n+:V}. Note that if $\gamma\in \{1,2\}$, by Proposition~\ref{prop:sec_K3n:a's}, we may choose
\begin{equation}
\label{eq:sec_K3n+:a}
a=\left\{\begin{array}{ll}
0&\hbox{when $\gamma=1$}\\
1&\hbox{when $\gamma=2$.}
\end{array}\right.
\end{equation}
In particular, by Lemma~\ref{lemma:sec_K3n:h}, $h\in \Lambda$ can be chosen as 
\begin{equation}
\label{eq:sec_K3n+:h}
h=\left\{\begin{array}{ll}
e+df&\hbox{when $\gamma=1$}\\
2(e+tf)-\ell&\hbox{when $\gamma=2$,}
\end{array}\right.
\end{equation}
where $d=4t-(n-1)$ when $\gamma=2$, see Proposition \ref{prop:sec_K3n:nonemp}.

Recall that $Q_h(-1)=\langle h\rangle^\perp\subset U\oplus \langle\ell\rangle$ is generated by $\left\{\frac{2a(n-1)}{\gamma}f-\ell, e-tf\right\}$, and
\begin{equation}
\label{eq:sec_K3n+:Q}
Q_h=\left(\begin{array}{cc}2(n-1)& 0\\ 0& 2d \end{array}\right)\;\hbox{if $\gamma=1$;}\;\;\;\;\;\; Q_h=\left(\begin{array}{cc}2(n-1)& -(n-1)\\ -(n-1)& 2t \end{array}\right)\;\hbox{if $\gamma=2$.}
\end{equation}

Summarizing the discussion above, and adding the conditions for modularity of the quasi-pullback in Lemma \ref{lemma:sec_premil:F}, the vanishing at the ramification in Proposition \ref{prop:sec_prelim:F=0}, and a numerical condition ensuring that $\widehat{O}^+\left(\Lambda,h\right)$ has no irregular cusps, we have:

\begin{proposition}
\label{prop:sec_K3n+:Fcusp}
Let $Q_h$ be the lattice defined in \eqref{eq:sec_K3n+:Q}, $Q_h\hookrightarrow E_8$ a primitive embedding, and $\Lambda_h\hookrightarrow{\rm{II}}_{2,26}$ the induced embedding. Assume further that:
\begin{enumerate}[(1)]
\item There exists a primitive $2$-elementary sublattice $V_+\subset E_8$ such that
$$z_1=\frac{2a(n-1)}{\gamma}f-\ell\in V_-:=\left(V_+\right)^\perp\subset E_8,\;\;\;\hbox{and}\;\;\;(z_1)^{\perp}\cap Q_h\subset V_+.$$
\item The number of roots $\left|R(Q_h^\perp)\right|$ in the orthogonal complement $Q_h^\perp\subset E_8$ is at least $2$ and at most $14$ (resp. $16$).
\item The rank of $V_-$ and half the number of roots orthogonal to $Q_h$ in $V_-$ have opposite parities:
$${\rm{rk}}(V_-)+\frac{1}{2}\left|R\left((z_1)^\perp\cap V_-\right)\right|\equiv 1\mod 2.$$
\item One of the following holds:
\begin{enumerate}[(a)]
\item The number $\frac{|R(Q_h^{\perp})|}{2}$ is even;
\item For any $r\in \Lambda_h\subset \rm{II}_{2,26}$ such that $-\sigma_r\in \widehat{O}^+\left(\Lambda,h\right)$ one has 
\[\left|R\left(\Lambda_h^\perp\right)\right|<\left|R\left(\left(\Lambda_h\right)_r^\perp\right)\right|,\]
where $\left(\Lambda_h\right)_r$ is the orthogonal complement of $r$ in $\Lambda_h$ and orthogonal complements in the inequality are taken in $\rm{II}_{2,26}$.
\end{enumerate}
\item $\frac{2d}{\gamma}\neq 1,2,4,8$.
\end{enumerate}
Then the modular variety \[\Omega\left(\Lambda_h\right)\big/\widehat{O}^+\left(\Lambda,h\right)\]
is of general type (resp.\ has non-negative Kodaira dimension). 
\end{proposition}

\begin{proof}
By Lemmas \ref{lemma:sec_K3n+:F} and \ref{lemma:sec_K3n:Ohat=Otilde}, the quasi-pullback $F$ of the quasi-pullback of Borcherds form $\Phi_{12}$ 
to $\Omega^\bullet\left(\Lambda_h\right)$ is modular with character ${\rm{det}}\colon\widehat{O}^+\left(\Lambda,h\right)\longrightarrow \mathbb{C}^*$ and weight $12+\frac{\left|R(Q_h^\perp)\right|}{2}$ \eqref{eq_weight_quasipullback}. 
Moreover, since $2\leq\left|R(Q_h^\perp)\right|\leq16$, $F$ is a cusp form \cite{GHS13}*{Corollary 8.12}.
We claim that $F$ vanishes along the ramification divisor of the modular projection
\[\pi\colon\Omega\left(\Lambda_h\right)\longrightarrow\Omega\left(\Lambda_h\right)\big/\widehat{O}^+\left(\Lambda,h\right).\]
Indeed, if the embedding $Q_h
\hookrightarrow E_8$ satisfies assumption (4a), this follows from Proposition~\ref{prop:sec_prelim:F=0}. If the embedding satisfies (4b), it follows from Lemma~\ref{lemma_vanishing_ramif},
where we use that $\mathrm{rk}(\Lambda_h)=22$.
The proposition now follows from Theorem \ref{thm:sec_prelim:low-weight} once we have shown that $\Omega(\Lambda_h)$ has no irregular cusps for $\widehat{O}^+\left(\Lambda,h\right)$. 
By Proposition~\ref{prop:sec_prelim:irr_1cusp}, it is enough to show this for 0-dimensional cusps.
Let $I\subset \Lambda_h$ be a rank one primitive isotropic sublattice corresponding to a $0$-dimensional cusp of $\Omega(\Lambda_h)$.
Assume $I$ is irregular for $\widehat{O}^+(\Lambda,h)$, so by Proposition \ref{prop:sec_prelim:irr_cusp} (see also the proof of \cite{Ma21}*{Proposition 4.11}),  $-\mathrm{Id}\not\in \widehat{O}^+\left(\Lambda,h\right)$ and $-E_{w}\in \Gamma(I)_{\mathbb{Q}}\cap \widehat{O}^+\left(\Lambda,h\right)$ for some $w\in L(I)_{\mathbb{Q}}$.
Since $I$ has rank one, $w$ can be written as $m\otimes l$ with $m\in \left(I^\perp\big/I\right)_{\mathbb{Q}}$ and $l\in I_\mathbb{Q}$. Then by \ref{lemma:sec_K3n:Ohat=Otilde}
\[E_{2m\otimes l}=\left(-E_{m\otimes l}\right)\circ\left(-E_{m\otimes l}\right)\in \Gamma(I)_{\mathbb{Q}}\cap \widetilde{O}^+\left(\Lambda_h\right).\]
The isomorphism \eqref{eq:sec_prelim:E_w} when restricted to $\widetilde{O}^+(\Lambda_h)$ induces an isomorphism $L(I)\cong U(I)\cap \widetilde{O}^+(\Lambda_h)$, see \cite{Ma21}*{Lemma 4.1}. In particular $2m\otimes l\in L(I)$. We choose a lift $\widetilde{m}\in \frac{1}{2}I^\perp$ of $m$, so for any $u\in(\Lambda_h)_{\mathbb{Q}}$
\begin{equation}
\label{eq:sec_K3n+:-E}
-E_{m\otimes l}(u)=-u+(\widetilde{m},u)l-(l,u)\widetilde{m}+\frac{1}{2}(m,m)(l,u)l.
\end{equation}
Moreover, any $g\in\widehat{O}^+\left(\Lambda,h\right)$ fixes the class $\overline{u}\in D(\Lambda_h)$ of any $u\in (z_1)^\perp\subset \left(\Lambda_h\right)^\vee$ (see Lemma~\ref{lemma:sec_K3n:Ohat=Otilde}).
In particular, for any such $u$ we have $-E_{m\otimes l}(u)\in u+\Lambda_h$. From \eqref{eq:sec_K3n+:-E}, since $(2\widetilde{m}, u), (l,u)$, and $(2m,m)$ are integers, it follows that
\[2u\in\frac{1}{4}\Lambda_h,\]
so $8u\in \Lambda_h$ and $\mathrm{ord}(\overline{u})$ divides $8$.
Take $u=\frac{\gamma}{2d}\left(az_1+\gamma z_2\right)$ as in the proof of Proposition~\ref{prop:sec_K3n:Otilde_h}.
Indeed, then $(u.z_1)=0$ and thus, $\sigma_{z_1}(u)=u$.
The order of the class $\overline{u}$ 
in $D(\Lambda_h)$
is $2d/\gamma$. In particular, we have 
$\frac{2d}{\gamma}\in\left\{1,2,4,8\right\}$.
\end{proof}

Recall that for $\gamma=1,2$, the moduli space $\mathcal{M}_{\mathrm{K}3^{[n]},2d}^\gamma$ is connected when non-empty, see Proposition~\ref{prop:sec_K3n:a's}.
As a corollary we have:

\begin{corollary}
\label{cor:sec_K3n+:main_cor}
Let $(n,d,\gamma)$ be a triple such that $\gamma\in \{1,2\}$ and $\mathcal{M}_{\mathrm{K}3^{[n]},2d}^\gamma$ is non-empty. Assume there exists an embedding $Q_h\hookrightarrow E_8$ satisfying the hypothesis of Proposition \ref{prop:sec_K3n+:Fcusp}. Then $\mathcal{M}_{\mathrm{K}3^{[n]},2d}^\gamma$ is of maximal (resp. non-negative) Kodaira dimension. 
\end{corollary}

\begin{proof}
This follows from Proposition \ref{prop:sec_K3n+:Fcusp} together with Theorem \ref{thm:sec_prelim:low-weight}.
\end{proof}

\section{Kodaira dimension for \texorpdfstring{K$3^{[n]}$}{K3[n]} type with divisibility 1}
\label{section:K3n_div1}

From now on we concentrate on the split case $\gamma=1$. Recall that in this case the moduli space $\mathcal{M}_{\mathrm{K}3^{[n]},2d}^1$ is irreducible, see Proposition~\ref{prop:sec_K3n:a's}. Note that in this case $a=0$ and the element $z_2$ in Proposition \ref{prop:sec_K3n+:Fcusp} is given by
\[z_2=\frac{2a(n-1)}{\gamma}f-\ell=-\ell.\]

\begin{proposition}
\label{prop:sec_K3n+:rootcount}
Consider the two primitive $2$-elementary orthogonal sublattices of $E_8$ given by
\[
\begin{aligned}
V_-&:=\langle e_1-e_5, e_2-e_6, e_3-e_7, e_4-e_8\rangle\cong A_1^{\oplus 4}\\
V_+&:=\langle e_1+e_5, e_2+e_6, e_3+e_7, e_4+e_8\rangle\cong A_1^{\oplus 4}.
\end{aligned}
\]
Let $Q_h$ be as in \eqref{eq:sec_K3n+:Q}, which in the case $\gamma=1$ has basis 
\begin{equation}
\label{eq:sec_K3n+:basis}
z_1=\frac{2a(n-1)}{\gamma}f-\ell=-\ell\;\;\;\hbox{and}\;\;\;z_2=e-tf=e-df.
\end{equation}
 Assume further that $n-1>6$, $n-1\equiv 1,2 \mod 4 $, and $n-1\not \in \{10,13,25,37,58, 85, 130, \star\}$. Then for any $d\ge 3$, $d\not \equiv 0\mod 4$, there is a primitive embedding $ Q_h\hookrightarrow E_8$ sending $z_1$ to $V_-$ and $z_2$ to $V_+$ such that:
\begin{enumerate}[(1)]
    \item the number $\frac{1}{2}\left|R\left((z_1)^\perp\cap V_-\right)\right|$ is odd and
    \item the number of roots $\left|R(Q_h^\perp)\right|$ in the orthogonal complement $Q_h^\perp\subset E_8$ satisfies
    \[2\le |R(Q_h^\perp)|\le 14.\]
    \item If moreover, $d\not\equiv 7 \mod 8$ and $d\not \in \{5,10,13,25,37,58,85,130, \star\}$, the primitive embedding $ Q_h\hookrightarrow E_8$ may be chosen so that  $|R(Q_h^\perp)|\equiv 0 \mod 4$. 
\end{enumerate}
\end{proposition}

\begin{proof}
By Theorem \ref{thm:sec_prelim:HK21}, we may write
\[n-1=\alpha_1^2+\alpha_2^2+\alpha_3^2,\]
where $\alpha_1, \alpha_2, \alpha_3$ are coprime and positive. Moreover, since $n-1\not \equiv 3 \mod 4$, we may assume without loss of generality that $\alpha_1$ is even and $\alpha_2$ is odd. Similarly, by Corollary~\ref{cor:sec_prelim:sumsofsquares2}, we may write
\begin{equation}
\label{eq:sec_K3n+:xs}
d=x_1^2+x_2^2+x_3^2+x_4^2,
\end{equation}
with $x_1\ge 0$ and $x_i>0$ for $i\in \{2,3,4\}$ (except in the case $d=5$, where $x_1=x_2=0$). Since we know $d\not \equiv 0 \mod 4$, we may without loss of generality take $x_1$ to be even and $x_4$ odd. 
In the case that  $d\not\equiv 7 \mod 8$ and $d\not \in \{10,13,25,37,58,85,130, \star\}$, we let $x_1=0$ (which is possible by Theorem \ref{thm:sec_prelim:HK21}).

Consider the embedding $Q_h\hookrightarrow E_8$ given by
\begin{equation}
\label{eq:sec_K3n+: gamma1 embedding}
\begin{aligned}
    z_1&\mapsto v_1=\alpha_1e_1+\alpha_2e_2+\alpha_3e_3 - \alpha_1e_5-\alpha_2e_6-\alpha_3e_7\\
    z_2 &\mapsto v_2 = x_1e_1 + x_2 e_2 + x_3e_3 + x_4 e_4+x_1e_5 + x_2 e_6 + x_3e_7 + x_4 e_8.
\end{aligned}
\end{equation}
Using Lemma \ref{lemma:sec_K3n:odd-coeffs} together with our assumptions on $x_i$ and $\alpha_i$ one checks that the embedding \eqref{eq:sec_K3n+: gamma1 embedding} is primitive.
Moreover, the given embedding has the property that $v_1\in V_-$ and $v_2\in V_+$. The intersection $v_1^\perp \cap V_-$ contains the single root $e_4-e_8$. In particular, the number of roots in $v_1^\perp \cap V_-$ is odd.

Now we count roots in $Q_h^\perp\subset E_8$. We assume first that $d\ne 5,6$. Then in particular,  $x_i>0$ for $i\in \{2,3,4\}$.  The integral roots are then:

\begin{enumerate}
    \item $\pm (e_4- e_8)$
    \item \label{gamma1_rootcount_2}$\pm (e_1+ e_{5})$ if $x_1=0$
    \item \label{gamma1_rootcount_3} $\pm (e_i-e_j)$, $\pm (e_{i+4}-e_{j+4})$ for $i,j\in \{1,2,3\}$ if $\alpha_i=\alpha_j$ and $x_i=x_j$. 
\end{enumerate}
Since $\alpha_1, \alpha_2, \alpha_3$ are coprime and thus in particular cannot all coincide, there are either $0$ or $4$ roots of type \eqref{gamma1_rootcount_3}. Hence the total number of integral roots is either $2$ or $6$ if $x_1\ne 0$ and is either $4$ or $8$ if $x_1=0$.

Suppose that $w$ is a fractional root in $Q_h^\perp$. Then $w$ is of the form 
\begin{equation}
\label{eq:sec_K3n+:wpm}
w=\frac{1}{2}\left(\pm e_1\pm e_2\pm e_3\pm e_4\pm e_5\pm e_6\pm e_7\pm e_8\right),
\end{equation}
where the number of $+$ signs is even. Hence, the number of indices $i\in \{1,2,3,4\}$ such that $e_i$ and $e_{i+4}$  have the same sign in \eqref{eq:sec_K3n+:wpm} must be even. 

\vspace{0.3cm}
\textbf{CASE 1:} There is no 
$i\in \{1,2,3\}$ such that $e_i$ and $e_{i+4}$ have the same sign in \eqref{eq:sec_K3n+:wpm}.
\vspace{0.1cm}

The equation $(w,v_1)=0$ implies that there is a choice of signs such that
\[\pm 2\alpha_1\pm 2\alpha_2\pm 2\alpha_3=0.\]
Hence the $\alpha_i$'s must satisfy a relation of the form $\alpha_\ell=\alpha_j+\alpha_k$. But $\alpha_1, \alpha_2, \alpha_3$ are positive and coprime, and $n-1>6$, 
so if such a relation holds, no two $\alpha_i$ can be equal. So if there are any fractional roots of this type, there must be exactly four, given by
\[\pm \frac{1}{2}\left((e_\ell-e_{\ell+4})-(e_j-e_{j+4})-(e_k-e_{k+4})\pm(e_4-e_8)\right).\] 

\vspace{0.2cm}
\textbf{CASE 2:} There is exactly one index $i\in \{1,2,3\}$ such that $e_i$ and $e_{i+4}$ have the same sign in \eqref{eq:sec_K3n+:wpm}. 
\vspace{0.1cm}

Then $e_4$ and $e_8$ must have the same sign in \eqref{eq:sec_K3n+:wpm}. Moreover, $(w,v_2)=0$ implies that there is a choice of sign such that
\[\pm 2x_i \pm 2x_4=0.\]
Since $x_i$ is non-negative and $x_4$ is positive, it follows that $x_i=x_4$. Additionally, $(w,v_1)=0$ implies that there is a choice of sign such that 
\[\pm 2\alpha_j\pm 2\alpha_k=0,\]
where $j,k\ne i$, and so $\alpha_j=\alpha_k$. Thus, since $\alpha_1, \alpha_2, \alpha_3$ cannot all be equal, if there are any fractional roots of this type, there must be exactly four, given by 
\[\frac{1}{2}\left(\pm (e_i+e_{i+4}-e_{4}-e_8)\pm (e_j-e_{j+4}-e_k+e_{k+4})\right).\]

\vspace{0.2cm}
\textbf{CASE 3: } There are exactly two indices $i\in \{1,2,3\}$ such that $e_i$ and $e_{i+4}$ have the same sign in \eqref{eq:sec_K3n+:wpm}.
\vspace{0.1cm}

In this case, the equation $(w,v_1)=0$ implies that there is a $j$ with $2\alpha_j=0$. This contradicts our assumptions on $\alpha_j$. 

\vspace{0.3cm}
\textbf{CASE 4:} All indices $i\in \{1,2,3\}$ are so that $e_i$ and $e_{i+4}$ have the same sign in \eqref{eq:sec_K3n+:wpm}.
\vspace{0.1cm}

In this case, $e_4$ and $e_8$ must have the same sign in \eqref{eq:sec_K3n+:wpm}. Since $(w,v_2)=0$, there is a choice of sign such that
\begin{equation}\label{eq: signs div 1}\pm x_1\pm x_2\pm x_3\pm x_4=0.\end{equation}
Note that each choice of signs such that \eqref{eq: signs div 1} holds, corresponds to exactly one fractional root in $Q_h^\perp$. The number of choices of signs such that \eqref{eq: signs div 1} holds is twice the number of choices of signs such that 
\begin{equation}\label{eq: signs div 1 rewrite} x_1=\pm x_2\pm x_3\pm x_4.\end{equation}
Note that if $x_1=0$ and there is a choice of signs such that \eqref{eq: signs div 1 rewrite} holds, then the opposite choice of signs also makes \eqref{eq: signs div 1 rewrite} hold. In general, if there are two choices of signs such that \eqref{eq: signs div 1 rewrite} holds,
 the choices of signs differ by either $1$, $2$, or $3$ signs, where the latter only occurs if $x_1=0$. 
If two such equations differ by $1$ sign, subtracting the equations yields $x_i=0$ for some $i\ne 1$, which is not possible. If two such equations differ by $2$ signs, we have relations $x_1=x_i$ for some $i\in \{1,2,3\}$ (in particular, $x_1\neq 0$) and $x_j=x_k$ for $j,k\ne i$. 
Hence there are at most two choices of signs such that \eqref{eq: signs div 1 rewrite} holds and if $x_1=0$, there are exactly two such choices.  It follows that there are at most four fractional roots of this type and that if $x_1=0$ and  there are any fractional roots of this type, there must be exactly four.

This gives us the root count for $d\neq 5,6$. Indeed, if $\alpha_1, \alpha_2, \alpha_3$ are all different, then there are at most $4$ integral roots, $4$ fractional roots from Case 1, and $4$ fractional roots from Case 3. 
Hence there are between $2$ and $12$  roots in $Q_h^\perp$. If moreover $x_1=0$, then there either $8$ or $12$ roots in $Q_h^\perp$.

If two of the $\alpha_i$ are equal and $x_1\ne 0$, then 
there are at most $6$ integral roots, no fractional roots from Case 1, $4$ fractional roots from Case 2, and $4$ fractional roots from Case 4. Hence there are between $2$ and $14$ roots. 

Finally, assume that two of the $\alpha_i$ are equal and $x_1=0$. 
If $x_2, x_3, x_4$ are all different, then the only roots are $4$ integral roots and either $0$ or $4$ fractional roots from Case 4. If two of $x_2, x_3, x_4$ are equal, then the only roots are either $4$ or $8$ integral roots and either $0$ or $4$ fractional roots from Case 2 (Case 4 cannot occur because $x_2,x_3,x_4$ are coprime and $d\ne 6$). Hence there are $4$, $8$, or $12$ roots in $Q_h^\perp$.
This finishes the count for $d\ne 5,6$.

When $d=5$, we take $(x_1,x_2,x_3,x_4)=(0,0,2,1)$. Since $\alpha_1\ne \alpha_2$, there are $6$ integral roots: $\pm (e_4-e_8),$ $\pm (e_1+e_5),$ $\pm (e_2+e_6)$. Moreover, if $w$ is a fractional root in $Q_h^\perp$, then $(w,v_2)=0$ implies that the signs of $e_3$ and $e_4$ are opposite to the signs of $e_7$ and $e_8$ in \eqref{eq:sec_K3n+:wpm}.
Then $\alpha_\ell=\alpha_j+\alpha_k$ in which case there are $4$ fractional roots (as in Case 1). 
Therefore there between $6$ and  $10$ roots in $Q_h^\perp$ in this case.

When $d=6$, we take $(x_1,x_2,x_3,x_4)=(0,2,1,1)$. Then there are $4$ integral roots. Moreover, there are either $0$ or $4$ fractional roots in Case 1, no fractional roots in Case 2 (since $\alpha_1\ne \alpha_2$), and either $0$ or $4$ fractional roots in Case 4. Hence there either $4$, $8$, or $12$ roots in $Q_h^\perp$. 
This finishes the proof for $d\geq 3$. 
\end{proof}

In light of Proposition \ref{prop:sec_K3n+:Fcusp}, the only missing ingredient in the proof of Theorem \ref{thm: thm5} is the vanishing of the quasi-pullback of Borcherds form along the ramification divisor of the modular projection 
\[\pi:\Omega\left(\Lambda_h\right)\longrightarrow\Omega\left(\Lambda_h\right)\big/\widehat{O}\left(\Lambda,h\right).\]
When $d$ satisfies the conditions of part~(3) of Proposition~\ref{prop:sec_K3n+:rootcount} above, this follows from Lemma~\ref{lemma_vanishing_ramif}: note that the weight of the pullback of Borcherds form to $\Omega(\Lambda_h)$ is $12+\frac12|R(Q_h^{\perp})|$ by \eqref{eq_weight_quasipullback}.
For other $d$, vanishing along the ramification divisor is the content of Proposition \ref{prop:sec_K3n+:F=0}. In order to prove it, we have to put constraints on the possible $r\in\Lambda$ such that $\pm\sigma_r$ lies inside the modular group $\widehat{O}^+\left(\Lambda,h\right)$. 
Recall that we are under the assumption that $Q_h$ is of the form $\eqref{eq:sec_K3n+:Q}$ with basis \eqref{eq:sec_K3n+:basis}. We keep the same basis for $Q_h(-1)$. Recall that 
\[D\left(\Lambda_h(-1)\right)\cong D\left(Q_h\right)=\left\langle\frac{z_1}{2(n-1)}\right\rangle\oplus\left\langle\frac{z_2}{2d}\right\rangle\cong \mathbb{Z}\big/2(n-1)\mathbb{Z}\oplus \mathbb{Z}\big/2d\mathbb{Z}.\]
As usual, we regard $\widehat{O}^+(\Lambda,h)$ as a subgroup of $O(\Lambda_h)$ via restriction. Elements $g\in\widehat{O}^+(\Lambda,h)$ act on $D\left(Q_h\right)$ as $\pm{\rm{Id}}\oplus {\rm{Id}}$.

\begin{lemma}
\label{lemma:sec_K3n+:r_1}
Let $r\in \Lambda_h$ be a reflective element such that $-\sigma_r\in \widehat{O}^+\left(\Lambda,h\right)$ and $-\sigma_r\not\in\widetilde{O}^+\left(\Lambda_h\right)$. Then $r^2=-2d$ and ${\rm{div}}(r)$ is either $d$ or $2d$.
\end{lemma}

\begin{proof}
Recall \eqref{eq:sec_prelim:sigma_r1} that if $-\sigma_r\in\widehat{O}^+\left(\Lambda,h\right)$, then
\begin{equation}
\label{eq:sec_K3n+:div(r)}
(r,r)\in\left\{{\rm{div}}(r), 2{\rm{div}}(r)\right\}.
\end{equation}
Let $r=Au+Bz_1+Cz_2$ with $u\in U^{\oplus 2}\oplus E_8^{\oplus 2}$ primitive in the unimodular part of $\Lambda_h(-1)$ and $A,B,C$ coprime integers. Write $(r,r)=2m$. Unimodularity and \eqref{eq:sec_K3n+:div(r)} imply that $m$ divides $A$. Furthermore,
since 
\[-\sigma_r\left(\frac{z_1}{2(n-1)}\right)\equiv -\frac{z_1}{2(n-1)}\;\;\;\hbox{and}\;\;\;-\sigma_r\left(\frac{z_2}{2d}\right)\equiv \frac{z_2}{2d}\mod \Lambda_h(-1),\]
one has 
\[\frac{B}{m}r\equiv 0\;\;\;\hbox{and}\;\;\;\frac{C}{m}r\equiv\frac{1}{d}z_2\mod \Lambda_h(-1).\]
As $r$ is primitive, it follows that $m$ divides $B$ and $dC$. 
Since ${\rm{gcd}}(A,B,C)=1$, $m$ cannot divide $C$. Comparing orders in the second equation together with our primitivity assumption on $r$ gives us that $m=d$.
\end{proof}

\begin{lemma}
\label{lemma:sec_K3n+:r_2}
Assume both $d$ and $n-1$ are not divisible by $4$. Then for any reflective
element $r\in \Lambda_h$ such that $-\sigma_r\in \widehat{O}^+\left(\Lambda,h\right)$, there exists an isometry $g\in O(\Lambda_h)$ such that 
\begin{equation}
\label{eq:sec_K3n+:g(r)}
g(r)=\left\{\begin{array}{ll}
de+z_2&\hbox{when ${\rm{div}}(r)=d$,}\\
z_2&\hbox{when ${\rm{div}}(r)=2d$}
\end{array}\right.
\end{equation}
and $-\sigma_{g(r)}$ acts as $-{\rm{Id}}\oplus {\rm{Id}}$ on $D(\Lambda_h)$. In particular, $-\sigma_{g(r)}\in \widehat{O}^+\left(\Lambda,h\right)$.
Moreover, if $d$ is an odd prime and $(n-1)d\equiv 2,3\mod 4$, then we can take $g\in\widehat{O}^+(\Lambda,h)$.
\end{lemma}

\begin{proof}
As before, let $r=Au+Bz_1+Cz_2\in \Lambda_h(-1)$ with ${\rm{gcd}}(A,B,C)=1$. Since $2d=(r,r)=A^2(u,u)-B^2\cdot 2(n-1)-C^2\cdot 2d$, see Lemma \ref{lemma:sec_K3n+:r_1}, and $d$ divides $A$ and $B$, then $d$ must divide $C^2-1$. Consider the map $\phi:D(\Lambda_h(-1))\to D(\Lambda_h(-1))$
\begin{equation}
\label{eq:sec_K3n+:phi}
(z_1)_*\mapsto(z_1)_*\;\;\;\hbox{and}\;\;\;(z_2)_*\mapsto\left\{\begin{array}{cc}C(z_2)_*&\hbox{if $C$ is odd,}\\(C+d)(z_2)_*&\hbox{if $C$ is even.}\end{array}\right.
\end{equation}
One checks that this map is a group isomorphism and $d\not\equiv 0\mod 4$ implies that it preserves the quadratic form, i.e., $\phi\in O(D(\Lambda_h(-1)))$. Now if ${\rm{div}}(r)=d$ and we write $r=A'du+B'dz_1+Cz_2$, then
\[r_*=\frac{r}{{\rm{div}}(r)}=A'u+B'z_1+\frac{C}{d}z_2\equiv2C\frac{z_2}{2d}=2C(z_2)_*\mod \Lambda_h(-1).\]
Note that via \eqref{eq:sec_K3n+:phi}, the element $2C(z_2)_*$ is in the same $O(D(\Lambda_h(-1)))$-orbit as $2(z_2)_*$ if $C$ is odd, and as $2(d+1)(z_2)_*=2(z_2)_*$ if $C$ is even. Since 
\[2\frac{z_2}{2d}\equiv \frac{de+z_2}{d}=\frac{de+z_2}{{\rm{div}}(de+z_2)}\mod \Lambda_h(-1),\]
and the projection map $O(\Lambda_h(-1))\to O(D(\Lambda_h(-1)))$ is surjective, Eichler's criterion (Theorem \ref{thm:sec_prelim:Eichler}) gives us the Lemma when ${\rm{div}}(r)=d$.
The case $\divis(r)=2d$, $B$ even is treated analogously. 
If $\divis(r)=2$ and $B$ is odd, we have $C^2-1+(n-1)d\equiv 0\mod 4d$. Under the additional assumption $n-1\not\equiv 0\mod 4$, it follows that $C$ is even. Using the map $\phi\in O(D(\Lambda_h(-1))$ that sends $(z_1)_*$ to $n(z_1)_*+d(z_2)_*$ and $(z_2)_*$ to $(n-1)(z_1)_*+C(z_2)_*$, one shows that $r$ is equivalent to $z_2$ under $O(\Lambda_h(-1))$.

If $d$ is an odd prime, then $c^2-1\equiv 0\mod d$ implies $c\equiv\pm 1\mod d$. Hence when $\divis(r)=d$, the class $r_*=2p(z_2)_*\in D(\Lambda_h(-1))$ is equivalent to $\pm 2(z_2)_*$, and the statement follows. 
When $\divis(r)=2d$ and $B$ is even, we have $c^2-1\equiv 0\mod 4d$ and therefore $c\equiv \pm 1\mod 2d$. The class $r_*=p(z_2)_*$ is equivalent to $\pm(z_2)_*$ and the statement follows. 
Finally, when $\divis(r)=2d$ and $B$ is odd we have $c^2-1+(n-1)d\equiv 0\mod 4d$, which cannot happen when $(n-1)d\equiv 2,3\mod 4$.
\end{proof}

\begin{proposition}
\label{prop:sec_K3n+:F=0}
Suppose that $\gamma=1$ and $d$ is prime, that either $n-1\equiv 1$ and $d\equiv3$ mod 4 or $n-1\equiv 2$ and $d\equiv1$ mod 4, and that $n-1>6$ and $n-1\notin\{10, 13, 25, 37, 58, 85, 130, \star\}$: in particular, (1) and (2) of Proposition~\ref{prop:sec_K3n+:rootcount} hold.
Let $Q_h\hookrightarrow E_8$ be the primitive embedding defined in \eqref{eq:sec_K3n+: gamma1 embedding} and consider the induced embedding $\Lambda_h\hookrightarrow \mathrm{II}_{2,26}$. Then for any $r\in \Lambda_h$ such that $-\sigma_r\in \widehat{O}^+\left(\Lambda,h\right)$ one has 
\[\left|R\left(\Lambda_h^\perp\right)\right|<\left|R\left(\left(\Lambda_h\right)_r^\perp\right)\right|.\]
In particular, the quasi-pullback $F$ of Borcherds form $\Phi_{12}$ to $\Omega^\bullet(\Lambda_h)$ is modular with respect to $\mathrm{det}:\widehat{O}^+\left(\Lambda,h\right)\longrightarrow\mathbb{C}^*$ and it vanishes along the ramification divisor of the modular projection
\[\pi:\Omega\left(\Lambda_h\right)\longrightarrow\Omega\left(\Lambda_h\right)\big/\widehat{O}^+\left(\Lambda,h\right).\]
\end{proposition}

\begin{proof}
If $-\sigma_r\in \widetilde{O}^+\left(\Lambda_h\right)$, the statement follows from Proposition \ref{prop:sec_K3n:F=0}. If $-\sigma_r\in \widehat{O}^+\left(\Lambda,h\right)$ and not in $\widetilde{O}^+(\Lambda_h)$, then in light of Lemma \ref{lemma:sec_K3n+:r_2} we may assume $r=de+z_2$ or $r=z_2$.
Assume first that $r=z_2$, then 
\[\left(\Lambda_h\right)_r^\perp=\left(\mathbb{Z}z_1\right)^\perp\subset E_8(-1)\]
and under the embedding \eqref{eq:sec_K3n+: gamma1 embedding} restricted to $\langle z_1\rangle$, our counting argument in Proposition \ref{prop:sec_K3n+:rootcount} leads to more than $14$ roots only orthogonal to $z_1$. Indeed, recall that by assumption if $d>1$, then $x_1,\ldots,x_3$ in \eqref{eq:sec_K3n+:xs} are not all equal to zero. If $x_i\neq 0$ with $i\in\{1,2,3\}$, then $e_i+e_{i+4}$ is orthogonal to $z_1$ but not to $z_2$, i.e., 
\[e_i+e_{i+4}\in R\left(\left(\Lambda_h\right)_r^\perp\right)\setminus R\left(\Lambda_h^\perp\right).\]
Assume now that $r=de+z_2$. Then 
\[\left(\Lambda_h\right)_r=U\oplus E_8(-1)^{\oplus 2}\oplus\langle e, z_1, 2f-z_2\rangle.\]
Note that $(e_i, e_{i+4})=-4x_i$. In particular, there is always a root $s\in E_8(-1)$ such that $(s,z_1)=0$ and $(s,z_2)=2k$ for some non-zero integer $k$. Consider the root given by 
\[t=ke+s\in {\rm{II}}_{2,26}.\]  
Then one immediately checks that $t$ is orthogonal to $e, z_1,$ and $2f-z_2$, but not orthogonal to $z_2$. In particular,
$t\in R\left(\left(\Lambda_h\right)_r^\perp\right)\setminus R\left(\Lambda_h^\perp\right)$.
\end{proof}

Summarizing, we have proved the following:
\begin{theorem}
\label{thm:sec_K3n+:gamma1_1}
Let $n,d$ be two positive integers such that 
\begin{enumerate}[(1)]
\item $n-1>6$, 
\item $n-1\equiv 1,2\mod 4$,
\item $n-1\not\in\{10,13,25,37,58,85,130,\star\},$ and
\item $d$ satisfies one of the following:
\begin{enumerate}
\item $d\geq 3$, $d\not\equiv 0,4,7\mod 8$ and $d\notin\{5,10,13,25,37,58,85,130,\star\}$;
\item $d$ is an odd prime and $d\equiv 3 \mod 4$.
\end{enumerate}
\end{enumerate}
Then $\mathcal{M}_{\mathrm{K}3^{[n]},2d}^1$ is of general type.
\end{theorem}
\begin{proof}
Note that under our assumptions $d\geq3$ and $(n-1)d\equiv 2,3\mod 4$. 
If assumption (4a) holds, the statement follows from Proposition~\ref{prop:sec_K3n+:rootcount} and Corollary~\ref{cor:sec_K3n+:main_cor}. If instead (4b) holds, we additionally use Proposition~\ref{prop:sec_K3n+:F=0},
where we note that under our assumptions $d\geq3$ and $(n-1)d\equiv 2,3\mod 4$.
\end{proof}

Similar to the strange duality phenomenon, one can construct various generically finite rational maps between moduli spaces $\mathcal{M}_{\mathrm{K}3^{[n]} 2d}^\gamma$ using lattice morphisms. This has already been 
done by O'Grady \cite{OG89}*{page 163} (see also \cites{kon93,kon99}) in the surface case;
our argument is similar. We restrict 
to the cases $\gamma=1,2$. The following are examples in a much larger class of generically finite maps that one can construct between the moduli spaces $\mathcal{M}_{{\rm{K}}3^{[n]},2d}^{\gamma}$. We will use them to prove Theorems \ref{thm: thm5} and \ref{thm:gamma2}.

\begin{lemma}
\label{lemma:sec_K3n+:fm1}
Let $d,r,n$ be positive integers with $n\geq 2$.
\begin{enumerate}[(1)]
\item The moduli spaces 
$\mathcal{M}_{\mathrm{K}3^{[n]},2dr^2}^1$ and $\mathcal{M}_{K3^{[(n-1)r^2+1]},2d}^1$
admit dominant rational maps to $\mathcal{M}_{\mathrm{K}3^{[n]},2d}^1$. 
\item Assume $n-1=k^2$ with $k$ odd.
The space $\mathcal{M}_{K3^{[n]}, 2d}^2$
is non-empty if and only if $\mathcal{M}_{\mathrm{K}3^{[2]},2d}^2$ is, 
and the former admits a dominant rational map to the latter.
\end{enumerate}
\end{lemma}

\begin{proof}
Consider the elements $h_1=e+(dr^2)f$ and $h_2=e+df$ in the lattices
\[\Lambda_1:=U^{\oplus 3}\oplus E_8(-1)^{\oplus 2}\oplus\langle -2(n-1)\rangle\;\;\;\hbox{and}\;\;\; \Lambda_2:=U^{\oplus 3}\oplus E_8(-1)^{\oplus 2}\oplus\langle-2\left((n-1)r^2\right)\rangle\]
respectively. Let $h=e+df$ in the standard lattice $\Lambda$ (see equation \eqref{eq:sec_prelim:Lambda0}). Then 
\[\left(\Lambda_i\right)_{h_i}=U^{\oplus 2}\oplus E_8(-1)^{\oplus 2}\oplus Q_i(-1)\]
with 
\[Q_1=\langle x_1,y_1\rangle=\left(\begin{array}{cc}2(n-1)&0\\0&2dr^2\end{array}\right)\;\;\;\hbox{and}\;\;\;Q_2=\langle x_2,y_2\rangle=\left(\begin{array}{cc}2(n-2)r^2&0\\0&2d\end{array}\right).\]
As usual $\Lambda_h=U^{\oplus 2}\oplus E_8(-1)^{\oplus 2}\oplus Q_h(-1)$, with $Q_h$ 
as in \eqref{eq:sec_K3n+:Q} with basis $\{z_1,z_2\}$. Consider the lattice morphism $Q_1\to Q_h$ given by $x_1\mapsto z_1$ and $y_1\mapsto rz_2$, and the induced 
map $\psi_1\colon\Lambda_{1,h_1}\to \Lambda_h$.
Via $\psi_1$, the group $\widetilde{O}^+((\Lambda_1)_{h_1})$
is identified with a finite index subgroup of $\widetilde{O}^+(\Lambda_h)$, see \cite{GHS13}*{Lemma 7.1}.
The reflection $\sigma_{x_1}\in\widehat{O}^+(\Lambda_1,h_1)$, which generates the quotient
$\widehat{O}^+(\Lambda_1,h_1)/\widetilde{O}^+((\Lambda_1)_{h_1})$, is the restriction to $\psi_1(\Lambda_{1,h_1})$ of $\sigma_{z_1}\in\widehat{O}^+(\Lambda,h)$.
This shows that $\psi_1$ identifies $\widehat{O}(\Lambda_1,h_1)$ with a finite index subgroup of $\widehat{O}(\Lambda,h)$. 
Therefore, $\psi_1$ induces a finite morphism $\Omega((\Lambda_1)_{h_1})/\widehat{O}^+(\Lambda_1,h_1)\to \Omega(\Lambda_h)/\widehat{O}^+(\Lambda_h)$, and hence, a dominant rational map $\mathcal{M}^1_{\mathrm{K}3^{[n]},2dr^2}\dashrightarrow \mathcal{M}^1_{\mathrm{K}3^{[n]},2d}$.
The argument for $\mathcal{M}_{K3^{[(n-1)r^2+1]},2d}^1$
is analogous. In the case $\gamma=2$, by \eqref{eq:sec_K3n:non_emp} both moduli spaces $\mathcal{M}^2_{\mathrm{K}3^{[n]},2d}$ and $\mathcal{M}^2_{\mathrm{K}3^{[2]},2d}$ are non-empty if and only if $d\equiv -1 \pmod 4$, since $k$ is odd.
Denote by $\Lambda'$ and $Q'$ the corresponding lattices for $(n,d,\gamma,a)=(k^2+1,d,2,1)$ with basis $\{z_1', z_2'\}$ for $Q'$, and by $\Lambda$, $Q$ the lattices for $(n,d,\gamma,a)=(2,d,2,k)$, with $\{z_1,z_2\}$ the standard basis of $Q$. Then the lattice morphism inducing the finite map on period domains is given by $z_1'\mapsto kz_1$ and $z_2'\mapsto z_2$. One checks that this map
descends to the quotient.
\end{proof}

\begin{proof}[Proof of Theorem \ref{thm: thm5}]

The first statement follows from Lemma~\ref{lemma:sec_K3n+:fm1} as $\mathcal{M}_{\mathrm{K}3^{[2]},2d}^1$ is of general type for all $d\ge 12$, see \cite{GHS10}.
For statement (2a), we may assume $e=0$ by Lemma~\ref{lemma:sec_K3n+:fm1}.
The statement then follows from Theorem~\ref{thm:sec_K3n+:gamma1_1}.
Similarly, for statement (2b) we may assume $r=1$ by Lemma~\ref{lemma:sec_K3n+:fm1}.
Then for $p=1$, the theorem follows from strange duality, Proposition~\ref{prop:sec_K3n:SD},
since our assumptions on $k$ imply that $n-1\geq 12$.
For $p$ prime with $p\equiv 3\mod 4$, the statement follows from Theorem~\ref{thm:sec_K3n+:gamma1_1}.
\end{proof}

\section{Non-split fourfolds of \texorpdfstring{K$3^{[2]}$}{K3[2]} type and the \texorpdfstring{K$3^{[n]}$}{K3[n]} divisibility 2 case}
\label{S:gamma2}

In this section, we establish general type results in the two-dimensional case and then use Lemma \ref{lemma:sec_K3n+:fm1} to generalize  to  higher dimensions. Recall that the moduli space $\mathcal{M}_{\mathrm{K}3^{[2]},2d}^2$ is irreducible when non-empty, see \cite{Apo14}. Let $\Lambda$ be as in \eqref{eq:sec_prelim:Lambda} of K$3^{[2]}$-type. The discriminant group $D(\Lambda)$ is cyclic of order $2$, thus $\gamma\in\{1,2\}$ and when $\gamma=2$, the moduli space $\mathcal{M}_{\mathrm{K}3^{[2]},2d}^\gamma$ is non-empty if and only if $d\equiv -1\mod 4$. Note that ${\rm{Id}}=-{\rm{Id}}$ in $D(\Lambda)$, in particular
\[\widetilde{O}\left(\Lambda,h\right)=\widehat{O}\left(\Lambda,h\right).\]
We fix $h\in \Lambda$ of degree $2d=8t-2$ and divisibility ${\rm{div}}(h)=\gamma=2$. In light of Theorem \ref{thm:sec_prelim:Eichler} we may assume 
\[h=2(e+tf)-\ell,\]
where as usual $\{e,f\}$ are the standard generators of the first copy of $U$ and $\langle\ell\rangle=\langle -2
\rangle$ a generator of the last factor of $\Lambda$ in \eqref{eq:sec_prelim:Lambda}, see Lemma \ref{lemma:sec_K3n:h}. Moreover, by \cite{GHS10}*{Proposition 3.12}, the natural inclusion in Lemma \ref{lemma:sec_prelim:O-tildes} is an equality
\begin{equation}
\label{eq:sec_K32:Oh-t}
\widetilde{O}^+\left(\Lambda_h\right)=\widetilde{O}^+\left(\Lambda,h\right).
\end{equation}
Let $\Lambda_h=U^{\oplus 2}\oplus E_8(-1)^{\oplus 2}\oplus Q_t(-1),$ where 
\begin{equation}
\label{eq:sec_K32:Q}
Q_t=\left(\begin{array}{cc}2&-1\\ -1&2t\end{array}\right).
\end{equation}
As a consequence of \eqref{eq:sec_K32:Oh-t} we have a special case of Proposition \ref{prop:sec_K3n:R(Q)}:
\begin{proposition}
\label{prop:ssec_K32:Qt}
Let $d=4t-1>16$ and assume there exists a primitive embedding $Q_t\subset E_8$ such that the number of roots $\left|R(Q_t^\perp)\right|$ is at least $2$ and at most $14$ (resp. $16$), then the modular variety $\Omega\left(\Lambda_h\right)\big/\widetilde{O}^+\left(\Lambda_h\right)$ is of general type (resp. non-negative Kodaira dimension). In particular, if such embedding exists, the moduli space $\mathcal{M}_{\mathrm{K}3^{[2]},2d}^{\gamma}$ is of general type (resp. non-negative Kodaira dimension) for $d=4t-1$ and $\gamma=2$.
\end{proposition}

Now we can prove our theorem about
the moduli space $\mathcal{M}_{\mathrm{K}3^{[2]},2d}^2$:

\begin{proof}[Proof of Theorem \ref{thm:sec_intro:K32}]
The cases $t=10, 12$ are treated in \cite{GHS13}*{Proposition 9.2} so we assume $t\geq 13$. Let $x_1$ be the unique positive integer such that 
\begin{equation}
\label{eq:sec_K32:x1}
x_1^2+(x_1+1)^2+6<2t<(x_1+1)^2+(x_1+2)^2+6.
\end{equation}
Note that $x_1^2+(x_1+1)^2+6$ and $(x_1+1)^2+(x_1+2)^2+6$ are odd and the above inequalities are strict. Now consider the odd number
\[R=2t-(x_1^2+(x_1+1)^2).\]
From \eqref{eq:sec_K32:x1} we have
\begin{equation}
\label{eq:sec_K32:R upper bound} 6<R<(x_1+1)^2+(x_1+2)^2+6-(x_1^2+(x_1+1)^2)=4x_1+10.
\end{equation}
By Corollary \ref{cor:sec_prelim:sumsofsquares} we may write
 \[R=x_5^2+x_6^2+x_7^2+x_8^2,\]
 where $x_5\ge x_6\ge x_7\ge x_8$ are coprime integers with $x_8$ positive except in the cases $R=9,11, 17, 29, 41$ where we take $x_8=0$, and $x_7$ positive. Consider the embedding $Q_t\hookrightarrow E_8$ given by $z_1\mapsto v_1, z_2\mapsto v_2$, where $\{z_1,z_2\}$ is the basis of \eqref{eq:sec_K32:Q} and
 \begin{equation}
 \begin{aligned}
 \label{eq:sec_K32:k32 embedding}
    v_1&=e_1+e_2\\
    v_2& = x_1 e_1 -(x_1+1) e_2 +x_5 e_5 + x_6 e_6+x_7e_7+x_8e_8.
\end{aligned}
\end{equation}
Note that $(v_1,v_1)=2, (v_2,v_2)=2t,$ and $(v_1,v_2)=-1$. As in the proof of Proposition \ref{prop:sec_K3n:roots}, primitivity of the embedding follows from Lemma \ref{lemma:sec_K3n:odd-coeffs}. Note that integral roots of $E_8$ in the orthogonal complement $Q_t^\perp$ are:
\begin{enumerate}
\item $\pm e_3\pm e_4$,
\item \label{n=2 root count 2}$\pm e_i\pm e_8$ for $i=3,4$ if $x_8=0$,
 \item \label{n=2 root count 3} $\pm (e_j-e_k)$ for $j,k\in \{5,6,7,8\}$ if $x_j=x_k$.
\end{enumerate}
Observe that if $x_8\ne 0$, there are at most $6$ roots of type \eqref{n=2 root count 3} and so there are between $4$ and $10$ total integral roots in $Q^\perp$. If $x_8=0$, then there are at most $2$ roots of type  \eqref{n=2 root count 3} and so there are between $12$ and $14$ total integral roots.

We now we count fractional roots of $E_8$ in the orthogonal complement $Q_t^\perp$. Suppose that $w$ is such a fractional root. Since $(w,v_1)=(w,v_2)=0$, we must have a choice of signs such that
\begin{equation}
\label{eq:sec_K32: n=2 signs}
2x_1+1=\pm x_5\pm x_6\pm x_7\pm x_8.
\end{equation}
Since $2x_1+1, x_5, x_6, x_7, x_8$ are all non-negative, it follows that 
\begin{equation}
\label{eq:sec_K32:K32 bound1}
2x_1+1\le x_5+\ldots+x_8.
\end{equation}
The above inequality together with  \eqref{eq:sec_K32:R upper bound} imply
\[
\begin{aligned}
\sum_{i=5}^8(x_i-1)^2&=R-2\sum_{i=5}^8x_i +4\\
&\le 4x_1+9-2(2x_1+1) +4\\
&=11.
\end{aligned}
\]
Then, either $x_5=4$ in which case, $x_6, x_7\le 2$ and $x_8\le 1$ or $x_5, x_6\le 3$, $x_7, x_8\le 2$. Since $\sum_{i=5}^8 x_i$ is odd, we must have $\sum_{i=5}^8 x_i\le 9.$ and $x_1\le 4$. In particular, if $t\ge 34$, there are no fractional roots in $Q_t^\perp\subset E_8$ and 
$$4\le |R(Q_t^\perp)|\le 14.$$ 
It remains to count fractional roots when $13\le t\le 33$. This is done case by case and the full list of lattice embeddings together with the count of fractional roots for each embedding is in Appendix \ref{Appendix}. 
\end{proof}

Finally, Lemma \ref{lemma:sec_K3n+:fm1} together with Theorem \ref{thm:sec_intro:K32} gives us Theorem \ref{thm:gamma2}.

\begin{proof}[Proof of Theorem \ref{thm:gamma2}]
Write $n-1=k^2$ for some odd integer $k$. By Proposition \ref{prop:sec_K3n:nonemp}, we then have that $\mathcal{M}_{\mathrm{K}3^{[n]},2d}^{2}$ is non-empty precisely when $d\equiv -1 \mod 4$ (as in the $n=2$ case). Since  $\gamma=2$, we know $(k,\gamma)=1$. Hence by Lemma \ref{lemma:sec_K3n+:fm1}, we have that $\mathcal{M}_{\mathrm{K}3^{[n]},2d}^{2}$ dominates 
$\mathcal{M}_{\mathrm{K}3^{[2]},2d}^{2}$. The result then follows from Theorem \ref{thm:sec_intro:K32}.
\end{proof}

\section{Non-split hyperk\"{a}hler tenfolds of OG10-type.}

Let $X$ be an hyperk\"{a}hler variety of OG10-type. The Beauville--Bogomolov--Fujiki lattice $\left(H^2(X,\mathbb{Z}), q_X\right)$ is isometric to 
$$\Lambda=U^{\oplus 3}\oplus E_8(-1)^{\oplus 2}\oplus A_2(-1)$$
and after fixing a marking, by \cite{Ono22} one has
\[
{\rm{Mon}}^2(\Lambda)=\widehat{O}^+(\Lambda)=O^+(\Lambda).
\]
As a direct consequence, we obtain (see also \cite{Ono22b}*{Section 1}):
\begin{proposition}[\cites{Ono22, Ono22b, Son21}]
\label{prop:sec_OG10:irreducible}
The moduli space $\mathcal{M}_{\mathrm{OG10},2d}^\gamma$ is irreducible and non-empty when $\gamma=1$ or $\gamma=3$ and $2d\equiv 12\mod 18$.
\end{proposition}
Note that the discriminant group $D(\Lambda)$ is cyclic of order $3$, hence $\gamma\in \lbrace 1,3 \rbrace$. This leaves two possibilities for $\gamma$. The first one, known as the {\textit{split case}}, is when $\gamma=1$. In this case $h$ can be chosen (up to the action of $\widetilde{O}(\Lambda)$) as $h=e+df\in U$ and the orthogonal complement of $h$ in $\Lambda$ is of the form
\begin{equation}
\label{eq:sec_OG10:OG10_split}
\Lambda_h=U^{\oplus 2}\oplus E_8(-1)^{\oplus 2}\oplus A_2(-1)\oplus\langle-2d\rangle,
\end{equation}
where the last factor is generated by $e-df$. This case was treated in \cite{GHS11}. The second possibility is $\gamma=3$, known as the {\textit{non-split case}}. In this case, $\mathcal{M}_{\mathrm{OG10},2d}^\gamma$ is non-empty if and only if $2d\equiv 12\mod 18$, see \cite{GHS11}*{Lemma 3.4}. Equivalently $2d=18t-6$ and we can choose 
$$h=3(e+tf)+(2a_1+a_2),$$
where again $\{e,f\}$ are standard generators of $U$ and $\{a_1,a_2\}$ of $A_2(-1)$. In this case
\begin{equation}
\label{eq:sec_OG10:OG10_nonsplit}
\Lambda_h=U^{\oplus 2}\oplus E_8(-1)^{\oplus 2}\oplus Q_t(-1),
\end{equation}
where $Q_t(-1)$ is given by
\begin{equation}
\label{eq:sec_OG10:Q_t}
Q_t(-1)=\left(\begin{array}{ccc}
-2 & 1 & 0 \\
1 & -2 & -1 \\
0 & -1 & -2t
\end{array}\right)
\end{equation}
with generators $\{a_2,e+a_1,e-tf\}\subset U\oplus A_2(-1)$. Moreover, the discriminant group $D(\Lambda_h)$ is cyclic of order $2d/3$, see \cite{GHS11}*{Theorem 3.1}.
\begin{lemma}
\label{lemma:sec_OG10:irr_cusp}
In the above setting the restriction map $\widehat{O}(\Lambda,h)\to O(\Lambda_h)$ induces an isomorphism
\begin{equation}
\label{eq:sec_OG10:restriction}\widehat{O}^+\left(\Lambda,h\right)\cong \widetilde{O}^+\left(\Lambda_h\right)
\end{equation}
and the arithmetic group $\widetilde{O}^+(\Lambda_h)$ has no irregular cusps.
\end{lemma}

\begin{proof}
The isomorphism \eqref{eq:sec_OG10:restriction} is \cite{GHS11}*{Theorem 3.1}. Moreover, $D(\Lambda_h)$ is cyclic of order $2d/3$ and by \cite{Ma21}*{Proposition 1.1}, the 
group $\widetilde{O}^+\left(\Lambda_h\right)$ has irregular cusps only if $D(\Lambda_h)\cong\mathbb{Z}\big/8\mathbb{Z}$, but this contradicts the assumption $2d\equiv 12\mod 18$.
\end{proof}

The following lemma follows immediately from \cite{GHS13}*{Proposition 8.13} together with the proof of \cite{GHS07}*{Theorem 4.2}.
\begin{lemma}[\cites{GHS07, GHS13}]
\label{lemma:sec_OG10:ramif}
Assume the exists a primitive embedding $Q_t\hookrightarrow E_8$ such that $2\leq\left|R\left(Q_t^\perp\right)\right|\leq16$ and let $\Lambda_h\hookrightarrow\mathrm{II}_{2,26}$ be the induced lattice embedding. 
Then the quasi-pull back $F$ of Borcherds form $\Phi_{12}$ to $\Omega^\bullet\left(\Lambda_h\right)$ is a cusp form of weight $12+\left|R(Q_t^\perp)\right|/2$ and character $\mathrm{det}:\widetilde{O}^+(\Lambda_h)\longrightarrow \mathbb{C}^*$ that vanishes along the ramification divisor of the modular projection
\[\Omega\left(\Lambda_h\right)\longrightarrow\Omega\left(\Lambda_h\right)\big/\widetilde{O}^+\left(\Lambda_h\right).\]
\end{lemma}

\begin{proof}
Since $D(\Lambda_h)$ is cyclic, isotropic subgroups are cyclic as well and the proof of \cite{GHS07}*{Theorems 6.2 (ii)} applies in this case, see \cite{GHS07}*{Theorem 4.2}. Vanishing at the ramification divisor is \cite{GHS13}*{Proposition 8.13}.
\end{proof}

\begin{proposition}
\label{prop:sec_OG10:mod}
If there exists a primitive embedding $Q_t\hookrightarrow E_8$ such that the number of roots $\left|R(Q_t^\perp)\right|$ in the orthogonal complement $Q_t^{\perp}\subset E_8$ is at least $2$ and at most $16$, then the non-split moduli space $\mathcal{M}_{\mathrm{OG10},2d}^\gamma$ is of general type, where $\gamma=3$ and $2d=18t-6$. 
\end{proposition}

\begin{proof}
The Proposition follows from Theorems \ref{thm:sec_prelim:low-weight} and \ref{thm:sec_prelim:Torelli} together with Lemma \ref{lemma:sec_OG10:ramif} for the modular variety 
$\Omega\left(\Lambda_h\right)\big/\widehat{O}^+\left(\Lambda,h\right)\cong\Omega\left(\Lambda_h\right)\big/\widetilde{O}^+\left(\Lambda_h\right)$.
\end{proof}

Now we can prove our main result regarding the non-split moduli space $\mathcal{M}_{\mathrm{OG10},2d}^3$. The proof is very similar to Theorem \ref{thm:sec_intro:K32} so we omit details.

\begin{proof}[Proof of Theorem \ref{thm:sec_intro:OG10}]
The case of $t=4$ was proved in \cite{GHS11}*{Corollary 4.3}. For the moment, we will assume that $t\ge 7$ and then tackle the cases $t=5,6$ subsequently. Let $x_1$ be the unique non-negative even integer such that
\begin{equation}
\label{eq:sec_OG10:x_1}
2x_1^2+(x_1+1)^2+12<2t<2(x_1+2)^2+(x_1+3)^2+12.
\end{equation}
and consider the odd number $R=2t-(2x_1^2+(x_1+1)^2)$. Then,
\begin{equation}
\label{eq:sec_OG10:OG10 R bound}12<R<12x_1+18.
\end{equation}
Recall \cite{HK82}*{Satz 1b} that any positive integer $n\equiv 5 \mod 8$ is a sum of three pairwise distinct coprime squares. In particular we can assume $\star\not\equiv 5 \mod 8$. Let
\[\Theta=\begin{cases}
0 & \mbox{ if } R\not \equiv 7 \mod 8, R\ne 19,27,33,43,51,57,67,99,123,163,177,187,267,627,\star\\
1 &\mbox{ if } R\equiv 7 \mod 8\text{ or } R= 19, 27, 43, 51, 67, 99, 123, 163, 187, 267, 627\\
2 & \mbox{ if } R=33, 57, 177, \star
\end{cases}
\]
and consider the integer
\[S=2t-(2x_1^2+(x_1+1)^2)-2\Theta^2.\]
Observe that $S$ can be written as a sum of three distinct coprime squares
\[S=x_6^2+x_7^2+x_8^2,\]
with $x_6>x_7>x_8\ge 0$. Our embedding $Q_t\hookrightarrow E_8$ is given by $z_1\mapsto v_i$, where $\{z_1,z_2,z_3\}$ is a basis for $Q_t$, with matrix given by minus the matrix defined in \eqref{eq:sec_OG10:Q_t}, and 
 \begin{equation}
 \begin{aligned}\label{eq:sec_OG10: OG10 embedding}
    v_1&=e_2-e_1\\
    v_2&=e_3-e_2\\
    v_3&= x_1 e_1 +x_1e_2+(x_1+1) e_3+\Theta e_4+\Theta e_5 + x_6 e_6+x_7e_7+x_8e_8.
\end{aligned}
\end{equation}
As in the proof of Proposition \ref{prop:sec_K3n:roots} and Theorem \ref{thm:sec_intro:K32}, primitivity follows from parity considerations together with Lemma \ref{lemma:sec_K3n:odd-coeffs}.

One observes that integral roots in $Q_t^\perp$ are given by 
\begin{enumerate}
     \item $\pm (e_4-e_5)$
     \item $\pm(e_4+e_5)$ if $\Theta=0$
     \item $\pm e_4\pm e_8$, $\pm e_5\pm e_8$ if $\Theta=e_8=0$
     \item $\pm (e_4-e_i),$ $\pm (e_5-e_i)$ for $i\in \{6,7,8\}$ if $\Theta=x_i\ne 0$.
 \end{enumerate}

Hence if $\Theta\ne 0$, there are between $2$ and $6$ integral roots. If $\Theta=0$, there are either $4$ or $12$ integral roots (with $12$ occurring only if $x_8=0$).

Suppose that $w$ is a fractional root in $Q_t^\perp$. Since $w.v_i=0$ for $i=1,2,3$, we must have a choice of signs such that
\begin{equation}
\label{eq:sec_ OG10:signs}
3x_1+1=\pm \Theta \pm \Theta \pm x_6\pm x_7\pm x_8,
\end{equation}
where the number of $+$ signs on the right hand side is even. 
Since $3x_1+1, x_5, x_6, x_7, x_8$ are all non-negative, 
\begin{equation}
\label{eq:sec_OG10:OG10 bound1}3x_1+1\le 2\Theta+ \sum_{i=6}^8x_i.
\end{equation}
From \eqref{eq:sec_OG10:OG10 R bound} and \eqref{eq:sec_OG10:OG10 bound1}, we have
\[
\begin{aligned}
\sum_{i=6}^8(x_i-2)^2&=S-4\sum_{i=6}^8x_i+12\\
&\le 12x_1+17-2\Theta^2-4(3x_1+1-2\Theta)+12\\
&=25+2\Theta(4-\Theta),
\end{aligned}
\]
where $2\Theta(4-\Theta)\in \{0,6,8\}$.
In particular, when $\Theta=0$, $\sum_{i=6}^8(x_i-2)^2\leq 25$, and $\sum_{i=6}^8 x_i\leq 13$. And when $\Theta=1,2$, then $\sum_{i=6}^8 x_i\leq 15$. From \eqref{eq:sec_OG10:OG10 bound1} it follows that either $x_1<6$ and $\Theta<2$ or $x_1<8$ and $\Theta=2$. In other words, if $x_1\geq 6$ there cannot be fractional roots in $Q_t^\perp$. Moreover, when $x_1=6$, it follows from \eqref{eq:sec_OG10:OG10 R bound} that the only values of $R$ for which $\Theta=2$ are $33$ and $57$. In these cases, taking  respectively $(x_6,x_7,x_8)=(4,3,0), (6,3,2)$ the bound \eqref{eq:sec_OG10:OG10 bound1} is not satisfied, i.e., there cannot be fractional roots in $Q_t^\perp$ in these cases either. We have shown the desired result for all $t\ge 67$. The embeddings and root counts in the cases $4\le t\le 66$ are listed in Appendix \ref{Appendix}. This shows that for all $t\ge 4$ there is a primitive embedding $Q_t\hookrightarrow E_8$ such that $2\le |R(Q_t^\perp)|\le 16$ and Proposition \ref{prop:sec_OG10:mod} gives us the result.
\end{proof}

\appendix
\numberwithin{equation}{section}

\section{Lattice embeddings for low degree}
\label{Appendix}
\addtocontents{toc}{\protect\setcounter{tocdepth}{0}}

\renewcommand{\theequation}{\arabic{equation}}

Here we complete the proof of Theorems \ref{thm:sec_intro:K32} and \ref{thm:sec_intro:OG10} by listing all the lattice embeddings for low values of $t$. 

\subsection*{Non-split hyperk\"{a}hler varieties of K3-type and low degree.}

We list the corresponding embeddings for $13\leq t\leq 33$. The embedding is always the one given by \eqref{eq:sec_K32:k32 embedding}, except when $t=21$ (see below). For $t\in\{19,20,22,31,32,33\}$ the parameters are:

\begin{tabular}[h]{M M M M l}
t & x_1 & R & (x_5,x_6,x_7,x_8) & fractional roots\\ \midrule
19 & 3 & 13 & (2,2,2,1) & 4 \\
20 & 3 & 15 & (3,2,1,1) & 4\\
22 & 3 & 19 & (4,1,1,1) & 4\\
31 & 4 & 21 & (3,2,2,2) & 4\\
32 & 4 & 23 & (3,3,2,1) & 4\\
33 & 4 & 25 & (4,2,2,1) & 4
\end{tabular}

In these cases one shows that there are $4$ fractional roots in $Q_t^\perp$ given by $\pm\frac{1}{2}(e_1-e_2\pm (e_4-e_5)-(e_5+e_6+e_7+e_8)$. Since $x_8\ne 0$, 
there are between $4$ and $10$ integral roots, so 
$8\le |R(Q_t^\perp)|\le 14$. When $24\leq t\leq 30$, $x_1=4$ and $7\leq R\leq 19$. Arguing as in the case $t\geq 34$ one rules out the possibility of fractional roots in $Q_t^\perp$.
For the following list

\begin{tabular}[h]{M M M M l}
t & x_1 & R & (x_5,x_6,x_7,x_8) & fractional roots\\ \midrule
15 & 2 & 17 & (3,2,2,0) & 0\\
16 & 3 & 7 & (2,1,1,1) & 0\\
17 & 3 & 9 & (2,2,1,0) & 0\\
18 & 3 & 11 & (3,1,1,0) & 0\\
23 & 3 & 21 & (3,2,2,2) & 0
\end{tabular}

\noindent one checks that there are no fractional roots. For the embeddings

\begin{tabular}[h]{M M M M l}
t & x_1 & R & (x_5,x_6,x_7,x_8) & fractional roots\\ \midrule
13 & 2 & 13 & (2,2,2,1) & 4\\
14 & 2 & 15 & (3,2,1,1) & 8
\end{tabular}

\noindent
there are $10$ integral roots and $4$ fractional roots in $Q_t^\perp$, respectively 
$6$ integral roots and $8$ fractional roots given by 
\[\pm\frac{1}{2}\left(e_1-e_2\pm(e_3+e_4) -e_5-e_6-e_7+e_8)\right)\]
and
\[\pm\frac{1}{2}\left(e_1-e_2\pm(e_3+e_4) -e_5-e_6\pm(e_7-e_8)\right)\]
respectively. Finally, when $t=21$ the embedding given in \eqref{eq:sec_K32:k32 embedding} yields too many roots in $Q_t^\perp$. However 
a method described in \cite{GHS13}*{Section 9}
called the ``$(\frac{1}{2}++)$-algorithm'' works in this case. The embedding $Q_t\hookrightarrow E_8$ is given by $z_1\mapsto v_1$, $z_2\mapsto v_2$, with
\[
 \begin{aligned}
    v_1&=-e_1+e_2\\
    v_2&= \frac{1}{2}\left(e_1-e_2+11e_3+5e_4+3e_5+3e_6+e_7+e_8\right)
\end{aligned}
\]
One checks that the embedding is primitive and there are $6$ integral roots in $Q_t^\perp$ given by $\pm(e_1+e_2), \pm (e_5-e_6), \pm(e_7-e_8)$ and $8$ fractional roots given by \[\pm\frac{1}{2}\left(e_1+e_2\pm(e_3-e_4-e_5-e_6\pm(e_7-e_8))\right).\] Hence there are a total of roots $14$ roots in $Q_t^\perp$. We have treated all cases $t\geq 13$. Proposition \ref{prop:ssec_K32:Qt} together with \cite{GHS13}*{Proposition 9.2} for $t=10, 12$ gives the theorem.

\subsection*{Non-split hyperk\"{a}hler varieties of OG10-type and low degree.}

Here we complete the proof of the Theorem \ref{thm:sec_intro:OG10} by listing the remaining data $t, R$, and $(x_6,x_7,x_8)$ defining the embedding $Q_t^\perp\hookrightarrow E_8$ given by \eqref{eq:sec_OG10: OG10 embedding} where the number of roots satisfies $2\leq \left|R(Q_t^\perp)\right|\leq 16$. Recall that the number of integral roots for $t\geq 7$ is between $2$ and $6$ when $\Theta\neq 0$ and between $4$ and $12$ when $\Theta=0$, with $12$ only occurring when $x_8=0$. For $54\leq t\leq 66$ the parameters are:

\vspace{-0.2cm}
\begin{multicols}{2}
\begin{tabular}[h]{M M M M M p{1.7cm}}
      t & x_1 & R & \Theta & (x_6, x_7, x_8) & fractional roots \\ \midrule
      54 & 4 & 51 & 0 & (6,3,2) & 0 \\
     55 & 4 & 53 & 0 & (7,2,0) & 0 \\
    56 & 4 & 55 & 1 & (7,2,0) & 0 \\
     57 & 4 & 57 & 2 & (6,3,2) & 0 \\
     58 & 4 & 59 & 0 & (7,3,1) & 0 \\
     59 & 4 & 61 & 0 & (6,5,0) & 0 \\
     60 & 4 & 63 & 1 & (6,5,0) & 0
\end{tabular}

\begin{tabular}[h]{M M M M M p{1.7cm}}
     t & x_1 & R & \Theta & (x_6, x_7, x_8) & fractional roots \\ \midrule
     61 & 4 & 65 & 0 & (7,4,0) & 0 \\
     62 & 4 & 67 & 1 & (7,4,0) & 0 \\
     63 & 4 & 69 & 0 & (7,4,2) & 4 \\
     64 & 4 & 71 & 1 & (7,4,2) & 4 \\
     65 & 4 & 73 & 0 & (8,3,0) & 0 \\
     66 & 4 & 75 & 0 & (7,5,1) & 4
\end{tabular}

\end{multicols}
\vspace{-0.2cm}

The fractional roots in the cases $t=63$, $t=64$ and $t=66$ are
\[\pm \frac{1}{2}\left(e_1+e_2+e_3\pm (e_4-e_5)-e_6-e_7-e_8\right).\]
When $35\leq t\leq 53$, and $t\neq 45, 50, 52$, then $13\leq R\leq 49$ and (as in the high degree case) if one assumes there is a fractional root, equations \eqref{eq:sec_OG10:OG10 R bound} and \eqref{eq:sec_OG10:OG10 bound1} lead to a contradiction. In these cases, there are no fractional roots. For $t=45, 50, 52$ the parameters are:

\begin{tabular}[h]{M M M M M l}
t & x_1 & R & \Theta & (x_6,x_7,x_8) & fractional roots\\ \midrule
45 & 4 & 33 & 2 & (4,3,0) & 0 \\
50 & 4 & 43 & 1 & (6,2,1) & 0 \\
52 & 4 & 47 & 1 & (5,4,2) & 0
\end{tabular}

For the range $15\leq t\leq 34$, and $t\neq 20, 21, 27, 28, 34$ the embedding parameters are:

\vspace{-0.2cm}
\begin{multicols}{2}

\begin{tabular}[h]{M M M M M p{1.7cm}}
t & x_1 & R & \Theta & (x_6, x_7, x_8) & fractional roots \\ \midrule
15 & 2 & 13 & 0 & (3,2,0) & 0\\
 16 & 2 & 15 & 1 & (3,2,0) & 2\\
 17 & 2 & 17 & 0 & (4,1,0) & 0 \\
 18 & 2 & 19 & 1 & (4,1,0) & 2 \\
 19 & 2 & 21 & 0 & (4,2,1) & 4 \\
 22 & 2 & 27 & 1 & (4,3,0) & 4 \\
 23 & 2 & 29 & 0 & (4,3,2) & 0 \\
 24 & 2 & 31 & 1 & (5,2,0) & 4 
\end{tabular}

\begin{tabular}[h]{M M M M M p{1.7cm}}
t & x_1 & R & \Theta & (x_6, x_7, x_8) & fractional roots \\ \midrule
25 & 2 & 33 & 2 & (4,3,0) & 4 \\
26 & 2 & 35 & 0 & (5,3,1) & 4 \\
29 & 2 & 41 & 0 & (5,4,0) & 0 \\
30 & 2 & 43 & 1 & (5,4,0) & 2 \\
31 & 2 & 45 & 0 & (5,4,2) & 4 \\
32 & 2 & 47 & 1 & (5,4,2) & 4 \\
33 & 2 & 49 & 0 & (6,3,2) & 4 
\end{tabular}

\end{multicols}
\vspace{-0.2cm}

And the cases $t=20, 21, 27, 34$ we use the 
embedding given by \eqref{eq:sec_OG10: OG10 embedding} with parameters:

\begin{tabular}[h]{M M M M l l}
t & x_1 & \Theta & (x_6,x_7,x_8) & integral roots & fractional roots\\ \midrule
20 & 0 & 3 & (2,1,0) & 2 & 2\\
21 & 0 & 0 & (6,2,1) & 4 & 4\\
27 & 2 & 4 & (2,1,0) & 2 & 2\\
34 & 2 & 5 & (1,0,0) & 6 & 0
\end{tabular}

When $t=28$ one considers the embedding given by $z_i\mapsto v_i$, where $v_1=e_2-e_1$, $v_2=e_3-e_2$ and 
\[v_3=2e_1+2e_2+3e_3+5e_4+3e_5+2e_6+e_7.\]
One can check that the embedding is primitive with no integral roots and $2$ fractional roots in the orthogonal complement. The embeddings in the cases $8\leq t\leq 14$ with $t\neq 12, 13$ are as in \eqref{eq:sec_OG10: OG10 embedding} with the following parameters:

\begin{tabular}[h]{M M M M M l}
t & x_1 & R & \Theta & (x_6,x_7,x_8) & fractional roots\\ \midrule
8 & 0 & 15 & 1 & (3,2,0) & 12\\
9 & 0 & 17 & 0 & (4,1,0) & 0\\
10 & 0 & 19 & 1 & (4,1,0) & 2\\
11 & 0 & 21 & 0 & (4,2,1) & 4\\
14 & 0 & 27 & 1 & (4,3,0) & 12
\end{tabular}

Finally, the cases $t=5, 6, 7, 12, 13$  are given by $z_i\mapsto v_i$, where $v_1=e_2-e_1$, $v_2=e_3-e_2$ and 
\[v_3=a_1e_1+\ldots+a_8e_8\] 
with $(a_1,\ldots,a_8)$ as follows:

\begin{tabular}[h]{M M l l}
t & (a_1,\ldots,a_8) & integral roots & fractional roots\\ \midrule
5 & (1,1,2,0,1,1,1,1) & 12 & 0\\
6 & (1,1,2,0,0,2,1,1) & 6 & 0\\
7 & (1,1,2,2,1,1,1,1) & 12 & 0\\
12 & (1,1,2,0,0,4,1,0) & 12 & 0\\
13 & (1,1,2,1,1,4,1,0) & 6 & 2
\end{tabular}

With these we have covered all cases $5\leq t\leq 66$ finishing the proof of Theorem \ref{thm:sec_intro:OG10}.

\bibliography{Kod_HK_bib}

\begin{bibdiv}
\begin{biblist}

\bib{Apo14}{article}{
      author={Apostolov, A.},
       title={Moduli spaces of polarized irreducible symplectic manifolds are
  not necessarily connected},
        date={2014},
     journal={Ann. Inst. Fourier (Grenoble)},
      volume={64},
      number={1},
       pages={189\ndash 202},
         url={https://aif.centre-mersenne.org/articles/10.5802/aif.2844/},
}

\bib{Apo14+}{thesis}{
      author={Apostolov, A.},
       title={On irreducible symplectic varieties of {K}$3^{[n]}$-type},
        type={{PhD} thesis},
        date={2014},
}

\bib{BB66}{article}{
      author={Baily, W.~L., Jr.},
      author={Borel, A.},
       title={Compactification of arithmetic quotients of bounded symmetric
  domains},
        date={1966},
        ISSN={0003-486X},
     journal={Ann. of Math. (2)},
      volume={84},
       pages={442\ndash 528},
         url={https://doi.org/10.2307/1970457},
}

\bib{BD85}{article}{
      author={Beauville, A.},
      author={Donagi, R.},
       title={La vari\'{e}t\'{e} des droites d'une hypersurface cubique de
  dimension {$4$}},
        date={1985},
     journal={C. R. Math. Acad. Sci. Paris},
      volume={301},
      number={14},
       pages={703\ndash 706},
}

\bib{BFRT89}{article}{
      author={Borosh, I.},
      author={Flahive, M.},
      author={Rubin, D.},
      author={Treybig, B.},
       title={A sharp bound for solutions of linear {D}iophantine equations},
        date={1989},
        ISSN={0002-9939},
     journal={Proc. Amer. Math. Soc.},
      volume={105},
      number={4},
       pages={844\ndash 846},
         url={https://doi.org/10.2307/2047041},
}

\bib{BKPS98}{article}{
      author={Borcherds, R.~E.},
      author={Katzarkov, L.},
      author={Pantev, T.},
      author={Shepherd-Barron, N.~I.},
       title={Families of {$K3$} surfaces},
        date={1998},
        ISSN={1056-3911},
     journal={J. Algebraic Geom.},
      volume={7},
      number={1},
       pages={183\ndash 193},
}

\bib{BLM+21}{article}{
      author={Bayer, A.},
      author={Lahoz, M.},
      author={Macr\`\i, E.},
      author={Nuer, H.},
      author={Perry, A.},
      author={Stellari, P.},
       title={Stability conditions in families},
        date={2021},
        ISSN={0073-8301},
     journal={Publ. Math. Inst. Hautes \'{E}tudes Sci.},
      volume={133},
       pages={157\ndash 325},
         url={https://doi.org/10.1007/s10240-021-00124-6},
}

\bib{Bor95}{article}{
      author={Borcherds, R.~E.},
       title={Automorphic forms on {${\rm O}_{s+2,2}({\bf R})$} and infinite
  products},
        date={1995},
        ISSN={0020-9910},
     journal={Invent. Math.},
      volume={120},
      number={1},
       pages={161\ndash 213},
         url={https://doi.org/10.1007/BF01241126},
}

\bib{CS88}{article}{
      author={Conway, J.~H.},
      author={Sloane, N. J.~A.},
       title={Low-dimensional lattices. {I}. {Q}uadratic forms of small
  determinant},
        date={1988},
        ISSN={0962-8444},
     journal={Proc. Roy. Soc. London Ser. A},
      volume={418},
      number={1854},
       pages={17\ndash 41},
}

\bib{Daw24_1}{misc}{
      author={Dawes, M.},
       title={{\YMD{20240505}}{O}n the {K}odaira dimension of the moduli of
  deformation generalised {K}ummer varieties},
        date={2024},
        note={Preprint, arXiv:1710.01672. To appear in Algebr. Geom.},
}

\bib{Daw24_2}{misc}{
      author={Dawes, M.},
       title={{\YMD{20240524}}{G}eneral type results for moduli of deformation
  generalised {K}ummer varieties},
        date={2024},
        note={Preprint, arXiv:2405.15675.},
}

\bib{Deb18}{article}{
      author={Debarre, O.},
       title={Hyper-{K}\"ahler manifolds},
        date={2022},
        ISSN={1424-9286,1424-9294},
     journal={Milan J. Math.},
      volume={90},
      number={2},
       pages={305\ndash 387},
         url={https://doi.org/10.1007/s00032-022-00366-x},
}

\bib{DV10}{article}{
      author={Debarre, O.},
      author={Voisin, C.},
       title={Hyper-k\"{a}hler fourfolds and grassmann geometry},
        date={2010},
     journal={J. Reine Angew. Math.},
      volume={649},
       pages={63\ndash 87},
}

\bib{Eic74}{book}{
      author={Eichler, M.},
       title={Quadratische {F}ormen und orthogonale {G}ruppen},
      series={Grundlehren Math. Wiss.},
   publisher={Springer-Verlag, Berlin-New York},
        date={1974},
      volume={63},
}

\bib{FM21}{article}{
      author={Fortuna, M.},
      author={Mezzedimi, G.},
       title={The {K}odaira dimension of some moduli spaces of elliptic {K}3
  surfaces},
        date={2021},
        ISSN={0024-6107},
     journal={J. Lond. Math. Soc. (2)},
      volume={104},
      number={1},
       pages={269\ndash 294},
         url={https://doi.org/10.1112/jlms.12430},
}

\bib{Fre83}{book}{
      author={Freitag, E.},
       title={Siegelsche {M}odulfunktionen},
      series={Grundlehren Math. Wiss.},
   publisher={Springer-Verlag, Berlin},
        date={1983},
      volume={254},
        ISBN={3-540-11661-3},
         url={https://doi.org/10.1007/978-3-642-68649-8},
}

\bib{FV18}{article}{
      author={Farkas, G.},
      author={Verra, A.},
       title={The universal {$K3$} surface of genus {$14$} via cubic
  fourfolds},
        date={2018},
        ISSN={0021-7824},
     journal={J. Math. Pures Appl. (9)},
      volume={111},
       pages={1\ndash 20},
         url={https://doi.org/10.1016/j.matpur.2017.07.014},
}

\bib{FV21}{article}{
      author={Farkas, G.},
      author={Verra, A.},
       title={The unirationality of the moduli space of {$K3$} surfaces of
  genus {$22$}},
        date={2021},
     journal={Math. Ann.},
      volume={380},
      number={3-4},
       pages={953\ndash 973},
}

\bib{GHS07}{article}{
      author={Gritsenko, V.~A.},
      author={Hulek, K.},
      author={Sankaran, G.~K.},
       title={The {K}odaira dimension of the moduli of {$K3$} surfaces},
        date={2007},
        ISSN={0020-9910},
     journal={Invent. Math.},
      volume={169},
      number={3},
       pages={519\ndash 567},
         url={https://doi.org/10.1007/s00222-007-0054-1},
}

\bib{GHS10}{article}{
      author={Gritsenko, V.},
      author={Hulek, K.},
      author={Sankaran, G.~K.},
       title={Moduli spaces of irreducible symplectic manifolds},
        date={2010},
        ISSN={0010-437X},
     journal={Compos. Math.},
      volume={146},
      number={2},
       pages={404\ndash 434},
         url={https://doi.org/10.1112/S0010437X0900445X},
}

\bib{GHS11}{article}{
      author={Gritsenko, V.},
      author={Hulek, K.},
      author={Sankaran, G.~K.},
       title={Moduli spaces of polarized symplectic {O}'{G}rady varieties and
  {B}orcherds products},
        date={2011},
        ISSN={0022-040X},
     journal={J. Differential Geom.},
      volume={88},
      number={1},
       pages={61\ndash 85},
         url={http://projecteuclid.org/euclid.jdg/1317758869},
}

\bib{GHS13}{incollection}{
      author={Gritsenko, V.},
      author={Hulek, K.},
      author={Sankaran, G.~K.},
       title={Moduli of {K}3 surfaces and irreducible symplectic manifolds},
        date={2013},
   booktitle={in: Handbook of moduli. {V}ol. {I}},
      series={Adv. Lect. Math.},
      volume={24},
   publisher={Int. Press, Somerville, MA},
       pages={459\ndash 526},
}

\bib{HK82}{article}{
      author={Halter-Koch, F.},
       title={Darstellung nat{\"u}rlicher {Z}ahlen als {S}umme von
  {Q}uadraten},
        date={1982},
     journal={Acta Arith.},
      volume={42},
      number={1},
       pages={11\ndash 20},
}

\bib{Huy16}{book}{
      author={Huybrechts, D.},
       title={Lectures on {K}3 surfaces},
      series={Cambridge Stud. Adv. Math.},
   publisher={Cambridge University Press, Cambridge},
        date={2016},
      volume={158},
        ISBN={978-1-107-15304-2},
         url={https://doi.org/10.1017/CBO9781316594193},
}

\bib{Huy99}{article}{
      author={Huybrechts, D.},
       title={Compact hyperk\"ahler manifolds: basic results},
        date={1999},
     journal={Invent. Math.},
      volume={135},
       pages={63\ndash 113},
}

\bib{IKKR19}{article}{
      author={Iliev, A.},
      author={Kapustka, G.},
      author={Kapustka, M.},
      author={Ranestad, K.},
       title={E{PW} cubes},
        date={2019},
        ISSN={0075-4102},
     journal={J. Reine Angew. Math.},
      volume={748},
       pages={241\ndash 268},
         url={https://doi.org/10.1515/crelle-2016-0044},
}

\bib{IR01}{article}{
      author={Iliev, A.},
      author={Ranestad, K.},
       title={{$K3$} surfaces of genus {$8$} and varieties of sums of powers of
  cubic fourfolds},
        date={2001},
     journal={Trans. Amer. Math. Soc.},
      volume={353},
       pages={1455\ndash 1468},
}

\bib{IR07}{article}{
      author={Iliev, A.},
      author={Ranestad, K.},
       title={Addendum to {$K3$} surfaces of genus {$8$} and varieties of sums
  of powers of cubic fourfolds},
        date={2007},
     journal={C. R. Acad. Bulgare Sci.},
      volume={60},
       pages={1265\ndash 1270},
}

\bib{Iwa97}{book}{
      author={Iwaniec, H.},
       title={Topics in classical automorphic forms},
      series={Graduate Studies in Mathematics},
   publisher={Amer. Math. Soc., Providence, RI},
        date={1997},
      volume={17},
        ISBN={0-8218-0777-3},
         url={https://doi.org/10.1090/gsm/017},
}

\bib{kon93}{article}{
      author={Kond{\=o}, Shigeyuki},
       title={On the {K}odaira dimension of the moduli space of {K}3 surfaces},
        date={1993},
     journal={Compos. Math.},
      volume={89},
      number={3},
       pages={251\ndash 299},
}

\bib{kon99}{article}{
      author={Kond{\=o}, Shigeyuki},
       title={On the {K}odaira dimension of the moduli space of {K}3 surfaces
  {II}},
        date={1999},
     journal={Compos. Math.},
      volume={116},
      number={2},
       pages={111\ndash 117},
}

\bib{LLSvS17}{article}{
      author={Lehn, C.},
      author={Lehn, M.},
      author={Sorger, C.},
      author={van Straten, D.},
       title={Twisted cubics on cubic fourfolds},
        date={2017},
     journal={J. Reine Angew. Math.},
      volume={731},
       pages={87\ndash 128},
}

\bib{Ma18}{article}{
      author={Ma, S.},
       title={On the {K}odaira dimension of orthogonal modular varieties},
        date={2018},
        ISSN={0020-9910},
     journal={Invent. Math.},
      volume={212},
      number={3},
       pages={859\ndash 911},
         url={https://doi.org/10.1007/s00222-017-0781-x},
}

\bib{Ma21}{article}{
      author={Ma, S.},
       title={Irregular cusps of orthogonal modular varieties},
        date={2024},
     journal={Canad. J. Math.},
       pages={1\ndash 32},
}

\bib{Mar11}{incollection}{
      author={Markman, E.},
       title={A survey of {T}orelli and monodromy results for
  holomorphic-symplectic varieties},
        date={2011},
   booktitle={in: Complex and differential geometry},
      series={Springer Proc. Math.},
      volume={8},
   publisher={Springer, Heidelberg},
       pages={257\ndash 322},
         url={https://doi.org/10.1007/978-3-642-20300-8_15},
}

\bib{Mon13}{thesis}{
      author={Mongardi, G.},
       title={Automorphisms of hyperk\"{a}hler manifolds},
        type={{PhD} thesis},
        date={2013},
}

\bib{Muk06}{incollection}{
      author={Mukai, S.},
       title={Polarized {$K3$} surfaces of genus thirteen},
        date={2006},
   booktitle={in: Moduli spaces and arithmetic geometry},
      series={Adv. Stud. Pure Math.},
      volume={45},
   publisher={Math. Soc. Japan, Tokyo},
       pages={315\ndash 326},
         url={https://doi.org/10.2969/aspm/04510315},
}

\bib{Muk10}{article}{
      author={Mukai, S.},
       title={Curves and symmetric spaces, {II}},
        date={2010},
        ISSN={0003-486X},
     journal={Ann. of Math. (2)},
      volume={172},
      number={3},
       pages={1539\ndash 1558},
         url={https://doi.org/10.4007/annals.2010.172.1539},
}

\bib{Muk16}{incollection}{
      author={Mukai, S.},
       title={K3 surfaces of genus sixteen},
        date={2016},
   booktitle={in: Minimal models and extremal rays ({K}yoto, 2011)},
      series={Adv. Stud. Pure Math.},
      volume={70},
   publisher={Math. Soc. Japan, [Tokyo]},
       pages={379\ndash 396},
         url={https://doi.org/10.2969/aspm/07010379},
}

\bib{Muk88}{incollection}{
      author={Mukai, S.},
       title={Curves, {$K3$} surfaces and {F}ano {$3$}-folds of genus {$\leq
  10$}},
        date={1988},
   booktitle={in: Algebraic geometry and commutative algebra, {V}ol. {I}},
   publisher={Kinokuniya, Tokyo},
       pages={357\ndash 377},
}

\bib{Muk92}{incollection}{
      author={Mukai, S.},
       title={Polarized {$K3$} surfaces of genus {$18$} and {$20$}},
        date={1992},
   booktitle={in: Complex projective geometry ({T}rieste, 1989/{B}ergen,
  1989)},
      series={London Math. Soc. Lecture Note Ser.},
      volume={179},
   publisher={Cambridge Univ. Press, Cambridge},
       pages={264\ndash 276},
         url={https://doi.org/10.1017/CBO9780511662652.019},
}

\bib{Mum77}{article}{
      author={Mumford, D.},
       title={Hirzebruch's proportionality theorem in the noncompact case},
        date={1977},
        ISSN={0020-9910},
     journal={Invent. Math.},
      volume={42},
       pages={239\ndash 272},
         url={https://doi.org/10.1007/BF01389790},
}

\bib{Nik80}{article}{
      author={Nikulin, V.~V.},
       title={{Integral symmetric bilinear forms and some of their
  applications}},
    language={English},
        date={1980},
        ISSN={0025-5726},
     journal={{Math. USSR, Izv.}},
      volume={14},
       pages={103\ndash 167},
}

\bib{Nue17}{article}{
      author={Nuer, H.},
       title={Unirationality of moduli spaces of special cubic fourfolds and
  {K}3 surfaces},
        date={2017},
        ISSN={2313-1691},
     journal={Algebr. Geom.},
      volume={4},
      number={3},
       pages={281\ndash 289},
         url={https://doi.org/10.14231/AG-2017-015},
}

\bib{OG06}{article}{
      author={O'Grady, K.},
       title={Irreducible symplectic 4-folds and {Eisenbud--Popescu--Walter}
  sextics},
        date={2006},
     journal={Duke Math. J.},
      volume={134},
       pages={99\ndash 137},
}

\bib{OG89}{article}{
      author={O'Grady, K.},
       title={On the {K}odaira dimension of moduli spaces of abelian surfaces},
        date={1989},
        ISSN={0010-437X},
     journal={Compos. Math.},
      volume={72},
      number={2},
       pages={121\ndash 163},
         url={http://www.numdam.org/item?id=CM_1989__72_2_121_0},
}

\bib{Ono22b}{article}{
      author={Onorati, C.},
       title={Connected components of moduli spaces of irreducible holomorphic
  symplectic manifolds of {K}ummer type},
        date={2022},
        ISSN={1120-7183},
     journal={Rend. Mat. Appl. (7)},
      volume={43},
      number={4},
       pages={251\ndash 266},
}

\bib{Ono22}{article}{
      author={Onorati, C.},
       title={On the monodromy group of desingularised moduli spaces of sheaves
  on {K}3 surfaces},
        date={2022},
        ISSN={1056-3911},
     journal={J. Algebraic Geom.},
      volume={31},
       pages={425\ndash 465},
         url={https://doi.org/10.1090/jag/802},
}

\bib{PPZ19}{article}{
      author={Perry, A.},
      author={Pertusi, L.},
      author={Zhao, X.},
       title={Stability conditions and moduli spaces for {K}uznetsov components
  of {G}ushel--{M}ukai varieties},
        date={2023},
     journal={Geom. Topol.},
      volume={26},
      number={7},
       pages={3055\ndash 3121},
}

\bib{Sca87}{article}{
      author={Scattone, F.},
       title={On the compactification of moduli spaces for algebraic {$K3$}
  surfaces},
        date={1987},
        ISSN={0065-9266},
     journal={Mem. Amer. Math. Soc.},
      volume={70},
      number={374},
       pages={x+86},
         url={https://doi.org/10.1090/memo/0374},
}

\bib{Son21}{article}{
      author={Song, J.},
       title={On the image of the period map for polarized hyperk\"ahler
  manifolds},
        date={2023},
        ISSN={1073-7928},
     journal={Int. Math. Res. Not. IMRN},
      number={13},
       pages={11404\ndash 11431},
         url={https://doi.org/10.1093/imrn/rnac155},
}

\bib{TVA19}{article}{
      author={Tanimoto, S.},
      author={V\'{a}rilly-Alvarado, A.},
       title={Kodaira dimension of moduli of special cubic fourfolds},
        date={2019},
        ISSN={0075-4102},
     journal={J. Reine Angew. Math.},
      volume={752},
       pages={265\ndash 300},
         url={https://doi.org/10.1515/crelle-2016-0053},
}

\bib{Ver13}{article}{
      author={Verbitsky, M.},
       title={Mapping class group and a global {T}orelli theorem for
  hyperk\"{a}hler manifolds},
        date={2013},
     journal={Duke Math. J.},
      volume={162},
      number={15},
       pages={2929\ndash 2986},
         url={https://doi.org/10.1215/00127094-2382680},
}

\bib{Vie95}{book}{
      author={Viehweg, E.},
       title={Quasi-projective moduli for polarized manifolds},
      series={Ergeb. Math. Grenzgeb. (3)},
   publisher={Springer-Verlag, Berlin},
        date={1995},
      volume={30},
        ISBN={3-540-59255-5},
         url={https://doi.org/10.1007/978-3-642-79745-3},
}

\end{biblist}
\end{bibdiv}
\bibliographystyle{alpha}
\end{document}